\documentclass[a4paper,10pt]{article}
\usepackage{amsmath,amssymb,amsthm}
\usepackage{graphicx}
\usepackage{color}
\usepackage{algorithm}
\usepackage{algpseudocode}
\usepackage{textcomp}
\usepackage{float}
\usepackage{hyperref}
\usepackage[caption = false]{subfig}
\usepackage{accents}
\usepackage{cleveref}
\usepackage{tikz}

\newtheorem{theorem}{Theorem}[section]
\newtheorem{corollary}{Corollary}[section]
\newtheorem{lemma}{Lemma}[section]
\newtheorem{proposition}{Proposition}[section]

\theoremstyle{definition}
\newtheorem{definition}{Definition}
\newtheorem{remark}{Remark}

\newcommand{\R}{\mathbf{R}}
\newcommand{\N}{\mathbf{N}}

\newcommand{\Disk}{\mathbf{S}}

\newcommand{\weakto}{\rightharpoonup}
\newcommand{\weakstarto}{\stackrel*\rightharpoonup}
\newcommand{\hd}{\mathcal{H}}
\newcommand{\hdone}{\mathcal{H}^1}
\newcommand{\restr}{{\mbox{\LARGE$\llcorner$}}}
\renewcommand{\restriction}{|}
\newcommand{\pushforward}[2]{{{#1}_{\#}#2}}
\DeclareMathOperator{\supp}{supp}

\DeclareMathOperator{\alphargmin}{argmin}
\DeclareMathOperator{\dive}{div}

\newcommand\F{\mathcal{G}}
\newcommand\f{{G}}
\newcommand\E{\mathcal{E}}
\renewcommand\L{\mathcal{L}}
\renewcommand{\d}{{\mathrm d}}
\newcommand{\dist}{{\mathrm{dist}}}
\newcommand{\flux}{{\mathcal{F}}}
\newcommand{\prob}{{\mathcal{P}}}
\newcommand{\meas}{{\mathcal{M}}}
\newcommand{\cont}{{\mathcal{C}}}
\newcommand{\Lip}{{\mathrm{Lip}}}

\newlength{\dhatheight}

\title{Phase field approximations of branched transportation problems}
\author{Luca A. D. Ferrari, Carolin Rossmanith, Benedikt Wirth}
\date{}

\begin{document}
\maketitle

\begin{abstract}
In branched transportation problems mass has to be transported from a given initial distribution to a given final distribution,
where the cost of the transport is proportional to the transport distance, but subadditive in the transported mass.
As a consequence, mass transport is cheaper the more mass is transported together,
which leads to the emergence of hierarchically branching transport networks.
We here consider transport costs that are piecewise affine in the transported mass with $N$ affine segments,
in which case the resulting network can be interpreted as a street network composed of $N$ different types of streets.
In two spatial dimensions we propose a phase field approximation of this street network using $N$ phase fields
and a function approximating the mass flux through the network. 
We prove the corresponding $\Gamma$-convergence and show some numerical simulation results.
\end{abstract}

\section{Introduction}
Branched transportation problems constitute a special class of optimal transport problems that have recently attracted lots of interest (see for instance \cite[\S4.4.2]{Sa15} and the references therein).
Given two probability measures $\mu_+$ and $\mu_-$ on some domain $\Omega\subset\R^n$, representing an initial and a final mass distribution, respectively,
one seeks the most cost-efficient way to transport the mass from the initial to the final distribution.
Unlike in classical optimal transport, the cost of a transportation scheme does not only depend on initial and final position of each mass particle,
but also takes into account how many particles travel together.
In fact, the transportation cost per transport distance is typically not proportional to the transported mass $m$,
but rather a concave, nondecreasing function $\tau:[0,\infty)\to[0,\infty)$, $m\mapsto\tau(m)$.
Therefore it is cheaper to transport mass in bulk rather than moving each mass particle separately.
This automatically results in transportation paths that exhibit an interesting, hierarchically ramified structure (see Figure~\ref{fig:teaser}).
Well-known models (with parameters $\alpha<1$, $a,b>0$) are obtained by the choices
\begin{align*}
\tau(m)&=m^\alpha&\text{(``branched transport'' \cite{BLTJ:BLTJ4250}),}\\
\tau(m)&=\min(m,am+b)&\text{(``urban planning'' \cite{MR2176114,MR2535060}),}\\
\tau(m)&=\begin{cases}0&\text{if }m=0,\\1&\text{ else}\end{cases}&\text{(``Steiner tree problem'' \cite{Ambr_Tilli,Gilb_Poll,Paol_Step}).}
\end{align*}

\begin{figure}
\center
\includegraphics[width = 0.3\textwidth,trim=25 15 35 15,clip]{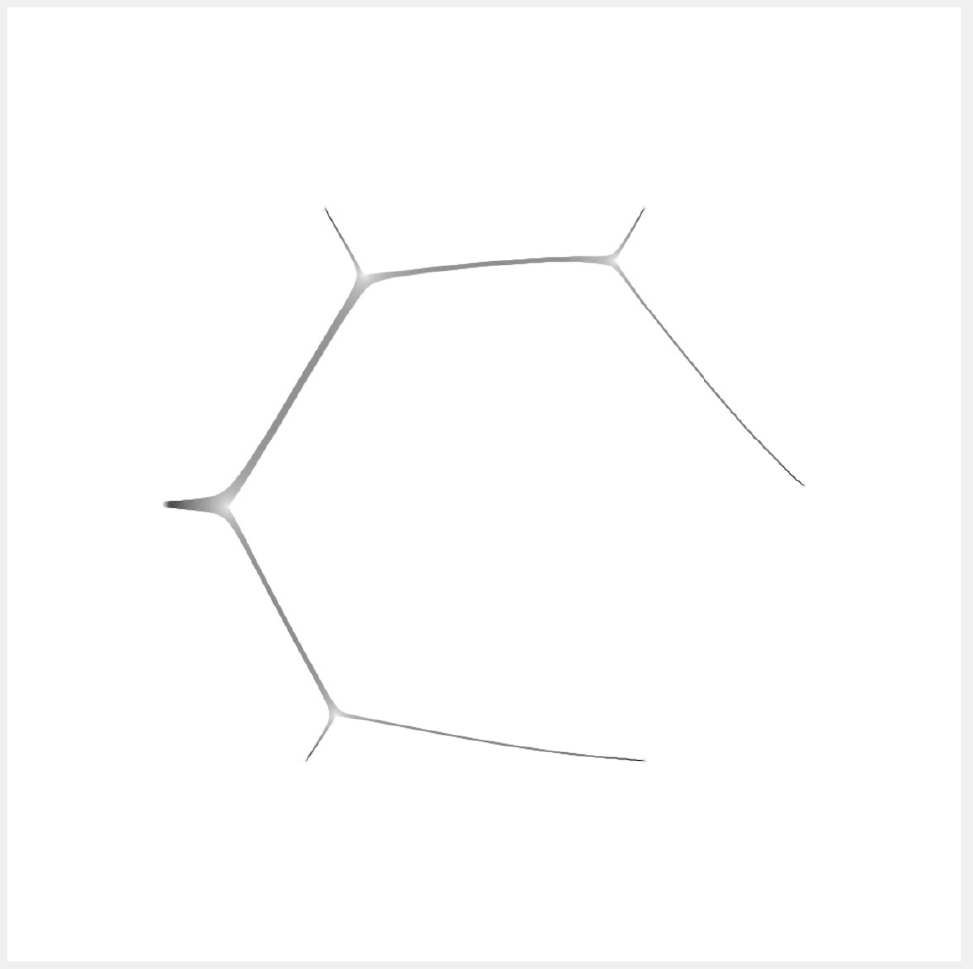}
\begin{picture}(0,0)(1,0)
\put(-113.12,57.52){\small$+$}
\put(-88.62,16.12){\small$-$}
\put(-34.62,16.12){\small$-$}
\put(-84.62,105.12){\small$-$}
\put(-34.62,105.12){\small$-$}
\put(-6.62,60.12){\small$-$}
\end{picture}
\begin{picture}(0,0)(1,0)\put(-60.1,-2.22){(a)}\end{picture}
\includegraphics[width = 0.3\textwidth,trim=25 15 35 15,clip]{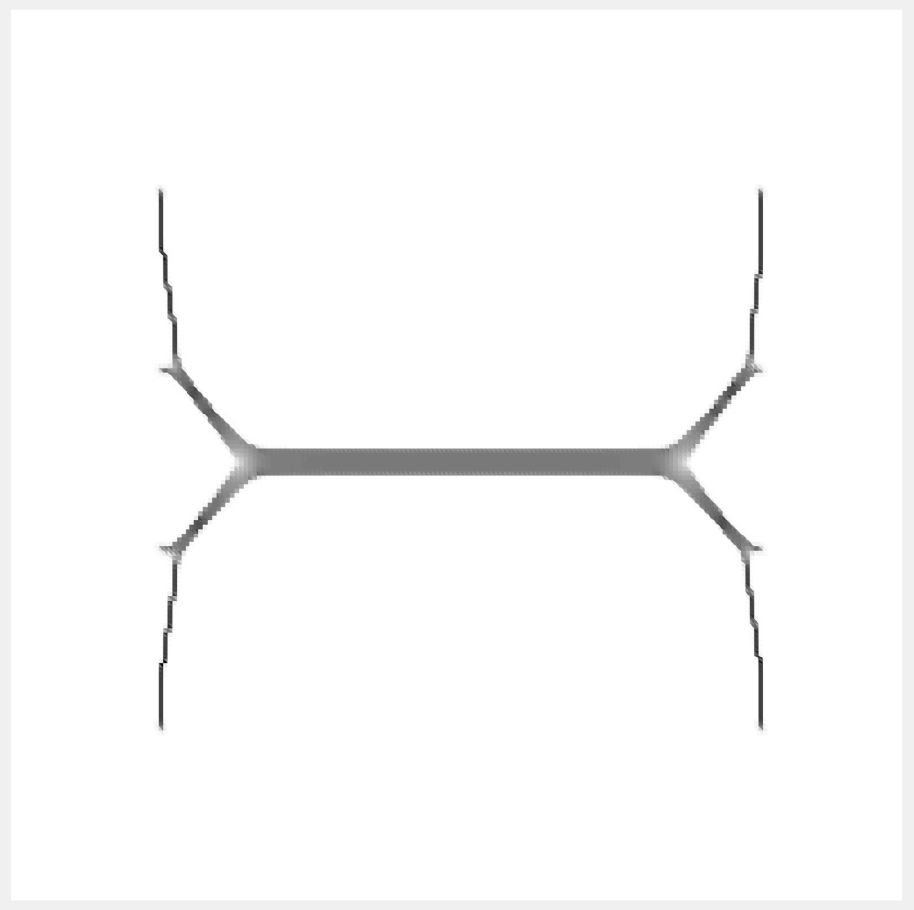}
\begin{picture}(0,0)(1,0)
\put(-111.12,18.52){\small$+$}
\put(-111.12,46.52){\small$+$}
\put(-111.12,75.52){\small$+$}
\put(-111.12,102.52){\small$+$}
\put(-5.12,18.52){\small$-$}
\put(-5.12,46.52){\small$-$}
\put(-5.12,75.52){\small$-$}
\put(-5.12,102.52){\small$-$}
\end{picture}
\begin{picture}(0,0)(1,0)\put(-60.1,-2.22){(b)}\end{picture}
\includegraphics[width = 0.3\textwidth,trim=20 10 15 10,clip]{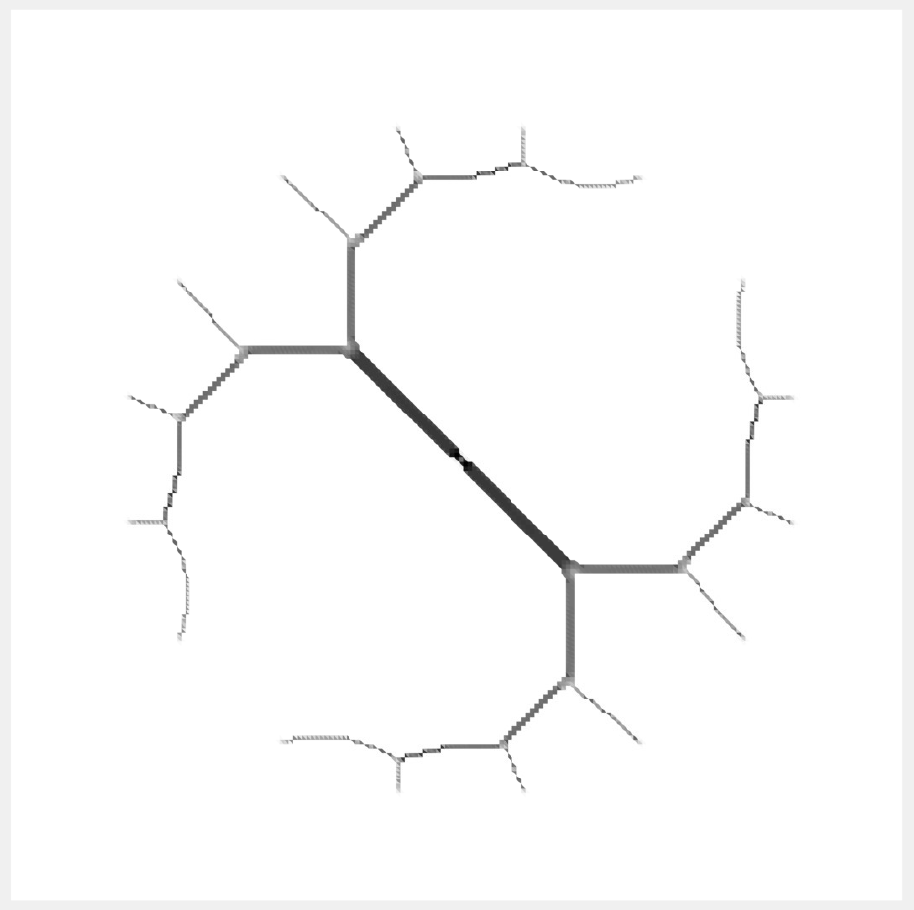}
\begin{picture}(0,0)(1,0)
\put(-59.12,60.52){\small$+$}

\put(-88.12,13.52){\small$-$}
\put(-72.12,6.52){\small$-$}
\put(-53.12,5.52){\small$-$}
\put(-36.12,12.52){\small$-$}

\put(-20.12,26.52){\small$-$}
\put(-11.12,44.52){\small$-$}
\put(-11.12,64.52){\small$-$}
\put(-17.12,82.52){\small$-$}

\put(-107.12,28.52){\small$-$}
\put(-114.12,46.52){\small$-$}
\put(-114.12,64.52){\small$-$}
\put(-103.12,83.52){\small$-$}

\put(-89.12,99.52){\small$-$}
\put(-71.12,106.52){\small$-$}
\put(-53.12,106.52){\small$-$}
\put(-37.12,100.52){\small$-$}
\end{picture}
\begin{picture}(0,0)(1,0)\put(-60.1,-2.22){(c)}\end{picture}
\caption{Numerically optimized transport networks with $\tau(m)=1+0.05m$ for transport from (a) one mass source to five sinks, (b) four mass sources to four mass sinks, (c) a central mass source to 16 sinks around it.}
\label{fig:teaser}
\end{figure}
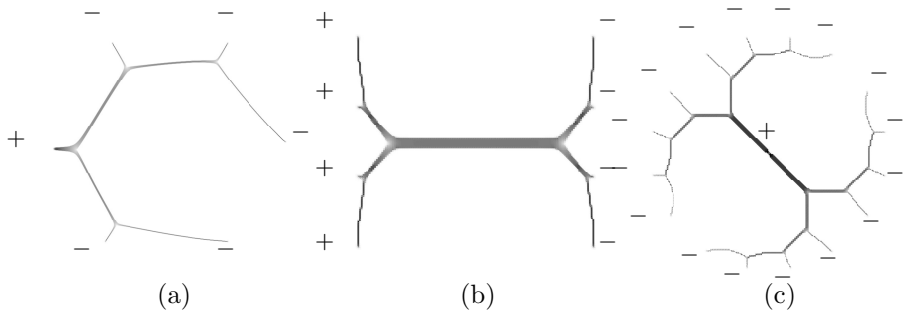


There exist multiple equivalent formulations for branched transportation problems.
One particular formulation models the mass flux $\flux$ as an unknown vector-valued Radon measure that describes the transport from $\mu_+$ to $\mu_-$. The cost functional of the flux, which one seeks to minimize, is defined as
\begin{equation}\label{eqn:GilbertEnergy}
E^{\mu_+,\mu_-}[\flux]=
\int_S\tau(|\sigma(x)|)\,\d\hdone(x)+\tau'(0)|\flux^\perp|(\R^n)
\end{equation}
if $\dive\flux=\mu_+-\mu_-$ in the distributional sense (so that indeed $\flux$ transports mass $\mu_+$ to $\mu_-$) and $\flux=\sigma\hdone\restr S+\flux^\perp$ for some countably $1$-rectifiable set $S\subset\R^n$, an $\hdone\restr S$-measurable function $\sigma:S\to\R^n$, and a diffuse part $\flux^\perp$ (composed of a Cantor measure and an absolutely continuous measure). Otherwise, $E^{\mu_+,\mu_-}[\flux]=\infty$.

In this article we devise phase field approximations of the functional $E^{\mu_+,\mu_-}$ in two space dimensions in the case of a piecewise affine transportation cost
\begin{equation*}
\tau(m)=\min\{\alpha_0 m,\alpha_1 m+\beta_1,\ldots,\alpha_N m+\beta_N\}
\end{equation*}
with positive parameters $\alpha_i,\beta_i$.
In fact, for a positive phase field parameter $\varepsilon$ we consider (a slightly improved but less intuitive version of) the phase field functional
\begin{multline*}
E_\varepsilon^{\mu_+,\mu_-}[\sigma,\varphi_1,\ldots,\varphi_N]
=\int_\Omega\min\left\{\alpha_0|\sigma(x)|,\min_{i=1,\ldots,N}\{\varphi_i(x)^2+\alpha_i^2\varepsilon^2/\beta_i\}\frac{|\sigma(x)|^2}{2\varepsilon}\right\}\\
+\sum_{i=1}^N\frac{\beta_i}2\left[\varepsilon|\nabla\varphi_i(x)|^2+\frac{(\varphi_i(x)-1)^2}\varepsilon\right]\,\d x
\end{multline*}
if $\dive\sigma=\mu_+^\varepsilon-\mu_-^\varepsilon$ and $E_\varepsilon^{\mu_+,\mu_-}[\sigma,\varphi_1,\ldots,\varphi_N]=\infty$ otherwise.
Our main result \Cref{thm:GammaLimit} shows that $E_\varepsilon^{\mu_+,\mu_-}$
$\Gamma$-converges in a certain topology to $E^{\mu_+,\mu_-}$ as $\varepsilon\searrow0$. Here the vector field $\sigma$ approximates the mass flux $\flux$, the auxiliary scalar phase fields $\varphi_1,\ldots,\varphi_N$ disappear in the limit, and $\mu_+^\varepsilon,\mu_-^\varepsilon$ are smoothed versions of $\mu_+,\mu_-$.
One motivation for our particular choice of $\tau$ is that this allows a phase field version of the urban planning model, the other motivation is that any concave cost $\tau$ can be approximated by a piecewise affine function.
Those phase field approximations then allow to find numerical approximations of optimal mass fluxes.

\subsection{Related work}
Phase field approximations represent a widely used tool to approach solutions to optimization problems depending on lower-dimensional sets. The concept takes advantage of the definition of $\Gamma$-convergence~\cite{Br98,Bra2002,DalMaso} to approximate singular energies by smoother elliptic functionals such that the associated minimizers converge as well. The term \textit{phase field} is due to Modica and Mortola~\cite{Mod_Mort} who study a rescaled version of the Cahn and Hilliard functional which models the phase transition between two immiscible liquids and which turns out to approximate the perimeter functional of a set. Subsequently a similar idea has been used by Ambrosio and Tortorelli in~\cite{Amb_Tort1,Amb_Tort2} to obtain an approximation to the Mumford-Shah functional~\cite{MuSh89}. Later on these techniques have been used as well in fracture theory and optimal partitions problems~\cite{Iur,ContiFocardiIurlano}.
More recently such phase field approximations have also been applied to branched transportation problems and the Steiner minimal tree problem (to find the graph of smallest length connecting a given set of points) which both feature high combinatorial complexity~\cite{Karp}. For instance, in~\cite{OuSa11} the authors propose an approximation to the branched transport problem based on the Modica--Mortola functional in which the phase field is replaced by a vector-valued function satisfying a divergence constraint. Similarly, in~\cite{MR3337998} Santambrogio et al.\ study a variational approximation to the Steiner minimal tree problem in which the connectedness constraint of the graph is enforced trough the introduction of a geodesic distance depending on the phase field. Our phase field approximations can be viewed as a generalization of recent work by two of the authors~\cite{ChaFerMer16,ChaFerMer17}, in which essentially~\eqref{eqn:GilbertEnergy} for $\tau(m)=\alpha m+\beta$ with $\alpha,\beta>0$ is approximated by an Ambrosio--Tortorelli-type functional defined as
\begin{equation*}
\tilde E_\varepsilon^{\mu_+,\mu_-}[\sigma,\varphi]=
\int_\Omega\frac1{2\varepsilon}\varphi(x)^2|\sigma(x)|^2+\frac\beta2\left(\varepsilon|\nabla\varphi(x)|^2+\frac1\varepsilon(\varphi(x)-1)^2\right)\,\d x
\end{equation*}
if  $\dive\sigma=\rho_\varepsilon*(\mu_+-\mu_-)$ for a smoothing kernel $\rho_\varepsilon$ and if $\varphi\geq\frac\alpha{\sqrt\beta}\varepsilon$ almost everywhere, and $\tilde E_\varepsilon^{\mu_+,\mu_-}[\sigma,\varphi]=\infty$ otherwise. The function $1-\varphi$ may be regarded as a smooth version of the characteristic function of a $1$-rectifiable set (the Steiner tree) whose total length is approximated by the Ambrosio--Tortorelli phase field term in parentheses. In addition, the first term in the integral forces this set to contain the support of a vector field $\sigma$ which encodes the mass flux from $\mu_+$ to $\mu_-$.

\subsection{Notation}
Throughout, $\Omega\subset\R^2$ is an open bounded domain with Lipschitz boundary.
The spaces of scalar and $\R^n$-valued continuous functions on the closure $\overline\Omega$ are denoted $\cont(\overline\Omega)$ and $\cont(\overline\Omega;\R^n)$
(for $m$ times continuously differentiable functions we use $\cont^m$)
and are equipped with the supremum norm $\|\cdot\|_\infty$.
Their topological duals are the spaces of scalar and $\R^n$-valued Radon measures (regular countably additive measures) $\meas(\overline\Omega)$ and $\meas(\overline\Omega;\R^n)$, equipped with the total variation norm $\|\cdot\|_\meas$.
Weak-* convergence in these spaces will be indicated by $\weakstarto$.
The subset $\prob(\overline\Omega)\subset\meas(\overline\Omega)$ shall be the space of probability measures (nonnegative Radon measures with total variation $1$).
The total variation measure of some Radon measure $\lambda$ will be denoted $|\lambda|$, and its Radon--Nikodym derivative with respect to $|\lambda|$ as $\frac\lambda{|\lambda|}$.
The restriction of $\lambda$ to a measurable set $A$ is abbreviated $\lambda\restr A$.
The pushforward of a measure $\lambda$ under a measurable map $T$ is denoted $\pushforward T\lambda$.
The standard Lebesgue and Sobolev spaces are indicated by $L^p(\Omega)$ and $W^{k,p}(\Omega)$, respectively;
if they map into $\R^n$ we write $L^p(\Omega;\R^n)$ and $W^{k,p}(\Omega;\R^n)$.
The associated norms are indicated by $\|\cdot\|_{L^p}$ and $\|\cdot\|_{W^{k,p}}$.
Finally, the $n$-dimensional Hausdorff measure is denoted $\hd^n$.\\

\medskip
\textbf{Structure of the paper:} In Section~\ref{Section2} we introduce the approximating energies and precisely state our results. Section~\ref{Section3} is devoted to the proof of the $\Gamma$-convergence result. Its first three subsections deal with the $\Gamma-\liminf$ inequality which is obtained via slicing, while the remaining two subsections prove an equicoercivity result and the $\Gamma-\limsup $ inequality. Finally, in Section~\ref{Section4} we introduce a numerical discretization and algorithmic scheme to perform some exemplary simulations of branched transportation networks.

\section{Model summary and $\Gamma$-convergence result}\label{Section2}
Here we state in more detail the considered variational model for transportation networks and its phase field approximation as well as the $\Gamma$-convergence result.

\subsection{Introduction of the energies}
Before we state the original energy and its phase field approximation, let us briefly recall the objects representing the transportation networks.
\begin{definition}[Divergence measure vector field and mass flux]
\begin{enumerate}
\item
A \emph{divergence measure vector field} is a measure $\flux\in\meas(\overline\Omega;\R^2)$,
whose weak divergence is a Radon measure, $\dive\flux\in\meas(\overline\Omega)$, where the weak divergence is defined as
\begin{equation*}
\int_{\R^2}\psi\,\d\dive\flux
=-\int_{\R^2}\nabla\psi\cdot\d\flux
\qquad\text{for all }\psi\in\cont^1(\R^2)\text{ with compact support.}
\end{equation*}
By \cite{Si07}, any divergence measure vector field $\flux$ can be decomposed as
\begin{equation*}
\flux=m_\flux\theta_\flux\hdone\restr S_\flux+\flux^\perp\,,
\end{equation*}
where $S_\flux\subset\Omega$ is countably $1$-rectifiable, $m_\flux:S_\flux\to[0,\infty)$ is $\hdone\restr S_\flux$-measurable, $\theta_\flux:S_\flux\to S^1$ is $\hdone\restr S_\flux$-measurable and orients the approximate tangent space of $S_\flux$,
and $\flux^\perp$ is $\hdone$-diffuse, that is, it is singular with respect to the one-dimensional Hausdorff measure on any countably $1$-rectifiable set (it is the sum of a Lebesgue-continuous and a Cantor part).
\item
A divergence measure vector field $\flux$ is \emph{polyhedral} if it is a finite linear combination
$$\flux=\sum_{i=1}^Mm_i\theta_i\hd^1\restr e_i$$
of vector-valued line measures, where $M\in\N$ and for $i=1,\ldots,M$ we have $m_i\in\R$, $e_i\subset\R^2$ a straight line segment, and $\theta_i$ its unit tangent vector.
\item
Given $\mu_+,\mu_-\in\prob(\overline\Omega)$ with compact support in $\Omega$, a \emph{mass flux} between $\mu_+$ and $\mu_-$ is a divergence measure vector field $\flux$ with $\dive\flux=\mu_+-\mu_-$.
The set of mass fluxes between $\mu_+$ and $\mu_-$ is denoted
\begin{equation*}
X^{\mu_+,\mu_-}
=\{\flux\in\meas(\overline\Omega;\R^2)\,|\,\dive\flux=\mu_+-\mu_-\}\,.
\end{equation*}
\end{enumerate}
\end{definition}

A mass flux between $\mu_+$ and $\mu_-$ can be interpreted as the material flow that transports the material from the initial mass distribution $\mu_+$ to the final mass distribution $\mu_-$.
Next we specify the cost functional for branched transportation networks for which we will propose a phase field approximation.

\begin{definition}[Cost functional]
\begin{enumerate}
\item
Given $N\in\N$ and $\alpha_0>\alpha_1>\dots>\alpha_N>0$, $0=\beta_0<\beta_1<\dots<\beta_N<\infty$, we define the piecewise affine \emph{transport cost} $\tau:[0,\infty)\to[0,\infty)$,
\begin{equation*}
\tau(m)=\min_{i=0,\dots,N}\{\alpha_im+\beta_i\}\,.
\end{equation*}
If $\alpha_0=\infty$ we interpret $\tau$ as
\begin{equation*}
\tau(m)=\begin{cases}0&\text{if }m=0\,,\\\min_{i=1,\dots,N}\{\alpha_im+\beta_i\}&\text{else.}\end{cases}
\end{equation*}
\item
We call $\mu_+,\mu_-\in\prob(\overline\Omega)$ an \emph{admissible source and sink} if they have compact support in $\Omega$
and in the case of $\alpha_0=\infty$ can additionally be written as a finite linear combination of Dirac masses.
\item
Given admissible $\mu_+,\mu_-\in\prob(\overline\Omega)$, we define the \emph{cost functional} $\E^{\mu_+,\mu_-}:X^{\mu_+,\mu_-}\to[0,\infty]$,
\begin{equation*}
\E^{\mu_+,\mu_-}[\flux]=
\int_{S_\flux}\tau(m_\flux(x))\,\d\hdone(x)+\tau'(0)|\flux^\perp|(\overline\Omega)
\end{equation*}
for $\alpha_0<\infty$ (above, $\tau'(0)=\alpha_0$ denotes the right derivative in $0$) and otherwise
\begin{equation*}
\E^{\mu_+,\mu_-}[\flux]=\begin{cases}
\int_{S_\flux}\tau(m_\flux(x))\,\d\hdone(x)&\text{if }\flux^\perp=0,\\
\infty&\text{else.}
\end{cases}
\end{equation*}
\end{enumerate}
\end{definition}

The phase field functional approximating $\E^{\mu_+,\mu_-}$ will depend on a vector field $\sigma$ approximating the mass flux and $N$ Ambrosio--Tortorelli phase fields $\varphi_1,\ldots,\varphi_N$ that indicate which term in the definition of $\tau$ is active.

\begin{definition}[Phase field cost functional]\label{def:phase fieldFunctional}
Let $\varepsilon>0$ and $\rho:\R^2\to[0,\infty)$ be a fixed smooth convolution kernel with support in the unit disk and $\int_{\R^2}\rho\,\d x=1$.
\begin{enumerate}
\item
Given admissible source and sink $\mu_+,\mu_-\in\prob(\overline\Omega)$, the \emph{smoothed source and sink} are
\begin{equation*}
\mu_\pm^\varepsilon=\rho_\varepsilon*\mu_\pm\,,\quad\text{where }\rho_\varepsilon=\tfrac1{\varepsilon^2}\rho(\tfrac\cdot\varepsilon)\,.
\end{equation*}
\item
The \emph{set of admissible functions} is
\begin{multline*}
X_\varepsilon^{\mu_+,\mu_-}
=\left\{(\sigma,\varphi_1,\ldots,\varphi_N)\in L^{2}(\Omega;\R^2)\times W^{1,2}(\Omega)^N\,\right|\\
\left.\,\dive\sigma=\mu_+^\varepsilon-\mu_-^\varepsilon,\,\varphi_1=\ldots=\varphi_N=1\text{ on }\partial\Omega\right\}\,.
\end{multline*}
\item
The \emph{phase field cost functional} is given by $\E_\varepsilon^{\mu_+,\mu_-}:X_\varepsilon^{\mu_+,\mu_-}\to[0,\infty)$,
\begin{equation*}
\E_\varepsilon^{\mu_+,\mu_-}[\sigma,\varphi_1,\ldots,\varphi_N]
=\int_\Omega\omega_\varepsilon\left(\alpha_0,\frac{\gamma_\varepsilon(x)}{\varepsilon},|\sigma(x)|\right)\,\d x
+\sum_{i=1}^N\beta_i\L_\varepsilon[\varphi_i]\,,
\end{equation*}
where we abbreviated (with some $p>1$)
\begin{align*}
\L_\varepsilon[\varphi]&=\frac12\int_\Omega\left[\varepsilon|\nabla\varphi(x)|^2+\frac{(\varphi(x)-1)^2}\varepsilon\right]\,\d x\,,\\
\gamma_\varepsilon(x)&=\min_{i=1,\ldots,N}\left\{\varphi_i(x)^2+\alpha_i^2\varepsilon^2/\beta_i\right\}\,,\\
\omega_\varepsilon\left(\alpha_0,\frac{\gamma_\varepsilon(x)}{\varepsilon},|\sigma(x)|\right)&=\left.\begin{cases}
\frac{\gamma_\varepsilon}{\varepsilon}\frac{|\sigma|^2}2&\text{if }|\sigma|\leq\frac{\alpha_0}{\gamma_\varepsilon/\varepsilon}\\
\alpha_0(|\sigma|-\frac{\alpha_0}{2\gamma_\varepsilon/\varepsilon})&\text{if }|\sigma|>\frac{\alpha_0}{\gamma_\varepsilon/\varepsilon}
\end{cases}\right\}+\varepsilon^p|\sigma(x)|^2&\text{for }\alpha_0<\infty\,,\\
\omega_\varepsilon\left(\alpha_0,\frac{\gamma_\varepsilon(x)}{\varepsilon},|\sigma(x)|\right)&=\frac{\gamma_\varepsilon}{\varepsilon}\frac{|\sigma|^2}2&\text{for }\alpha_0=\infty\,.
\end{align*}
\end{enumerate}
\end{definition}

Note that the pointwise minimum inside $\gamma_\varepsilon$ is well-defined almost everywhere, since all elements of $X_\varepsilon^{\mu_+,\mu_-}$ are Lebesgue-measurable.
Note also that for fixed phase fields $\varphi_1,\ldots,\varphi_N$ the phase field cost functional $\E_\varepsilon^{\mu_+,\mu_-}$ is convex in $\sigma$.
This ensures the existence of minimizers for $\E_\varepsilon^{\mu_+,\mu_-}$, which follows by a standard application of the direct method.

\begin{proposition}[Existence of minimizers to the phase field functional]
The phase field cost functional $\E_\varepsilon^{\mu_+,\mu_-}$ has a minimizer $(\sigma,\varphi_1,\ldots,\varphi_N)\in X_\varepsilon^{\mu_+,\mu_-}$.
\end{proposition}
\begin{proof}
The functional is bounded below by $0$ and has a nonempty domain.
Indeed, choose $\hat\varphi_1\equiv\ldots\equiv\hat\varphi_N\equiv1$ and $\hat\sigma=\nabla\psi$, where $\psi$ solves $\Delta\psi=\mu_+^\varepsilon-\mu_-^\varepsilon$ in $\Omega$ with Neumann boundary conditions $\nabla\psi\cdot\nu_{\partial\Omega}=0$, $\nu_{\partial\Omega}$ being the unit outward normal to $\partial\Omega$.
(Since $\int_\Omega\mu_+^\varepsilon-\mu_-^\varepsilon\,\d x=0$, a solution $\psi$ exists and lies in $W^{2,2}(\Omega)$ by standard elliptic regularity.)
Obviously, $(\hat\sigma,\hat\varphi_1,\ldots,\hat\varphi_N)\in X_\varepsilon^{\mu_+,\mu_-}$ with $\E_\varepsilon^{\mu_+,\mu_-}[\hat\sigma,\hat\varphi_1,\ldots,\hat\varphi_N]<\infty$.

Now consider a minimizing sequence $(\sigma^k,\varphi_1^k,\ldots,\varphi_N^k)\in X_\varepsilon^{\mu_+,\mu_-}$, $k\in\N$,
with $$\E_\varepsilon^{\mu_+,\mu_-}[\sigma^k,\varphi_1^k,\ldots,\varphi_N^k]\to\inf\E_\varepsilon^{\mu_+,\mu_-}$$ monotonically as $k\to\infty$.
Since $\E_\varepsilon^{\mu_+,\mu_-}$ is coercive with respect to $H=L^{2}(\Omega;\R^2)\times W^{1,2}(\Omega)^N$,
$(\sigma^k,\varphi_1^k,\ldots,\varphi_N^k)$ is uniformly bounded in $H$ so that we can extract a weakly converging subsequence, still indexed by $k$ for simplicity,
$$(\sigma^k,\varphi_1^k,\ldots,\varphi_N^k)\weakto(\sigma,\varphi_1,\ldots,\varphi_N)\,.$$
Due to the closedness of $X_\varepsilon^{\mu_+,\mu_-}$ with respect to weak convergence in $H$ we have $(\sigma,\varphi_1,\ldots,\varphi_N)\in X_\varepsilon^{\mu_+,\mu_-}$.
Note that the integrand of $\E_\varepsilon^{\mu_+,\mu_-}$ is convex in $\sigma(x)$ and the $\nabla\varphi_i(x)$ as well as continuous in $\sigma(x)$ and the $\varphi_i(x)$,
thus $\E_\varepsilon^{\mu_+,\mu_-}$ is lower semi-continuous along the sequence.
Indeed, consider a subsequence along which each term $\L_\varepsilon[\varphi_i^k]$ converges
and along which the $\varphi_i^k$ converge pointwise almost everywhere (so that also $\gamma_\varepsilon^k(x)=\min_{i=1,\ldots,N}\left\{\varphi_i^k(x)^2+\alpha_i^2\varepsilon^2/\beta_i\right\}$ converges for almost all $x\in\Omega$).
By Mazur's lemma, a sequence of convex combinations $\sum_{j=k}^{m_k}\lambda_j^k\sigma^j$ of the $\sigma^k$ converges strongly (and up to another subsequence again pointwise)
so that by Fatou's lemma we have
\begin{align*}
\inf\E_\varepsilon^{\mu_+,\mu_-}
&=\lim_{k\to\infty}\E_\varepsilon^{\mu_+,\mu_-}[\sigma^k,\varphi_1^k,\ldots,\varphi_N^k]\\
&=\lim_{k\to\infty}\int_\Omega\omega_\varepsilon\left(\alpha_0,\frac{\gamma_\varepsilon^k(x)}{\varepsilon},|\sigma^k(x)|\right)\,\d x
+\sum_{i=1}^N\beta_i\lim_{k\to\infty}\L_\varepsilon[\varphi_i^k]\\
&\geq\lim_{k\to\infty}\sum_{j=k}^{m_k}\lambda_j^k\int_\Omega\omega_\varepsilon\left(\alpha_0,\frac{\gamma_\varepsilon^j(x)}{\varepsilon},|\sigma^j(x)|\right)\,\d x
+\sum_{i=1}^N\beta_i\L_\varepsilon[\varphi_i]\\
&\geq\int_\Omega\liminf_{k\to\infty}\sum_{j=k}^{m_k}\lambda_j^k\omega_\varepsilon\left(\alpha_0,\frac{\gamma_\varepsilon^j(x)}{\varepsilon},|\sigma^j(x)|\right)\,\d x
+\sum_{i=1}^N\beta_i\L_\varepsilon[\varphi_i]\\
&\geq\int_\Omega\liminf_{k\to\infty}\sum_{j=k}^{m_k}\lambda_j^k\omega_\varepsilon\left(\alpha_0,\inf_{i=k,\ldots,m_k}\frac{\gamma_\varepsilon^i(x)}{\varepsilon},|\sigma^j(x)|\right)\,\d x
+\sum_{i=1}^N\beta_i\L_\varepsilon[\varphi_i]\\
&\geq\int_\Omega\liminf_{k\to\infty}\omega_\varepsilon\left(\alpha_0,\inf_{i=k,\ldots,m_k}\frac{\gamma_\varepsilon^i(x)}{\varepsilon},\sum_{j=k}^{m_k}\lambda_j^k|\sigma^j(x)|\right)\,\d x
+\sum_{i=1}^N\beta_i\L_\varepsilon[\varphi_i]\\
&=\E_\varepsilon^{\mu_+,\mu_-}[\sigma,\varphi_1,\ldots,\varphi_N]\,,
\end{align*}
where we exploited the weak lower semi-continuity of $\L_\varepsilon$, the monotonicity of $\omega_\varepsilon\left(\alpha_0,\frac{\gamma_\varepsilon(x)}{\varepsilon},|\sigma(x)|\right)$ in its second argument, its convexity in its last argument, and its continuity in its latter two arguments.
\end{proof}

\begin{remark}[Regularization of $\sigma$]
Note that the phase field cost functional $\E_\varepsilon^{\mu_+,\mu_-}$ is $L^2(\Omega;\R^2)$-coercive in $\sigma$,
which is essential to have sequentially weak compactness of subsets of $X_\varepsilon^{\mu_+,\mu_-}$ with finite cost (and as a consequence existence of minimizers).
For $\alpha_0<\infty$ this is ensured by the regularization term $\varepsilon^p|\sigma|^2$ (which has no other purpose).
Without it, the functional would only feature weak-$*$ coercivity for $\sigma$ in $\meas(\overline\Omega;\R^2)$,
however, the integral $\int_\Omega\omega_\varepsilon\big(\alpha_0,\frac{\gamma_\varepsilon(x)}\varepsilon,|\sigma(x)|\big)\,\d x$ with $\gamma_\varepsilon$ Lebesgue-measurable
would in general not be well-defined for $\sigma\in\meas(\overline\Omega;\R^2)$.
\end{remark}

\begin{remark}[Motivation of $\omega_\varepsilon\left(\alpha_0,\frac{\gamma_\varepsilon(x)}{\varepsilon},|\sigma(x)|\right)$ via relaxation]\label{rem:relaxation}
Keeping the phase fields $\varphi_1,\ldots,\varphi_N$ fixed and ignoring the regularizing term $\varepsilon^p|\sigma|^2$,
the integrand $\omega_\varepsilon(\alpha_0,\frac{\gamma_\varepsilon}\varepsilon,|\sigma|)$ is the convexification in $\sigma$ of
\begin{equation*}
\min\left\{\alpha_0|\sigma|,\frac{(\varphi_1^2+\alpha_1^2\varepsilon^2/\beta_1)|\sigma|^2}{2\varepsilon},\ldots,\frac{(\varphi_N^2+\alpha_N^2\varepsilon^2/\beta_N)|\sigma|^2}{2\varepsilon}\right\}
=\min\left\{\alpha_0|\sigma|,\frac{\gamma_\varepsilon}\varepsilon\frac{|\sigma|^2}{2}\right\}\,,
\end{equation*}
which shows the intuition of the phase field functional much clearer.
Indeed, the minimum over $N+1$ terms parallels the minimum in the definition of $\tau$,
and the $i$\textsuperscript{th} term for $i=0,\ldots,N$ describes (part of) the transportation cost $\alpha_im+\beta_i$.
However, since the above is not convex with respect to $\sigma$, a functional with this integrand would not be weakly lower semi-continuous in $\sigma$
and consequently possess no minimizers in general.
Taking the lower semi-continuous envelope corresponds to replacing the above by $\omega_\varepsilon(\alpha_0,\frac{\gamma_\varepsilon}\varepsilon,|\sigma|)$
(note that this only ensures existence of minimizers, but will not change the $\Gamma$-limit of the phase field functional).
\end{remark}

\subsection{Statement of $\Gamma$-convergence and equi-coercivity}
Let us extend both $\E^{\mu_+,\mu_-}$ and $\E_\varepsilon^{\mu_+,\mu_-}$ to $\meas(\overline\Omega;\R^2)\times L^1(\Omega)^N$ via
\begin{align*}
E^{\mu_+,\mu_-}[\sigma,\varphi_1,\ldots,\varphi_N]&=\begin{cases}
\E^{\mu_+,\mu_-}[\sigma]&\text{if }\sigma\in X^{\mu_+,\mu_-}\text{ and }\varphi_1=\ldots=\varphi_N=1\text{ almost everywhere},\\
\infty&\text{else,}
\end{cases}\\
E_\varepsilon^{\mu_+,\mu_-}[\sigma,\varphi_1,\ldots,\varphi_N]&=\begin{cases}
\E_\varepsilon^{\mu_+,\mu_-}[\sigma,\varphi_1,\ldots,\varphi_N]&\text{if }(\sigma,\varphi_1,\ldots,\varphi_N)\in X_\varepsilon^{\mu_+,\mu_-},\\
\infty&\text{else.}
\end{cases}
\end{align*}

We have the following $\Gamma$-convergence result.

\begin{theorem}[Convergence of phase field cost functional]\label{thm:GammaLimit}
For admissible $\mu_+,\mu_-\in\prob(\overline\Omega)$ we have $$\Gamma-\lim_{\varepsilon\to0}E_\varepsilon^{\mu_+,\mu_-}=E^{\mu_+,\mu_-}\,,$$ where the $\Gamma$-limit is with respect to weak-$*$ convergence in $\meas(\overline\Omega;\R^2)$ and strong convergence in $L^1(\Omega)^N$.
\end{theorem}

The proof of this result is provided in the next section.
Together with the following equicoercivity statement, whose proof is also deferred to the next section,
we have that minimizers of the phase field cost functional $E_\varepsilon^{\mu_+,\mu_-}$ approximate minimizers of the original cost functional $E^{\mu_+,\mu_-}$.

\begin{theorem}[Equicoercivity]\label{thm:equicoercivity}
For $\varepsilon\to0$ let $(\sigma^\varepsilon,\varphi_1^\varepsilon,\ldots,\varphi_N^\varepsilon)$ be a sequence
with uniformly bounded phase field cost functional $E_\varepsilon^{\mu_+,\mu_-}[\sigma^\varepsilon,\varphi_1^\varepsilon,\ldots,\varphi_N^\varepsilon]<C<\infty$.
Then, along a subsequence, $\sigma^\varepsilon\weakstarto\sigma$ in $\meas(\overline\Omega;\R^2)$ for some $\sigma\in\meas(\overline\Omega;\R^2)$ and $\varphi_i^\varepsilon\to1$ in $L^1(\Omega)$, $i=1,\ldots,N$.

As a consequence, if $\mu_+,\mu_-\in\prob(\overline\Omega)$ are admissible and such that there exists $\flux\in X^{\mu_+,\mu_-}$ with $\E^{\mu_+,\mu_-}[\flux]<\infty$,
then any sequence of minimizers of $E_\varepsilon^{\mu_+,\mu_-}$ contains a subsequence converging to a minimizer of $E^{\mu_+,\mu_-}$ as $\varepsilon\to0$.
\end{theorem}

\begin{remark}[Phase field boundary conditions]
Recall that we imposed boundary conditions $\varphi_i=1$ on $\partial\Omega$.
Without those, the recovery sequence from the following section could easily be adapted such that all full phase field profiles near the boundary will be replaced by half, one-sided phase field profiles.
It is straightforward to show that the resulting limit functional would become
\begin{multline*}
\int_{S_\flux\cap\Omega}\min\{\alpha_0m_\flux,\alpha_1m_\flux+\beta_1,\ldots,\alpha_Nm_\flux+\beta_N\}\\
+\int_{S_\flux\cap\partial\Omega}\min\{\alpha_0m_\flux,\alpha_1m_\flux+\beta_1/2,\ldots,\alpha_Nm_\flux+\beta_N/2\}
+\alpha_0|\flux^\perp|(\overline\Omega)\,,
\end{multline*}
where fluxes along the boundary are cheaper and thus preferred.
\end{remark}

\begin{remark}[Divergence measure vector fields and flat chains]\label{rem:flatChains}
Any divergence measure vector field can be identified with a flat $1$-chain or a locally normal $1$-current (see for instance \cite[Sec.\,5]{Si07} or \cite[Rem.\,2.29(3)]{BrWi17}; comprehensive references for flat chains and currents are \cite{Wh57,Fe69}).
Furthermore, for a sequence $\sigma^j$, $j\in\N$, of divergence measure vector fields with uniformly bounded $\|\dive\sigma^j\|_\meas$, weak-$*$ convergence is equivalent to convergence of the corresponding flat $1$-chains with respect to the flat norm \cite[Rem.\,2.29(4)]{BrWi17}.
Analogously, scalar Radon measures of finite total variation and bounded support can be identified with flat $0$-chains or locally normal $0$-currents \cite[Thm.\,2.2]{Wh99},
and for a bounded sequence of compactly supported scalar measures, weak-$*$ convergence is equivalent to convergence with respect to the flat norm of the corresponding flat $0$-chains.

From the above it follows that in \Cref{thm:GammaLimit,thm:equicoercivity} we may replace weak-$*$ convergence by convergence with respect to the flat norm.
Indeed, for both results it suffices to consider sequences $(\sigma^\varepsilon,\varphi_1^\varepsilon,\ldots,\varphi_N^\varepsilon)$ with uniformly bounded cost $E_\varepsilon^{\mu_+,\mu_-}$.
For those we have uniformly bounded $\|\sigma^\varepsilon\|_\meas$ (by \Cref{thm:equicoercivity}) as well as uniformly bounded $\|\dive\sigma^\varepsilon\|_\meas=\|\mu_+^\varepsilon-\mu_-^\varepsilon\|_\meas$
so that weak-$*$ and flat norm convergence are equivalent.
\end{remark}

\section{The $\Gamma$-limit of the phase field functional}\label{Section3}
In this section we prove the $\Gamma$-convergence result. As is canonical, we begin with the $\liminf$-inequality, after which we prove the $\limsup$-inequality as well as equicoercivity.

\subsection{The $\Gamma-\liminf$ inequality for the dimension-reduced problem}
Here we consider the energy reduced to codimension-$1$ slices of the domain $\Omega$.
In our particular case of a two-dimensional domain, each slice is just one-dimensional, which will simplify notation a little (the procedure would be the same for higher codimensions, though).
The reduced functional depends on the (scalar) normal flux $\vartheta$ through the slice as well as the scalar phase fields $\varphi_1,\ldots,\varphi_N$ restricted to the slice.

\begin{definition}[Reduced functionals]\label{def:reducedFunctionals}
Let $I\subset\R$ be an open interval.
\begin{enumerate}
\item
The decomposition of a measure $\vartheta\in\meas(\overline I)$ into its atoms and the remainder is denoted
\begin{equation*}
\vartheta=m_\vartheta\hd^0\restr S_\vartheta+\vartheta^\perp\,,
\end{equation*}
where $S_\vartheta\subset\overline I$ is the set of atoms of $\vartheta$, $m_\vartheta:S_\vartheta\to\R$ is $\hd^0\restr S_\vartheta$-measurable, and $\vartheta^\perp$ contains no atoms.
\item
We define the \emph{reduced cost functional} $\F[\cdot;I]:\meas(\overline I)\to[0,\infty)$,
\begin{equation*}
\F[\vartheta;I]=\sum_{x\in S_\vartheta\cap I}\tau(|m_\vartheta(x)|)+\tau'(0)|\vartheta^\perp|(I)\,
\end{equation*}
for $\alpha_0<\infty$ (above, $\tau'(0)=\alpha_0$ denotes the right derivative in $0$) and otherwise
\begin{equation*}
\F[\vartheta;I]=\begin{cases}
\sum_{x\in S_\vartheta\cap I}\tau(|m_\vartheta(x)|)&\text{if }\vartheta^\perp=0,\\
\infty&\text{else.}
\end{cases}
\end{equation*}
Its extension to $\meas(\overline I)\times L^1(I)^N$ is $\f[\cdot;I]:\meas(\overline I)\times L^1(I)^N\to[0,\infty)$,
\begin{equation*}
\f[\vartheta,\varphi_1,\ldots,\varphi_N;I]=\begin{cases}
\F[\vartheta;I]&\text{if }\varphi_1=\ldots=\varphi_N=1\text{ almost everywhere,}\\
\infty&\text{else.}
\end{cases}
\end{equation*}
\item
For any $(\vartheta,\varphi_1,\ldots,\varphi_N)\in L^2(I)\times W^{1,2}(I)^N$ we define the \emph{reduced phase field functional} on $I$ as
 \begin{align*}
\F_\varepsilon[\vartheta,\varphi_1,\ldots,\varphi_N;I]
&=\int_I\omega_\varepsilon\left(\alpha_0,\tfrac{\gamma_\varepsilon(x)}{\varepsilon},|\vartheta(x)|\right)\,\d x
+\sum_{i=1}^N\beta_i\L_\varepsilon[\varphi_i;I]\,,\\
\L_\varepsilon[\varphi;I]&=\frac12\int_I\left[\varepsilon|\varphi'(x)|^2+\frac{(\varphi(x)-1)^2}\varepsilon\right]\,\d x\,,
\end{align*}
with $\omega_\varepsilon$ and $\gamma_\varepsilon$ from \Cref{def:phase fieldFunctional}.
Likewise we define $\f_\varepsilon[\cdot;I]:\meas(\overline I)\times L^1(I)^N\to[0,\infty)$,
\begin{equation*}
\f_\varepsilon[\vartheta,\varphi_1,\ldots,\varphi_N;I]
=\begin{cases}
\F_\varepsilon[\vartheta,\varphi_1,\ldots,\varphi_N;I]&\text{if }(\vartheta,\varphi^1,\dots,\varphi^N)\in L^2(I)\times W^{1,2}(I)^{N},\\
\infty&\text{else.}
\end{cases}
\end{equation*}
\end{enumerate}
\end{definition}

For notational convenience, we next introduce the sets $K_i^\varepsilon$ on which the pointwise minimum inside $\F_\varepsilon$ (or also $\E_\varepsilon^{\mu_+,\mu_-}$) is realized by the $i$\textsuperscript{th} element.

\begin{definition}[Cost domains]
For given $(\vartheta,\varphi_1,\ldots,\varphi_N)\in L^2(I)\times W^{1,2}(I)^{N}$ we set
\begin{align*}
K_0^\varepsilon
&=K_0^\varepsilon(\vartheta,\varphi_1,\ldots,\varphi_N;I)
= \left\{x\in I\,\middle|\,|\vartheta(x)|>\frac{\alpha_0\varepsilon}{\gamma_\varepsilon}\right\}\,,\\
K_i^\varepsilon
&=K_i^\varepsilon(\vartheta,\varphi_1,\ldots,\varphi_N;I)
=\left\{x\in I\setminus\textstyle\bigcup_{j=0}^{i-1}K_j^\varepsilon\,\middle|\,\varphi_i(x)^2+\tfrac{\alpha_i^2\varepsilon^2}{\beta_i}=\gamma_\varepsilon(x)\right\}\,,
\;i=1,\ldots,N\,.
\end{align*}
The sets are analogously defined for $(\sigma,\varphi_1,\ldots,\varphi_N)\in L^2(\Omega;\R^2)\times W^{1,2}(\Omega)^N$,
where we use the same notation (which case is referred to will be clear from the context).
\end{definition}

We now show the following lower bound on the energy, from which the $\Gamma-\liminf$ inequality for the dimension-reduced situation will automatically follow.

\begin{proposition}[Lower bound on reduced phase field functional]\label{thm:lowerBoundReduced}
Let $I=(a,b)\subset\R$ and $0\leq\delta\leq\eta\leq1$.
Let $I_\eta\subset\{x\in I\,|\,\varphi_1(x),\ldots,\varphi_N(x)\geq\eta\}$, and denote the collection of connected components of $I\setminus I_\eta$ by $\mathcal C_\eta$.
Furthermore define the subcollection of connected components $\mathcal C_\eta^\geq=\{C\in\mathcal C_\eta\,|\,\inf_C\varphi_1,\ldots,\inf_C\varphi_N\geq\delta\}$ and $C^\geq=\bigcup_{C\in\mathcal C_\eta^\geq}C$.
Finally assume $\varphi_i(a),\varphi_i(b)\geq\eta$ for all $i=1,\ldots,N$.
\begin{enumerate}
\item\label{enm:lbFiniteA0}
If $\alpha_0<\infty$ we have
\begin{multline*}
\F_\varepsilon[\vartheta,\varphi_1,\ldots,\varphi_N;I]
\geq(\eta-\delta)^2\int_{I_\eta\cup C^\geq}\alpha_0|\vartheta|\,\d x\,\\
+\,(\eta-\delta)^2\sum_{C\in\mathcal C_\eta\setminus\mathcal C_\eta^\geq}\max\left\{\beta_1,\tau\left(\int_C|\vartheta|\,\d x\right)\right\}
\,-\,\alpha_0^2\hdone(I)\frac{\varepsilon}{\delta^2}\,.
\end{multline*}
\item\label{enm:lbInfiniteA0}
If $\alpha_0=\infty$ we have
\begin{multline*}
\F_\varepsilon[\vartheta,\varphi_1,\ldots,\varphi_N;I]
\geq\frac{\delta^2}{2\varepsilon\hdone(I)}\left(\int_{I_\eta\cup C^\geq}|\vartheta|\,\d x\right)^2\\
+(\eta-\delta)^2\sum_{C\in\mathcal C_\eta\setminus\mathcal C_\eta^\geq}\max\left\{\beta_1,\tau\left(\int_C|\vartheta|\,\d x\right)\right\}\, -\,\alpha_1^2\hdone(I)\frac{\varepsilon}{\delta^2}\,.
\end{multline*}
\end{enumerate}
\end{proposition}
\begin{proof}
\emph{\ref{enm:lbFiniteA0}.\ ($\alpha_0<\infty$)}
We first show that without loss of generality we may assume
\begin{equation}\label{eqn:largePhase field}
|\vartheta|\geq\frac{\alpha_0\varepsilon}{2\gamma_\varepsilon}\frac1{1-(\eta-\delta)^2}\quad\text{on }K_0^\varepsilon\,.
\end{equation}
The motivation is that there may be regions in which a phase field $\varphi_i$ is (still) small,
but in which we actually have to pay $\alpha_0|\vartheta|$.
Thus, in those regions we would like $\omega_\varepsilon(\alpha_0,\frac{\gamma_\varepsilon(x)}\varepsilon,|\vartheta(x)|)$ to approximate $\alpha_0|\vartheta(x)|$ sufficiently well,
and the above condition on $\vartheta$ ensures
\begin{equation}\label{eqn:largePhase field2}
\omega_\varepsilon\left(\alpha_0,\frac{\gamma_\varepsilon(x)}\varepsilon,|\vartheta(x)|\right)
=\alpha_0|\vartheta(x)|-\frac{\alpha_0^2\varepsilon}{2\gamma_\varepsilon(x)}
\geq(\eta-\delta)^2\alpha_0|\vartheta(x)|
\quad\text{for }x\in K_0^\varepsilon\,.
\end{equation}
We achieve \eqref{eqn:largePhase field} by modifying $\vartheta$ while keeping the cost as well as $\int_{I_\eta}|\vartheta|\,\d x$ and $\int_C|\vartheta|\,\d x$ for all $C\in\mathcal C_\eta$ the same
so that the overall estimate of the proposition is not affected.
The modification mimics the relaxation from \Cref{rem:relaxation}: the modified $\vartheta$ oscillates between small and very large values.
Indeed, for fixed $C\in\mathcal C_\eta\cup\{I_\eta\}$ and $x\in C$ set
\begin{equation*}
\hat\vartheta(x)=\begin{cases}
\max\{\frac{\alpha_0\varepsilon}{2\gamma_\varepsilon(x)}\frac1{1-(\eta-\delta)^2},\vartheta(x)\}
&\text{if }x\in K_0^\varepsilon\cap(-\infty,t_C]\,,\\
\frac{\alpha_0\varepsilon}{\gamma_\varepsilon(x)}
&\text{if }x\in K_0^\varepsilon\cap(t_C,\infty)\,,\\
\vartheta(x)&\text{else,}
\end{cases}
\end{equation*}
where $t_C$ is chosen such that $\int_C|\hat\vartheta|\,\d x=\int_C|\vartheta|\,\d x$
(this is possible, since for $t_C=\infty$ we have $|\hat\vartheta|\geq|\vartheta|$ and for $t_C=-\infty$ we have $|\hat\vartheta|\leq|\vartheta|$ everywhere on $C$).
The cost did not change by this modification since
\begin{align*}
&\F_\varepsilon[\hat\vartheta,\varphi_1,\ldots,\varphi_N;I]
-\F_\varepsilon[\vartheta,\varphi_1,\ldots,\varphi_N;I]\\
&=\int_{K_0^\varepsilon}\omega_\varepsilon\left(\alpha_0,\frac{\gamma_\varepsilon(x)}\varepsilon,|\hat\vartheta(x)|\right)-\omega_\varepsilon\left(\alpha_0,\frac{\gamma_\varepsilon(x)}\varepsilon,|\vartheta(x)|\right)\,\d x\\
&=\int_{K_0^\varepsilon}\left(\alpha_0|\hat\vartheta(x)|-\frac{\alpha_0^2\varepsilon}{2\gamma_\varepsilon(x)}\right)-\left(\alpha_0|\vartheta(x)|-\frac{\alpha_0^2\varepsilon}{2\gamma_\varepsilon(x)}\right)\,\d x\\
&=\alpha_0\int_{K_0^\varepsilon}|\hat\vartheta(x)|\,\d x-\alpha_0\int_{K_0^\varepsilon}|\vartheta(x)|\,\d x\\
&=0\,.
\end{align*}
Note that the modification $\hat\vartheta$ has a different set $K_0^\varepsilon(\hat\vartheta,\varphi_1,\ldots,\varphi_N;I)$ than the original $\vartheta$.
Indeed, by definition of $\hat\vartheta$ we have $K_0^\varepsilon(\hat\vartheta,\varphi_1,\ldots,\varphi_N;I)\subset K_0^\varepsilon(\vartheta,\varphi_1,\ldots,\varphi_N;I)$
and $|\vartheta(x)|\geq\frac{\alpha_0\varepsilon}{2\gamma_\varepsilon(x)}\frac1{1-(\eta-\delta)^2}$ on $K_0^\varepsilon(\hat\vartheta,\varphi_1,\ldots,\varphi_N;I)$, as desired.

Let us now abbreviate $m_0=\int_{(I_\eta\cup C^\geq)\setminus K_0^\varepsilon}|\vartheta|\,\d x$.
Using the definition of $\omega_\varepsilon$ as well as $\gamma_\varepsilon\geq\delta^2$ on $I_\eta\cup C^\geq$ we compute
\begin{multline*}
\F_\varepsilon[\vartheta,\varphi_1,\ldots,\varphi_N;I_\eta\cup C^\geq]
\geq\int_{(I_\eta\cup C^\geq)\cap K_0^\varepsilon}\omega_\varepsilon\left(\alpha_0,\frac{\gamma_\varepsilon}\varepsilon,|\vartheta|\right)\,\d x
+\frac{\delta^2}{2\varepsilon}\int_{(I_\eta\cup C^\geq)\setminus K_0^\varepsilon}|\vartheta|^2\,\d x\\
\geq(\eta-\delta)^2\int_{I_\eta\cup C^\geq}\alpha_0|\vartheta|\,\d x
-\alpha_0m_0
+\frac{\delta^2}{2\varepsilon\hdone((I_\eta\cup C^\geq)\setminus K_0^\varepsilon)}m_0^2\,,
\end{multline*}
where we have employed \eqref{eqn:largePhase field2} and Jensen's inequality.
Upon minimizing in $m_0$, which yields the optimal value $\alpha_0\varepsilon\hdone((I_\eta\cup C^\geq)\setminus K_0^\varepsilon)/\delta^2$ for $m_0$, we thus obtain
\begin{equation*}
\F_\varepsilon[\vartheta,\varphi_1,\ldots,\varphi_N;I_\eta\cup C^\geq]
\geq(\eta-\delta)^2\int_{I_\eta\cup C^\geq}\alpha_0|\vartheta|\,\d x-\frac{\alpha_0^2\varepsilon\hdone(I_\eta\cup C^\geq)}{2\delta^2}\,.
\end{equation*}

Next consider for each $C\in\mathcal C_\eta\setminus\mathcal C_\eta^\geq$ and $i=1,\ldots,N$ the subsets
\begin{align*}
C_i^\geq&=C\cap K_i^\varepsilon\cap\{\varphi_i\geq\delta\}\,,&
C_i^<&=C\cap K_i^\varepsilon\cap\{\varphi_i<\delta\}\,,
\end{align*}
and abbreviate $m_A=\int_A|\vartheta|\,\d x$ for any $A\subset I$.
Using Young's inequality, for $i,j=1,\ldots,N$ we have
\begin{multline*}
\int_{C_i^\geq}\frac{\varphi_i^2+\alpha_i^2\varepsilon^2/\beta_i}{2\varepsilon}|\vartheta|^2\,\d x
\geq\int_{C_i^\geq}\frac{\delta^2+\alpha_i^2\varepsilon^2/\beta_i}{2\varepsilon}|\vartheta|^2\,\d x\\
\geq\int_{C_i^\geq}\alpha_j|\vartheta|-\frac{\alpha_j^2\varepsilon/2}{\delta^2+\alpha_i^2\varepsilon^2/\beta_i}\,\d x
\geq\alpha_jm_{C_i^\geq}-\frac{\alpha_j^2\varepsilon}{2\delta^2}\hdone(C_i^\geq)\,.
\end{multline*}
Similarly, using Jensen's inequality we have
\begin{multline*}
\int_{C_i^<}\frac{\varphi_i^2+\alpha_i^2\varepsilon^2/\beta_i}{2\varepsilon}|\vartheta|^2+\frac{\beta_i}{2\varepsilon}(\varphi_i-1)^2\,\d x
\geq\int_{C_i^<}\frac{\alpha_i^2\varepsilon}{2\beta_i}|\vartheta|^2+\frac{\beta_i}{2\varepsilon}(\delta-1)^2\,\d x\\
\geq\frac{\alpha_i^2\varepsilon}{2\beta_i}\frac1{\hdone(C_i^<)}\left(\int_{C_i^<}|\vartheta|\,\d x\right)^2
+\frac{\beta_i}{2\varepsilon}(1-\delta)^2\hdone(C_i^<)
\geq\alpha_i(1-\delta)m_{C_i^<}\,,
\end{multline*}
where in the last step we optimized for $\hdone(C_i^<)$.
Finally, if $\inf_C\varphi_i\leq\delta$ we have (using Young's inequality)
\begin{multline*}
\int_{C\setminus C_i^<}\frac{\beta_i}2\left(\varepsilon|\varphi_i'|^2+\frac{(\varphi_i-1)^2}\varepsilon\right)\,\d x
\geq\beta_i\int_{C\setminus C_i^<}|\varphi_i'|\,|1-\varphi_i|\,\d x\\
\geq\beta_i\left(\int_c^d|\varphi_i'|\,|1-\varphi_i|\,\d x+\int_e^f|\varphi_i'|\,|1-\varphi_i|\,\d x\right)
\geq\beta_i\left(\int_\eta^\delta\varphi_i-1\,\d\varphi_i+\int_\delta^\eta1-\varphi_i\,\d\varphi_i\right)\\
=\beta_i((1-\delta)^2-(1-\eta)^2)
\geq\beta_i(\eta-\delta)^2\,,
\end{multline*}
where $(c,d)$ and $(e,f)$ denote the first and the last connected component of $C\setminus C_i^<$.

Next, for $C\in\mathcal C_\eta\setminus\mathcal C_\eta^\geq$ define $j(C)=\max\{j\in\{1,\ldots,N\}\,|\,\inf_C\varphi_j<\delta\}$. Summarizing the previous estimates we obtain
\begin{align*}
&\F_\varepsilon[\vartheta,\varphi_1,\ldots,\varphi_N;I\setminus(I_\eta\cup C^\geq)]\\
&\geq\sum_{C\in\mathcal C_\eta\setminus\mathcal C_\eta^\geq}\Bigg(\int_{C\cap K_0^\varepsilon}\omega_\varepsilon\left(\alpha_0,\frac{\gamma_\varepsilon}\varepsilon,|\vartheta|\right)\,\d x+\sum_{i=1}^N\bigg(
\int_{C_i^\geq}\frac{\varphi_i^2+\alpha_i^2\varepsilon^2/\beta_i}{2\varepsilon}|\vartheta|^2\,\d x\\
&\quad+\int_{C_i^<}\frac{\varphi_i^2+\alpha_i^2\varepsilon^2/\beta_i}{2\varepsilon}|\vartheta|^2+\frac{\beta_i}{2\varepsilon}(\varphi_i-1)^2\,\d x
+\int_{C\setminus C_i^<}\frac{\beta_i}2\left(\varepsilon|\varphi_i'|^2+\frac{(\varphi_i-1)^2}\varepsilon\right)\,\d x\bigg)\Bigg)\\
&\geq\sum_{C\in\mathcal C_\eta\setminus\mathcal C_\eta^\geq}\left((\eta-\delta)^2\alpha_0m_{C\cap K_0^\varepsilon}+\sum_{i=1}^N\left(
\alpha_{j(C)}m_{C_i^\geq}-\frac{\alpha_{j(C)}^2\varepsilon}{2\delta^2}\hdone(C_i^\geq)
+\alpha_i(1-\delta)m_{C_i^<}\right)
+\sum_{\substack{i=1\\\inf_C\varphi_i\leq\delta}}^N\beta_i(\eta-\delta)^2\right)\\
&\geq\sum_{C\in\mathcal C_\eta\setminus\mathcal C_\eta^\geq}\left((\eta-\delta)^2(\alpha_{j(C)}m_{C}+\beta_{j(C)})-\frac{\alpha_{j(C)}^2\varepsilon}{2\delta^2}\hdone(C)\right)\\
&\geq(\eta-\delta)^2\sum_{C\in\mathcal C_\eta\setminus\mathcal C_\eta^\geq}\max\left\{\beta_1,\tau\left(\int_C|\vartheta|\,\d x\right)\right\}-\frac{\alpha_1^2\varepsilon}{2\delta^2}\hdone(I)\,.
\end{align*}

Finally, we obtain the desired estimate,
\begin{multline*}
\F_\varepsilon[\vartheta,\varphi_1,\ldots,\varphi_N;I]
=\F_\varepsilon[\vartheta,\varphi_1,\ldots,\varphi_N;I_\eta\cup C^\geq]+\F_\varepsilon[\vartheta,\varphi_1,\ldots,\varphi_N;I\setminus(I_\eta\cup C^\geq)]\\
\geq(\eta-\delta)^2\int_{I_\eta\cup C^\geq}\alpha_0|\vartheta|\,\d x-\frac{\alpha_0^2\varepsilon\hdone(I)}{2\delta^2}
+(\eta-\delta)^2\sum_{C\in\mathcal C_\eta\setminus\mathcal C_\eta^\geq}\max\left\{\beta_1,\tau\left(\int_C|\vartheta|\,\d x\right)\right\}-\frac{\alpha_1^2\varepsilon}{2\delta^2}\hdone(I)\,.
\end{multline*}

\medskip
\emph{\ref{enm:lbInfiniteA0}.\ ($\alpha_0=\infty $)} In this case the set $K_0^\varepsilon$ is empty, and the cost functional reduces to
\begin{equation*}
\F_\varepsilon[\vartheta,\varphi_1,\ldots,\varphi_N;I]
=\int_{I}\frac{\gamma_\varepsilon(x)|\vartheta(x)|^2}{2\varepsilon}
+\sum_{i=1}^N\frac{\beta_i}2\left[\varepsilon|\varphi_i'(x)|^2+\frac{(\varphi_i(x)-1)^2}\varepsilon\right]\,\d x.
\end{equation*}
With Jensen's inequality we thus compute
\begin{multline*}
\F_\varepsilon[\vartheta,\varphi_1,\ldots,\varphi_N;I_\eta\cup C^\geq]
\geq\int_{I_\eta\cup C^\geq}\frac{\gamma_\varepsilon}{2\varepsilon}|\vartheta|^2\,\d x
\geq \frac{\delta^2}{2\varepsilon}\int_{I_\eta\cup C^\geq}|\vartheta|^2\,\d x\\
\geq \frac{\delta^2}{2\varepsilon}\frac{\left(\int_{I_\eta\cup C^\geq}|\vartheta|\,\d x\right)^2}{\hdone(I_\eta\cup C^\geq)}
\geq \frac{\delta^2}{2\varepsilon}\frac{\left(\int_{I_\eta\cup C^\geq}|\vartheta|\,\d x\right)^2}{\hdone(I)}\,.
\end{multline*}
Furthermore, the same calculation as in the case $\alpha_0<\infty$ yields
\begin{align*}
\F_\varepsilon[\tilde\vartheta,\varphi_1,\ldots,\varphi_N;I\setminus(I_\eta\cup C^\geq)]
\geq(\eta-\delta)^2\sum_{C\in\mathcal C_\eta\setminus\mathcal C_\eta^\geq}\max\left\{\beta_1,\tau\left(\int_C|\vartheta|\,\d x\right)\right\}-\frac{\alpha_1^2\varepsilon}{2\delta^2}\hdone(I)
\end{align*}
so that we obtain the desired estimate
\begin{multline*}
\F_\varepsilon[\vartheta,\varphi_1,\ldots,\varphi_N;I]
=\F_\varepsilon[\vartheta,\varphi_1,\ldots,\varphi_N;I_\eta\cup C^\geq]+\F_\varepsilon[\vartheta,\varphi_1,\ldots,\varphi_N;I\setminus(I_\eta\cup C^\geq)]\\
\geq\frac{\delta^2}{2\varepsilon}\frac{\left(\int_{I_\eta\cup C^\geq}|\vartheta|\,\d x\right)^2}{\hdone(I)}+(\eta-\delta)^2\sum_{C\in\mathcal C_\eta\setminus\mathcal C_\eta^\geq}\max\left\{\beta_1,\tau\left(\int_C|\vartheta|\,\d x\right)\right\}-\frac{\alpha_1^2\varepsilon}{\delta^2}\hdone(I)\,.
\qedhere
\end{multline*}
\end{proof}

\begin{corollary}[$\Gamma-\liminf$ inequality for reduced functionals]\label{thm:GammaLiminfReduced}
Let $J\subset\R$ be open and bounded, $\vartheta\in\meas(\overline J)$, and $\varphi_1,\ldots,\varphi_N\in L^1(\Omega)$. Then
\begin{equation*}
\Gamma-\liminf_{\varepsilon\to0}\f_\varepsilon[\vartheta,\varphi_1,\ldots,\varphi_N;J]\geq \f[\vartheta,\varphi_1,\ldots,\varphi_N;J]
\end{equation*}
with respect to weak-$*$ convergence in $\meas(\overline J)$ and strong convergence in $L^1(J)^N$.
\end{corollary}
\begin{proof}
Let $(\vartheta^\varepsilon,\varphi_1^\varepsilon,\ldots,\varphi_N^\varepsilon)$ be an arbitrary sequence converging to $(\vartheta,\varphi_1,\ldots,\varphi_N)$ in the considered topology as $\varepsilon\to0$,
and assume without loss of generality that the limit inferior of $\f_\varepsilon[\vartheta^\varepsilon,\varphi_1^\varepsilon,\ldots,\varphi_N^\varepsilon;J]$ is actually a limit and is finite (else there is nothing to show).
Further we may assume $\f_\varepsilon[\vartheta^\varepsilon,\varphi_1^\varepsilon,\ldots,\varphi_N^\varepsilon;J]=\F_\varepsilon[\vartheta^\varepsilon,\varphi_1^\varepsilon,\ldots,\varphi_N^\varepsilon;J]$ along the sequence.

It suffices to show the $\liminf$-inequality for each connected component $\tilde I=(a,b)$ of $J$ separately.
Due to $\F_\varepsilon[\vartheta^\varepsilon,\varphi_1^\varepsilon,\ldots,\varphi_N^\varepsilon;\tilde I]\geq\frac{\beta_i}{2\varepsilon}\|\varphi_i^\varepsilon-1\|_{L^2}^2$ for $i=1,\ldots,N$
we must have $\varphi_i^\varepsilon\to1$ in $L^2(\tilde I)$ and thus also in $L^1(\tilde I)$ so that $\varphi_i=1$.
Even more, after passing to another subsequence, by Egorov's theorem all $\varphi_i^\varepsilon$ converge uniformly to $1$ outside a set of arbitrarily small measure.
In particular, for any $\xi>0$ we can find an open interval $(a+\xi,b-\xi)\subset I\subset\tilde I$ such that $\varphi_i^\varepsilon\to1$ uniformly on $\partial I$,
and for any $\eta<1$ there is an open set $I_\eta\subset I$ with $\hdone(I\setminus I_\eta)\leq1-\eta$ such that $\varphi_i^\varepsilon\geq\eta$ on $I_\eta\cup\partial I$ for all $i=1,\ldots,N$ and all $\varepsilon$ small enough.

We now choose $\delta=\varepsilon^{1/3}$ and $\eta=1-\varepsilon$ and denote by $\mathcal C_\eta(\varepsilon)$ and $\mathcal C_\eta^\geq(\varepsilon)$ the collections of connected components of $I\setminus I_\eta$ from \Cref{thm:lowerBoundReduced} (which now depend on $\varepsilon$).
Further we abbreviate $\mathcal C_\eta^<(\varepsilon)=\mathcal C_\eta(\varepsilon)\setminus\mathcal C_\eta^\geq(\varepsilon)$.
The bound of \Cref{thm:lowerBoundReduced} implies that the number of elements in $\mathcal C_\eta^<(\varepsilon)$ is bounded uniformly in $\varepsilon$ and $\eta$.
Passing to another subsequence we may assume $\mathcal C_\eta^<(\varepsilon)$ to contain $K$ sets $C_1(\varepsilon),\ldots,C_K(\varepsilon)$ whose midpoints converge to $x_1,\ldots,x_K\in\overline I$, respectively.
Thus for an arbitrary $\zeta>0$ we have that for all $\varepsilon$ small enough each $C\in\mathcal C_\eta^<(\varepsilon)$ lies inside the closed $\zeta$-neighbourhood $B_\zeta(\{x_1,\ldots,x_K\})$ of $\{x_1,\ldots,x_K\}$.

Now for $\alpha_0<\infty$ we obtain from \Cref{thm:lowerBoundReduced}
\begin{align*}
&\liminf_{\varepsilon\to0}\F_\varepsilon[\vartheta^\varepsilon,\varphi_1^\varepsilon,\ldots,\varphi_N^\varepsilon;I]\\
&\geq\liminf_{\varepsilon\to0}(\eta-\delta)^2\int_{I_\eta\cup C^\geq(\varepsilon)}\alpha_0|\vartheta^\varepsilon|\,\d x\,+\,(\eta-\delta)^2\sum_{i=1}^K\max\left\{\beta_1,\tau\left(\int_{C_i(\varepsilon)}|\vartheta^\varepsilon|\,\d x\right)\right\}\\
&\geq\liminf_{\varepsilon\to0}(\eta-\delta)^2\int_{I\setminus B_\zeta(\{x_1,\ldots,x_K\})}\alpha_0|\vartheta^\varepsilon|\,\d x\,+\,(\eta-\delta)^2\sum_{i=1}^K\tau\left(\int_{B_\zeta(\{x_i\})}|\vartheta^\varepsilon|\,\d x\right)\\
&\geq\alpha_0|\vartheta|(I\setminus B_\zeta(\{x_1,\ldots,x_K\}))+\sum_{i=1}^K\tau\left(|\vartheta|(B_\zeta(\{x_i\}))\right)\,,
\end{align*}
where in the second inequality we used
\begin{multline*}
\int_{A\cup B}\alpha_0|\vartheta|\,\d x+\tau\left(\int_C|\vartheta|\,\d x\right)
\geq\int_{A}\alpha_0|\vartheta|\,\d x+\tau\left(\int_{B}|\vartheta|\,\d x\right)+\tau\left(\int_{C}|\vartheta|\,\d x\right)\\
\geq\int_{A}\alpha_0|\vartheta|\,\d x+\tau\left(\int_{B\cup C}|\vartheta|\,\d x\right)
\end{multline*}
for all measurable $A,B,C\subset I$ (due to the subadditivity of $\tau$) and in the third inequality we used $\tau(m)\leq\alpha_0m$ as well as the lower semi-continuity of the mass on an open set.
Letting now $\zeta\to0$ (so that by the $\sigma$-continuity of $\vartheta$ we have $|\vartheta|(I\setminus B_\zeta(\{x_1,\ldots,x_K\}))\to|\vartheta|(I\setminus\{x_1,\ldots,x_K\})$) we obtain
\begin{multline*}
\liminf_{\varepsilon\to0}\F_\varepsilon[\vartheta^\varepsilon,\varphi_1^\varepsilon,\ldots,\varphi_N^\varepsilon;I]
\geq\alpha_0|\vartheta|(I\setminus\{x_1,\ldots,x_K\})+\sum_{i=1}^K\tau(|m_\vartheta(x_i)|)\\
\geq\alpha_0|\vartheta^\perp|(I)+\int_{S_\vartheta\cap I}\tau(|m_\vartheta|)\,\d\hd^0
=\F[\vartheta;I]\,.
\end{multline*}

If on the other hand $\alpha_0=\infty$ we obtain from \Cref{thm:lowerBoundReduced}
\begin{multline*}
\liminf_{\varepsilon\to0}\F_\varepsilon[\vartheta^\varepsilon,\varphi_1^\varepsilon,\ldots,\varphi_N^\varepsilon;I]
\geq\liminf_{\varepsilon\to0}\frac{\delta^2}{2\varepsilon\hdone(I)}\left(\int_{I_\eta\cup C^\geq(\varepsilon)}|\vartheta^\varepsilon|\,\d x\right)^2\\
\geq\liminf_{\varepsilon\to0}\frac{\delta^2}{2\varepsilon\hdone(I)}\left(\int_{I\setminus B_\zeta(\{x_1,\ldots,x_K\})}|\vartheta^\varepsilon|\,\d x\right)^2
\end{multline*}
which implies $|\vartheta|(I\setminus\{x_1,\ldots,x_K\})=0$ and thus $|\vartheta^\perp|(\overline I)=0$ as well as $S_\vartheta\cap\overline I\subset\{x_1,\ldots,x_K\}$.
Next note that for all $i\in\{1,\ldots,K\}$ with $|m_\vartheta(x_i)|>0$ we have $\int_{C_i(\varepsilon)}|\vartheta^\varepsilon|\,\d x>0$ for all $\varepsilon$ small enough.
Indeed,
\begin{equation*}
|m_\vartheta(x_i)|
\leq\liminf_{\varepsilon\to0}\int_{B_\zeta(\{x_i\})}|\vartheta^\varepsilon|\,\d x
=\liminf_{\varepsilon\to0}\int_{C_i(\varepsilon)}|\vartheta^\varepsilon|\,\d x+\int_{B_\zeta(\{x_i\})\setminus C_i(\varepsilon)}|\vartheta^\varepsilon|\,\d x\,,
\end{equation*}
where $\int_{B_\zeta(\{x_i\})\setminus C_i(\varepsilon)}|\vartheta^\varepsilon|\,\d x$ decreases to zero by \Cref{thm:lowerBoundReduced}.
Therefore, \Cref{thm:lowerBoundReduced} implies
\begin{align*}
&\liminf_{\varepsilon\to0}\F_\varepsilon[\vartheta^\varepsilon,\varphi_1^\varepsilon,\ldots,\varphi_N^\varepsilon;I]\\
&\geq\liminf_{\varepsilon\to0}\frac{\delta^2}{2\varepsilon\hdone(I)}\left(\int_{I_\eta\cup C^\geq(\varepsilon)}|\vartheta^\varepsilon|\,\d x\right)^2+(\eta-\delta)^2\sum_{\substack{i=1\\|m_\vartheta(x_i)|>0}}^K\tau\left(\int_{C_i(\varepsilon)}|\vartheta^\varepsilon|\,\d x\right)\\
&\geq\liminf_{\varepsilon\to0}\frac{\delta^2}{2\varepsilon\hdone(I)}\sum_{\substack{i=1\\|m_\vartheta(x_i)|>0}}^K\left(\int_{B_\zeta(\{x_i\})\setminus C_i(\varepsilon)}|\vartheta^\varepsilon|\,\d x\right)^2\\
&\qquad\qquad\qquad+\sum_{\substack{i=1\\|m_\vartheta(x_i)|>0}}^K\left(\tau\left(\int_{B_\zeta(\{x_i\})}|\vartheta^\varepsilon|\,\d x\right)-\alpha_1\int_{B_\zeta(\{x_i\})\setminus C_i(\varepsilon)}|\vartheta^\varepsilon|\,\d x\right)\\
&\geq\liminf_{\varepsilon\to0}\sum_{\substack{i=1\\|m_\vartheta(x_i)|>0}}^K\left(\tau\left(\int_{B_\zeta(\{x_i\})}|\vartheta^\varepsilon|\,\d x\right)-\frac{\alpha_1^2\varepsilon\hdone(I)}{2\delta^2}\right)\\
&\geq\sum_{\substack{i=1\\|m_\vartheta(x_i)|>0}}^K\tau\left(|\vartheta|(B_\zeta(\{x_i\}))\right)\\
&=\int_{S_\vartheta}\tau(|m_\vartheta|)\,\d\hd^0
=\F[\vartheta;I]\,,
\end{align*}
where in the second inequality we used $\tau(m_1+m_2)\leq\tau(m_1)+\alpha_1m_2$ for any $m_1>0$, $m_2\geq0$ and in the third we optimized in $\int_{B_\zeta(\{x_i\})\setminus C_i(\varepsilon)}|\vartheta^\varepsilon|\,\d x$.

The proof is concluded by letting $\xi\to0$ and noting $\liminf_{\xi\to0}\F[\vartheta;I]\geq\F[\vartheta;\tilde I]$.
\end{proof}

\subsection{Slicing of vector-valued measures}
We derive now some technical construction for divergence measure vector fields which are needed to reduce the $\Gamma-\liminf$ inequality to the lower-dimensional setting of the previous section.
In particular, we will introduce slices of a divergence measure vector field, which in the language of geometric measure theory correspond to slices of currents. We will slice in the direction of a unitary vector  $\xi\in \Disk^{n-1}$ with orthogonal hyperplanes of the form
\begin{equation*}
H_{\xi,t}=\pi_{\xi}^{-1}(t)
\qquad\text{for the projection }
\pi_\xi:\R^{n}\to \R,\;\pi_\xi(x)=x\cdot \xi\,.
\end{equation*}
The orthogonal projection onto $H_{\xi,t}$ is denoted
\begin{equation*}
\pi_{H_{\xi,t}}(x)=(I-\xi\otimes\xi)x+t\xi\,.
\end{equation*}
The slicing will essentially be performed via disintegration.
Let $\sigma$ be a compactly supported divergence measure vector field.
By the Disintegration Theorem~\cite[Thm.~2.28]{AmFuPa00}, for all $\xi\in \Disk^{n-1}$ and almost all $t\in\R$ there exists a unique measure $\nu_{\xi,t}\in\meas(H_{\xi,t})$ such that 
\begin{equation*}
\|\nu_{\xi,t}\|_\meas =1\quad \text{ and }\quad \sigma\cdot\xi= \nu_{\xi,t}\otimes \pushforward{\pi_\xi} |\sigma\cdot\xi|(t)\,.
\end{equation*}
We decompose $\pushforward{\pi_\xi} |\sigma\cdot\xi|$ into its absolutely continuous and singular part according to
\begin{equation*}
\pushforward{\pi_\xi} |\sigma\cdot\xi|=\sigma_\xi(t)\,\d t+\sigma_\xi^\perp
\end{equation*}
for $\d t$ the Lebesgue measure on $\R$.
\begin{lemma}\label{lemma:L1density}
For any $\xi\in\Disk^{n-1}$ and any compactly supported divergence measure vector field $\sigma$ we have $\sigma_\xi^\perp=0$, that is,
the measure $\pushforward{\pi_\xi}{|\sigma \cdot\xi|}=\sigma_\xi(t)\,\d t$ is absolutely continuous with respect to the Lebesgue measure on $\R$.
Moreover, for almost all $t\in\R$ and any compactly supported $\theta\in\cont^\infty(\R^n)$ we have
\begin{equation}\label{eqn:slicingCharacterization}
\sigma_\xi(t)\int_{H_{\xi,t}}\theta\,\d\nu_{\xi,t}
=\int_{\{\xi\cdot x<t\}}\nabla\theta\cdot\d\sigma+\int_{\{\xi\cdot x<t\}}\theta\,\d\dive\sigma\,.
\end{equation}
\end{lemma}
\begin{proof}
Abbreviate $H=\xi^\perp=H_{\xi,0}$ with corresponding orthogonal projection $\pi_H$, let $\phi\in\cont^\infty(H)$ and $\psi\in\cont^\infty(\R)$ be compactly supported, and define
\begin{equation*}
I(\phi,\psi)=\int_{\R^n}\phi(\pi_H(x))\psi(\pi_\xi(x))\,\d(\sigma\cdot\xi)(x)\,.
\end{equation*}
Introducing $\Psi(t)=\int_{-\infty}^t\psi(s)\,\d s$ we obtain via the chain and product rule
\begin{equation*}
(\phi\circ\pi_H)(\psi\circ\pi_\xi)\xi
=(\phi\circ\pi_H)\nabla[\Psi\circ\pi_\xi]
=\nabla[(\phi\circ\pi_H)(\Psi\circ\pi_\xi)]-\nabla[\phi\circ\pi_H](\Psi\circ\pi_\xi)
\end{equation*}
so that (denoting by $\chi_A$ the characteristic function of a set $A$)
\begin{align*}
I(\phi,\psi)
&=\int_{\R^n}\nabla[(\phi\circ\pi_H)(\Psi\circ\pi_\xi)]\cdot\d\sigma-\int_{\R^n}\nabla[\phi\circ\pi_H](\Psi\circ\pi_\xi)\cdot\d\sigma\\
&=-\int_{\R^n}(\phi\circ\pi_H)(\Psi\circ\pi_\xi)\,\d\dive\sigma-\int_{\R^n}\nabla[\phi\circ\pi_H](\Psi\circ\pi_\xi)\cdot\d\sigma\\
&=-\int_{\R^n}\phi(\pi_H(x))\left(\int_{-\infty}^{\pi_\xi(x)}\psi(s)\,\d s\right)\,\d\dive\sigma(x)\\
&\qquad-\int_{\R^n}\nabla[\phi\circ\pi_H](x)\left(\int_{-\infty}^{\pi_\xi(x)}\psi(s)\,\d s\right)\cdot\d\sigma\\
&=-\int_{\R^n}\phi(\pi_H(x))\left(\int_\R\chi_{\{\xi\cdot x\geq s\}}\psi(s)\,\d s\right)\,\d\dive\sigma(x)\\
&\qquad-\int_{\R^n}\nabla[\phi\circ\pi_H](x)\left(\int_\R\chi_{\{\xi\cdot x\geq s\}}\psi(s)\,\d s\right)\cdot\d\sigma\,.
\end{align*}
(Note that we could just as well have used $\chi_{\{\xi\cdot x>s\}}$ instead of $\chi_{\{\xi\cdot x\geq s\}}$,
which would ultimately lead to integration domains $\{\xi\cdot x\leq t\}$ in \eqref{eqn:slicingCharacterization};
for almost all $t$ this will be the same.)
Applying the Fubini--Tonelli Theorem we obtain
\begin{align*}
I(\phi,\psi)
&=-\int_\R\psi(s)\left[\int_{\{\xi\cdot x\geq s\}}\nabla[\phi\circ\pi_H](x)\cdot\d\sigma(x)
+\int_{\{\xi\cdot x\geq s\}}\phi(\pi_H(x))\,\d\dive\sigma(x)\right]\,\d s\\
&=\int_\R\psi(s)\left[\int_{\{\xi\cdot x<s\}}\nabla[\phi\circ\pi_H](x)\cdot\d\sigma(x)
+\int_{\{\xi\cdot x<s\}}\phi(\pi_H(x))\,\d\dive\sigma(x)\right]\,\d s\,,
\end{align*}
where in the second step we just added $0=\int_{\R^n}\nabla[\phi\circ\pi_H]\cdot\d\sigma+\int_{\R^n}\phi\circ\pi_H\,\d\dive\sigma$ in the square brackets.
On the other hand, using the disintegration of $\sigma\cdot\xi$ we also have
\begin{align*}
I(\phi,\psi)
&=\int_{\R^n}\phi(\pi_H(x))\psi(\pi_\xi(x))\,\d(\sigma\cdot\xi)(x)\\
&=\int_\R\left[\psi(s)\int_{H_{\xi,s}}\phi(\pi_H(y))\,\d\nu_{\xi,s}(y)\right](\sigma_\xi(s)\,\d s+\d\sigma_\xi^\perp(s))\,.
\end{align*}
Comparing both expressions for $I(\phi,\psi)$ we can identify
\begin{multline*}
\left[\int_{H_{\xi,s}}\phi(\pi_H(y))\,\d\nu_{\xi,s}(y)\right](\sigma_\xi(s)\,\d s+\d\sigma_\xi^\perp(s))\\
=\left[\int_{\{\xi\cdot x<s\}}\nabla[\phi\circ\pi_H](x)\cdot\d\sigma(x)
+\int_{\{\xi\cdot x<s\}}\phi(\pi_H(x))\,\d\dive\sigma(x)\right]\,\d s\,.
\end{multline*}
Since the right-hand side has no singular component with respect to the Lebesgue measure, we deduce $\left[\int_{H_{\xi,s}}\phi(\pi_H(y))\,\d\nu_{\xi,s}(y)\right]\sigma_\xi^\perp(s)=0$.
Now note that any compactly supported function in $\cont^0(\R^n)$ or $\cont^1(\R^n)$ can be arbitrarily well approximated (in the respective norm)
by finite linear combinations of tensor products $(\phi\circ\pi_H)(\psi\circ\pi_\xi)$ with $\phi\in\cont^\infty(H)$ and $\psi\in\cont^\infty(\R)$ with compact support.
Thus, the above implies $\int_\R\left[\int_{H_{\xi,s}}\theta(y)\,\d\nu_{\xi,s}(y)\right]\d\sigma_\xi^\perp(s)=0$ for any compactly supported $\theta\in\cont^0(\R^n)$ so that
\begin{equation*}
\nu_{\xi,s}\otimes\sigma_\xi^\perp(s)=0
\qquad\text{and thus}\qquad
\sigma_\xi^\perp(s)=0\,.
\end{equation*}
Summarizing, we have $\sigma\cdot\xi=\sigma_\xi(s)\nu_{\xi,s}\otimes\d s$ and
\begin{multline*}
\int_{H_{\xi,s}}\phi(\pi_H(x))\sigma_\xi(s)\,\d\nu_{\xi,s}(x)\,\d s\\
=\int_{\{\xi\cdot x<s\}}\nabla[\phi\circ\pi_H](x)\cdot\d\sigma(x)
+\int_{\{\xi\cdot x<s\}}\phi(\pi_H(x))\d\dive\sigma(x)
\end{multline*}
for all compactly supported $\phi\in\cont^\infty(H)$.
Note that the right-hand side is left-continuous in $s$ so that the left-hand side is as well.
Consequently, $\sigma_\xi(s)\nu_{\xi,s}$ is left-continuous in $s$ with respect to weak-* convergence.
Now let $\chi\in\cont^\infty(\R)$ with $\chi=1$ on $(-\infty,0]$, $\chi=0$ on $[1,\infty)$, and $0\leq\chi\leq1$, and define for $\rho>0$
\begin{equation*}
\chi^\rho(x)=\chi\left(\tfrac{\pi_\xi(x)-t}\rho\right),\quad
\sigma^\rho=\chi^\rho\sigma,\quad
\mu^\rho=\chi^\rho\dive\sigma,\quad
\sigma_{\xi,s}^\rho=\tfrac1\rho\chi'\left(\tfrac{\pi_\xi(\cdot)-t}\rho\right)\sigma_\xi(s)\nu_{\xi,s}.
\end{equation*}
In the distributional sense we have
\begin{equation*}
\dive\sigma^\rho=\mu^\rho+\sigma_{\xi,s}^\rho\otimes\d s
\end{equation*}
so that for any compactly supported $\theta\in\cont^\infty(\R^n)$ we have
\begin{equation*}
\int_{\R^n}\nabla\theta\cdot\d\sigma^\rho+\int_{\R^n}\theta\,\d\mu^\rho=-\int_\R\int_{H_{\xi,s}}\theta\,\d\sigma_{\xi,s}^\rho\,\d s\,.
\end{equation*}
Letting $\rho\to0$ and using the left-continuity of $\sigma_\xi(s)\nu_{\xi,s}$ in $s$ we arrive at \eqref{eqn:slicingCharacterization}.
\end{proof}

We now define the slice of a divergence measure vector field as the measure obtained via disintegration with respect to the one-dimensional Lebesgue measure.

\begin{definition}[Sliced sets, functions, and measures]\label{def:slicing}
Let $\xi\in \Disk^{n-1}$ and $t\in\R$.
\begin{enumerate}
\item
For $A\subset\R^n$ we define the \emph{sliced set} $A_{\xi,t}=A\cap H_{\xi,t}$.
\item
For $f:A\to\R$ we define the \emph{sliced function} $f_{\xi,t}:A_{\xi,t}\to\R$, $f_{\xi,t}=f|_{A_{\xi,t}}$.
For $f:A\to\R^n$ we define $f_{\xi,t}:A_{\xi,t}\to\R^n$, $f_{\xi,t}=\xi\cdot f|_{A_{\xi,t}}$.
\item
We define the \emph{sliced measure} of a compactly supported divergence measure vector field $\sigma$ as
\begin{equation*}
\sigma_{\xi,t}=\sigma_\xi(t)\;\nu_{\xi,t}\,.
\end{equation*}
By \Cref{lemma:L1density} it holds $\sigma\cdot\xi=\sigma_{\xi,t}\otimes\d t$.
\end{enumerate}
\end{definition}

\begin{remark}[Properties of sliced functions and measures]\label{rem:slicedFunction}
\begin{enumerate}
\item
By Fubini's theorem it follows that for any function $f$ of Sobolev-type $W^{m,p}$
the corresponding sliced function $f_{\xi,t}$ is well-defined and also of Sobolev-type $W^{m,p}$ for almost all $\xi\in \Disk^{n-1}$ and $t\in\R$.
For the same reason, strong convergence $f_j\to_{j\to\infty} f$ in $W^{m,p}$ implies strong convergence $(f_j)_{\xi,t}\to f_{\xi,t}$ in $W^{m,p}$ on the sliced domain.
\item
The definitions of sliced functions and measures are consistent in the following sense.
If we identify a Lebesgue function $f$ with the measure $\chi=f\mathcal L$ for $\mathcal L$ the Lebesgue measure,
then the same identification holds between $f_{\xi,t}$ and $\chi_{\xi,t}$ for almost all $\xi\in \Disk^{n-1}$ and $t\in\R$.
\item\label{enm:disintegrationProperty}
Let $\sigma$ be a divergence measure vector field, then the properties \cite[Thm.~2.28]{AmFuPa00} of the disintegration $\sigma\cdot\xi=\nu_{\xi,t}\otimes \pushforward{\pi_\xi} |\sigma\cdot\xi|(t)=\nu_{\xi,t}\otimes\sigma_\xi(t)\,\d t=\sigma_{\xi,t}\otimes\d t$ immediately imply the following.
The map $t\mapsto \|\sigma_{\xi,t}\|_{\meas}$ is integrable and satisfies
\begin{equation*}\label{eq:sigmaDis1}
\int_{\R}\|\sigma_{\xi,t}\|_\meas\,\d t = \int_{\R} \sigma_\xi(t)\,\d t =\|\sigma\cdot \xi\|_\meas\,.
\end{equation*}
Furthermore, for any measurable function $f:\R^n\to\R$, absolutely integrable with respect to $|\sigma\cdot \xi|$, it holds
\begin{equation*}\label{eq:sigmaDis2}
\int_{\R^n}f(x)\,\d \sigma\cdot\xi
=\int_\R \int_{H_{\xi,t}} f(x)\,\d \nu_{\xi,t}(x)\,\d\pushforward{\pi_\xi}|\sigma\cdot\xi|(t)
=\int_\R \int_{H_{\xi,t}} f(x)\,\d \sigma_{\xi,t}(x)\,\d t\,.
\end{equation*}
\end{enumerate}
\end{remark}

We briefly relate our definition of sliced measures to other notions of slices from the literature.

\begin{remark}[Notions of slices]\label{rem:notionsSlicing}
\begin{enumerate}
\item
Let $\Lip(A)$ denote the set of bounded Lipschitz functions on $A\subset\R^n$.
An alternative definition of the slice of a divergence measure vector field $\sigma$ was introduced by \v{S}ilhav\'y \cite{Si07} as the linear operator
\begin{equation}\label{eqn:sliceSilhavy}
\sigma_{\xi,t}:\Lip(H_{\xi,t})\to\R,\quad
\sigma_{\xi,t}(\varphi|_{H_{\xi,t}})=\lim_{\delta\searrow0}\frac1\delta\int_{\{x\in\R^n\,|\,t-\delta<x\cdot\xi<t\}}\varphi\xi\cdot\d\sigma
\end{equation}
for all $\varphi\in\Lip(\R^n)$ (the right-hand side is well-defined and only depends on $\varphi|_{H_{\xi,t}}$ \cite[Thm.\,3.5 \& Thm.\,3.6]{Si07}).
This $\sigma_{\xi,t}$ equals the so-called normal trace of $\sigma$ on $H_{\xi,t}$ (see \cite{Si07} for its definition and properties).
In general it is not a measure but continuous on $\Lip(H_{\xi,t})$ in the sense
\begin{equation*}
\sigma_{\xi,t}(\varphi)\leq(\|\sigma\|_\meas+\|\dive\sigma\|_\meas)\|\varphi\|_{W^{1,\infty}}
\quad\text{for all }\varphi\in\Lip(H_{\xi,t})\,.
\end{equation*}
\item
Interpreting a divergence measure vector field as a $1$-current or a flat $1$-chain,
\v{S}ilhav\'y's definition of $\sigma_{\xi,t}$ is identical to the classical slice of $\sigma$ on $H_{\xi,t}$ as for instance defined in \cite{Wh99} or \cite[4.3.1]{Fe69}
(note that \v{S}ilhav\'y's definition corresponds to \cite[(3.8)]{Si07}, whose analogue for currents is \cite[4.3.2(5)]{Fe69}).
\item\label{enm:divCharacterization}
Our notion of a sliced measure from \Cref{def:slicing} is equivalent to both above-mentioned notions.
Indeed,
\eqref{eqn:slicingCharacterization} implies
\begin{equation*}
\sigma_{\xi,t}=(\dive\sigma)\restr\{x\cdot\xi<t\}-\dive(\sigma\restr\{x\cdot\xi<t\})\,,
\end{equation*}
which shows that the sliced measure represents the normal flux through the hyperplane $H_{\xi,t}=\{x\cdot\xi=t\}$.
This, however, is the same characterization as given in \cite[(3.6)]{Si07} and \cite[4.2.1]{Fe69} for both above notions of slices.
\end{enumerate}
\end{remark}

We conclude the section with several properties needed for the $\Gamma-\liminf$ inequality.
The following result makes use of the Kantorovich--Rubinstein norm (see for instance \cite[eq.\,(2)\,\&\,(5)]{LeloScVa14}; in geometric measure theory it is known as the flat norm) on $\meas(\R^n)$, defined by
\begin{align*}
\|\mu\|_{\mathrm{KR}}
&=\inf\{\|\mu_1\|_\meas+\|\mu_2\|_\meas\,|\,\mu_1\in\meas(\R^n),\,\mu_2\in\meas(\R^n;\R^n),\,\mu=\mu_1+\dive\mu_2\}\\
&=\sup\left\{\int_\Omega f\,\d\mu\,\middle|\,f\text{ Lipschitz with constant }1,\,|f|\leq1\right\}\,.
\end{align*}
For measures of uniformly bounded support and uniformly bounded mass it is known to metrize weak-$*$ convergence (see for instance \cite[Rem.\,2.29(3)-(4)]{BrWi17}).
We will furthermore make use of the following fact.
Let $T_s:x\mapsto x-s\xi$ be the translation by $s$ in direction $-\xi$.
It is straightforward to check that for any divergence measure vector field $\mu\in\meas(\R^n;\R^n)$ we have
\begin{equation*}
\dive(\pushforward{\pi_{H_{\xi,t}}}(\mu-\mu\cdot\xi\,\xi))=\pushforward{\pi_{H_{\xi,t}}}(\dive\mu)\,.
\end{equation*}
As a consequence, for any $\mu\in\meas(H_{\xi,t})$ and $\nu\in\meas(H_{\xi,t+s})$ we have
\begin{equation*}
\|\mu-\nu\|_{\mathrm{KR}}
\geq\|\mu-\pushforward{T_s}\nu\|_{\mathrm{KR}}\,.
\end{equation*}
Indeed, let $\delta>0$ arbitrary and $\mu_1\in\meas(\R^n)$, $\mu_2\in\meas(\R^n;\R^n)$ with $\mu-\nu=\mu_1+\dive\mu_2$ such that $\|\mu-\nu\|_{\mathrm{KR}}\geq\|\mu_1\|_\meas+\|\mu_2\|_\meas-\delta$,
then $\tilde\mu_1=\pushforward{\pi_{H_{\xi,t}}}{\mu_1}$ and $\tilde\mu_2=\pushforward{\pi_{H_{\xi,t}}}(\mu_2-\mu_2\cdot\xi\,\xi)$ satisfy $\mu-\pushforward{T_s}\nu=\tilde\mu_1+\dive\tilde\mu_2$
and thus $$\|\mu-\pushforward{T_s}\nu\|_{\mathrm{KR}}\leq\|\tilde\mu_1\|_\meas+\|\tilde\mu_2\|_\meas\leq\|\mu_1\|_\meas+\|\mu_2\|_\meas\leq\|\mu-\nu\|_{\mathrm{KR}}+\delta\,.$$

\begin{theorem}[Weak convergence of sliced measures]\label{thm:weakConvergenceSlices}
Let $\sigma^j\weakstarto\sigma$ as $j\to\infty$ for a sequence $\{\sigma^j\}$ of compactly supported divergence measure vector fields with uniformly bounded $\|\dive\sigma^j\|_\meas$.
Then for almost all $\xi\in\Disk^{n-1}$ and $t\in\R$ we have $$\sigma_{\xi,t}^j\weakstarto\sigma_{\xi,t}\,.$$
\end{theorem}
\begin{proof}
It suffices to show $\sigma_{\xi,t}^j\weakstarto\sigma_{\xi,t}$ for a subsequence.

Consider the measures $\nu^j=|\sigma^j|+|\dive\sigma^j|$.
Since $\|\nu^j\|_\meas$ is uniformly bounded, a subsequence converges weakly-$*$ to some compactly supported nonnegative $\nu\in\meas(\R^n)$ (the subsequence is still indexed by $j$).
For $I\subset\R$ introduce the notation $H_{\xi,I}=\bigcup_{t\in I}H_{\xi,t}$.
Then for almost all $t\in\R$, $\nu(H_{\xi,[t-s,t+s]})\to0$ as well as $(|\sigma|+|\dive\sigma|)(H_{\xi,[t-s,t+s]})\to0$ as $s\searrow0$.
For such a $t$ we show convergence of $\sigma_{\xi,t}^j-\sigma_{\xi,t}$ to zero in the Kantorovich--Rubinstein norm which implies weak-$*$ convergence.
To this end fix some arbitrary $\delta>0$.
Given $\zeta>0$ let $\rho_\zeta=\rho(\cdot/\zeta)/\zeta$ for a nonnegative smoothing kernel $\rho\in\cont^\infty(\R)$ with support in $[-1,1]$ and $\int_\R\rho\,\d t=1$.
For any compactly supported divergence measure vector field $\lambda$ we now define the convolved slice $\lambda_{\xi,\zeta,t}$ by
\begin{align*}
\int_{H_{\xi,t}}g\,\d\lambda_{\xi,\zeta,t}
&=\int_\R\rho_\zeta(-s)\int_{H_{\xi,t}}g\,\d\pushforward{T_s}{\lambda_{\xi,t+s}}\,\d s\\
&=\int_\R\rho_\zeta(-s)\int_{H_{\xi,t}}g\circ T_s\,\d\lambda_{\xi,t+s}\,\d s
\quad\forall g\in\cont(H_{\xi,t})\,,
\end{align*}
where $T_s:x\mapsto x-s\xi$ is the translation by $s$ in direction $-\xi$.
By \Cref{rem:slicedFunction}\eqref{enm:disintegrationProperty} we have $\sigma^j_{\xi,\zeta,t}\weakstarto\sigma_{\xi,\zeta,t}$.
Furthermore, there exist $\zeta>0$ and $J\in\N$ such that $\|\sigma_{\xi,t}-\sigma_{\xi,\zeta,t}\|_{\mathrm{KR}}\leq\frac\delta3$ and $\|\sigma_{\xi,t}^j-\sigma^j_{\xi,\zeta,t}\|_{\mathrm{KR}}\leq\frac\delta3$ for all $j\geq J$.
Indeed, for a compactly supported divergence measure vector field $\lambda$ we have
\begin{align*}
\|\lambda_{\xi,t}-\lambda_{\xi,\zeta,t}\|_{\mathrm{KR}}
&\leq\int_\R\rho_\zeta(-s)\|\lambda_{\xi,t}-\pushforward{T_s}{\lambda_{\xi,t+s}}\|_{\mathrm{KR}}\,\d s\\
&\leq\int_\R\rho_\zeta(-s)\|\lambda_{\xi,t}-\lambda_{\xi,t+s}\|_{\mathrm{KR}}\,\d s\\
&=\int_\R\rho_\zeta(-s)\|\dive(\lambda\restr H_{\xi,[t,t+s)})-(\dive\lambda)\restr H_{\xi,[t,t+s)}\|_{\mathrm{KR}}\,\d s\\
&\leq\int_\R\rho_\zeta(-s)\left[|\lambda|(H_{\xi,[t,t+s)})+|\dive\lambda|(H_{\xi,[t,t+s)})\right]\,\d s\\
&\leq|\lambda|(H_{\xi,[t-\zeta,t+\zeta]})+|\dive\lambda|(H_{\xi,[t-\zeta,t+\zeta]})\,,
\end{align*}
where in the equality we employed \Cref{rem:notionsSlicing}\eqref{enm:divCharacterization}.
Thus, we can simply pick $\zeta$ such that $|\sigma|(H_{\xi,[t-\zeta,t+\zeta]})+|\dive\sigma|(H_{\xi,[t-\zeta,t+\zeta]})\leq\frac\delta3$ and $\nu(H_{\xi,[t-\zeta,t+\zeta]})\leq\frac\delta6$,
while we choose $J$ such that $(\nu^j-\nu)(H_{\xi,[t-\zeta,t+\zeta]})\leq\frac\delta6$ for all $j>J$.
Now let $\bar J\geq J$ such that $\|\sigma^j_{\xi,\zeta,t}-\sigma_{\xi,\zeta,t}\|_{\mathrm{KR}}\leq\frac\delta3$ for all $j\geq\bar J$, then we obtain
\begin{equation*}
\|\sigma_{\xi,t}^j-\sigma_{\xi,t}\|_{\mathrm{KR}}
\leq\|\sigma_{\xi,t}^j-\sigma^j_{\xi,\zeta,t}\|_{\mathrm{KR}}
+\|\sigma_{\xi,\zeta,t}^j-\sigma_{\xi,\zeta,t}\|_{\mathrm{KR}}
+\|\sigma_{\xi,\zeta,t}-\sigma_{\xi,t}\|_{\mathrm{KR}}
\leq\delta
\end{equation*}
for all $j>\bar J$.
The arbitrariness of $\delta$ concludes the proof.
\end{proof}

\begin{remark}[Flat convergence of sliced currents]
The convergence from \Cref{thm:weakConvergenceSlices} is consistent with the following property of slices of $1$-currents:
If $\sigma^j$, $j\in\N$, is a sequence of $1$-currents of finite mass with $\sigma^j\to\sigma$ in the flat norm,
then (potentially after choosing a subsequence) $\sigma^j_{\xi,t}\to\sigma_{\xi,t}$ in the flat norm for almost every $\xi\in \Disk^{n-1}$, $t\in\R$ (see \cite[step\,2 in proof of Prop.\,2.5]{CoRoMa17} or \cite[Sec.\,3]{Wh99}).
\end{remark}

\begin{remark}[Characterization of sliced measures]\label{rem:characterizationSlicing}
\begin{enumerate}
\item
Let the compactly supported divergence measure vector field $\sigma$ be countably $1$-rectifiable, that is, $\sigma=\theta m\hdone\restr S$ for a countably $1$-rectifiable set $S\subset\R^n$ and $\hd^1\restr S$-measurable functions $m:S\to[0,\infty)$ and $\theta:S\to \Disk^{n-1}$, tangent to $S$ $\hdone$-almost everywhere.
Then the coarea formula for rectifiable sets \cite[Thm.\,3.2.22]{Fe69} implies
$|\theta\cdot\xi|\hdone\restr S=\hd^0\restr S_{\xi,t}\otimes\hd^1(t)$ so that
\begin{equation*}
\int_{\R^n}f\,\d\sigma\cdot\xi
=\int_Sfm\theta\cdot\xi\,\d\hd^1
=\int_\R\int_{S_{\xi,t}}fm\mathop{\mathrm{sgn}}(\xi\cdot\theta)\,\d\hd^0\,\d t
\end{equation*}
for any Borel function $f$. Hence, for almost all $t$,
\begin{equation*}
\sigma_{\xi,t}=\mathop{\mathrm{sgn}}(\xi\cdot\theta)\,m\hd^0\restr S_{\xi,t}\,.
\end{equation*}
The choice $f=\frac{\tau(m)}m\mathop{\mathrm{sgn}}(\xi\cdot\theta)$ yields
\begin{equation*}
\int_S\tau(m)|\theta\cdot\xi|\,\d\hdone
=\int_{\R}\int_{S_{\xi,t}}\tau(m)\,\d\hd^0\,\d t\,.
\end{equation*}
\item
Let the compactly supported divergence measure vector field $\sigma$ be $\hdone$-diffuse, that is, it is singular with respect to the one-dimensional Hausdorff measure on any countably $1$-rectifiable set.
Then for almost all $\xi\in \Disk^{n-1}$ and $t\in\R$, $\sigma_{\xi,t}$ is $\hd^0$-diffuse, that is, it does not contain any atoms.
Indeed, let $\sigma_{\xi,t}$ have an atom at $x\in H_{\xi,t}$, then
\begin{equation*}
x\in\Theta(\sigma)=\left\{x\in\R^n\,\middle|\,\liminf_{\rho\searrow0}|\sigma|(B_{\rho}(x))/\rho>0\right\}\,,
\end{equation*}
where $B_{\rho}(x)$ denotes the open ball of radius $\rho$ centred at $x$.
This can be deduced as follows.
Let $\phi\in\cont^\infty(\R)$ be smooth and even with support in $(-1,1)$ and $\phi(0)=\mathop{\mathrm{sgn}}(\sigma_{\xi,t}(\{x\}))$.
Further abbreviate $K=\max_{x\in\R}|\phi'(x)|>0$ and $\phi_\rho=\phi(|\cdot-x|/\rho)$ for any $\rho>0$.
Equation\,\eqref{eqn:slicingCharacterization} now implies
\begin{align*}
\int_{H_{\xi,t}}\phi_\rho\,\d\sigma_{\xi,t}
&=\int_{\{\xi\cdot x<t\}}\phi_\rho\,\d\dive\sigma+\int_{\{\xi\cdot x<t\}}\nabla\phi_\rho\cdot\d\sigma\\
&\leq\int_{\{\xi\cdot x<t\}}\phi_\rho\,\d\dive\sigma+K\frac{|\sigma|(B_\rho(x))}\rho\,.
\end{align*}
Taking on both sides the limit inferior as $\rho\to0$ we obtain $|\sigma_{\xi,t}|(\{x\})\leq K\liminf_{\rho\searrow0}|\sigma|(B_{\rho}(x))/\rho$, as desired.

As a result, for a given $\xi$ the set of $t$ such that $\sigma_{\xi,t}$ is not $\hd^0$-diffuse is a subset of $\pi_\xi(\Theta)$.
Thus it remains to show that for almost all $\xi\in \Disk^{n-1}$ the set $\pi_\xi(\Theta)$ is a Lebesgue-nullset.
Writing
\begin{equation*}
\Theta=\bigcup_{p\in\N}\Theta_p
\quad\text{for }
\Theta_p=\left\{x\in\R^n\,\middle|\,\liminf_{\rho\searrow0}\sigma(B_{\rho}(x))/\rho\geq\frac1p\right\}\,,
\end{equation*}
it actually suffices to show that $\pi_\xi(\Theta)$ is a Lebesgue-nullset for any $p\in\N$.
Now by the properties of the $1$-dimensional density of a measure \cite[Thm.\,256]{AmFuPa00},
\begin{equation*}
\hd^1(\Theta_p)\leq\frac p2|\sigma|(\Theta_p)
\end{equation*}
so that $\Theta_p$ can be decomposed into a countably $1$-rectifiable and a purely $1$-unrectifiable set \cite[p.\,83]{AmFuPa00},
\begin{equation*}
\Theta_p=\Theta_p^r\cup\Theta_p^u
\end{equation*}
($\Theta_p^u$ purely $1$-unrectifiable means $\hd^1(\Theta_p^u\cap f(\R))=0$ for any Lipschitz $f:\R\to\R^n$).
By the $\hdone$-diffusivity assumption on $\sigma$ we have (abbreviating the Lebesgue measure by $\mathcal L$)
\begin{equation*}
\mathcal L(\pi_\xi(\Theta_p^r))\leq\hd^1(\Theta_p^r)\leq\frac p2|\sigma|(\Theta_p^r)=0\,,
\end{equation*}
and by a result due to Besicovitch \cite[Thm.\,2.65]{AmFuPa00} we have
\begin{equation*}
\mathcal L(\pi_\xi(\Theta_p^u))=0
\end{equation*}
for almost all $\xi\in\Disk^{n-1}$.
Thus, for almost all $\xi\in\Disk^{n-1}$ we have $\mathcal L(\pi_\xi(\Theta_p))=0$, as desired.

\begin{remark}[Characterization of divergence measure vector fields]
By a result due to Smirnov \cite{Sm93}, any divergence measure vector field $\sigma$ can be decomposed into simple oriented curves
$\sigma_\gamma=\pushforward{\gamma}{\dot\gamma\d s\restr[0,1]}$ with $\gamma:[0,1]\to\R^n$ a Lipschitz curve and $\d s$ the Lebesgue measure, that is,
\begin{equation*}
\sigma=\int_{ J} \sigma_\gamma\,\d\mu_\sigma(\gamma)
\end{equation*}
with $J$ the set of Lipschitz curves and $\mu_\sigma$ a nonnegative Borel measure.
The results of this section can alternatively be derived by resorting to this characterization,
since the slice of a simple oriented curve $\sigma_\gamma$ can be explicitly calculated.
\end{remark}

\end{enumerate}
\end{remark}

\subsection{The $\Gamma-\liminf$ inequality}

We now prove the desired $\liminf$-inequality, which as already anticipated will be obtained by slicing.

\begin{proposition}[$\Gamma-\liminf$ of phase field functional]\label{thm:GammaLiminf}
Let $\mu_+,\mu_-\in\prob(\overline\Omega)$.
We have
\begin{equation*}
\Gamma-\liminf_{\varepsilon\to0} E_\varepsilon^{\mu_+,\mu_-}= E^{\mu_+,\mu_-}
\end{equation*}
with respect to weak-$*$ convergence in $\meas(\overline\Omega;\R^2)$ and strong convergence in $L^1(\Omega)^N$.
\end{proposition}
\begin{proof}
Let $(\sigma^\varepsilon,\varphi_1^\varepsilon,\ldots,\varphi_N^\varepsilon)$ converge to $(\sigma,\varphi_1,\ldots,\varphi_N)$ in the considered topology.
We first extend $\sigma^\varepsilon$ and $\sigma$ to $\R^2\setminus\overline\Omega$ by zero and $\varphi_i^\varepsilon$ and $\varphi_i$ to $\R^2\setminus\Omega$ by $1$ for $i=1,\ldots,N$.
The phase field cost functional and the cost functional are extended to $\R^2$ in the obvious way (their values do not change).
Without loss of generality (potentially after extracting a subsequence)
we may assume $\lim_{\varepsilon\to0}E_\varepsilon^{\mu_+,\mu_-}[\sigma^\varepsilon,\varphi_1^\varepsilon,\ldots,\varphi_N^\varepsilon]$ to exist and to be finite (else there is nothing to show).
As a consequence we have $\dive\sigma^\varepsilon=\mu_+^\varepsilon-\mu_-^\varepsilon$ as well as $\dive\sigma=\mu_+-\mu_-$
and $\varphi_1\equiv\ldots\equiv\varphi_N\equiv1$ (since the phase field cost functional is bounded below by $\sum_{i=1}^N\frac{\beta_i}{2\varepsilon}\|\varphi_i^\varepsilon-1\|_{L^2}^2$).

Now let $A\subset\R^2$ open and bounded; the restriction of the phase field cost functional to a domain $A$ will be denoted $\E_\varepsilon^{\mu_+,\mu_-}[\cdot;A]$.
Choosing some $\xi\in S^1$, by Fubini's decomposition theorem we have
\begin{align*}
&\E_\varepsilon^{\mu_+,\mu_-}[\sigma^\varepsilon,\varphi_1^\varepsilon,\ldots,\varphi_N^\varepsilon;A]\\
&=\int_{-\infty}^\infty\int_{A_{\xi,t}}\bigg(\omega_\varepsilon\left(\alpha_0,\frac{\gamma_\varepsilon(x)}{\varepsilon},|\sigma^\varepsilon(x)|\right)
+\sum_{i=1}^N\frac{\beta_i}2\left[\varepsilon|\nabla\varphi_i|^2+\frac{(\varphi_i-1)^2}\varepsilon\right]\bigg)_{\xi,t}\,\d x\,\d t\\
&\geq\int_{-\infty}^\infty\int_{A_{\xi,t}}\omega_\varepsilon\left(\alpha_0,\frac{(\gamma_\varepsilon)_{\xi,t}}{\varepsilon},|\sigma^\varepsilon_{\xi,t}|\right)
+\sum_{i=1}^N\frac{\beta_i}2\left[\varepsilon|(\varphi_i^\varepsilon)_{\xi,t}'|^2+\frac{((\varphi_i^\varepsilon)_{\xi,t}-1)^2}\varepsilon\right]\,\d x\,\d t\\
&=\int_{-\infty}^\infty\F_\varepsilon[\sigma^\varepsilon_{\xi,t},(\varphi_1^\varepsilon)_{\xi,t},\ldots,(\varphi_N^\varepsilon)_{\xi,t};A_{\xi,t}]\,\d t\,,
\end{align*}
where $\F_\varepsilon$ is the dimension-reduced phase field energy from \Cref{def:reducedFunctionals} and for simplicity we identified the domain $A_{\xi,t}$ of the sliced functions with an open subset of the real line.
Fatou's lemma thus implies
\begin{equation*}
\liminf_{\varepsilon\to0}\E_\varepsilon^{\mu_+,\mu_-}[\sigma^\varepsilon,\varphi_1^\varepsilon,\ldots,\varphi_N^\varepsilon;A]
\geq\int_{-\infty}^\infty\liminf_{\varepsilon\to0}\F_\varepsilon[\sigma^\varepsilon_{\xi,t},(\varphi_1^\varepsilon)_{\xi,t},\ldots,(\varphi_N^\varepsilon)_{\xi,t};A_{\xi,t}]\,\d t\,.
\end{equation*}
By assumption, the left-hand side is finite so that the right-hand side integrand is finite for almost all $t\in\R$ as well.
Pick any such $t$ and pass to a subsequence such that $\liminf$ turns into $\lim$.
Indeed $\sigma^\varepsilon_{\xi,t}\weakstarto\sigma_{\xi,t}$  for every $\xi$ and almost all $t$, as $\sigma^\varepsilon\weakstarto\sigma$ and~\Cref{thm:weakConvergenceSlices}. Thus, \Cref{thm:GammaLiminfReduced} on the reduced dimension problem implies
\begin{equation*}
\liminf_{\varepsilon\to0}\F_\varepsilon[\sigma^\varepsilon_{\xi,t},(\varphi_1^\varepsilon)_{\xi,t},\ldots,(\varphi_N^\varepsilon)_{\xi,t};A_{\xi,t}]
\geq\F[\sigma_{\xi,t},(\varphi_1)_{\xi,t},\ldots,(\varphi_N)_{\xi,t};A_{\xi,t}]
\end{equation*}
for almost all $t\in\R$ so that
\begin{equation*}
\liminf_{\varepsilon\to0}\E_\varepsilon^{\mu_+,\mu_-}[\sigma^\varepsilon,\varphi_1^\varepsilon,\ldots,\varphi_N^\varepsilon;A]
\geq\int_{-\infty}^\infty\F[\sigma_{\xi,t},(\varphi_1)_{\xi,t},\ldots,(\varphi_N)_{\xi,t};A_{\xi,t}]\,\d t\,.
\end{equation*}

For notational convenience let us now define the auxiliary function $\kappa$, defined for open subsets $A\subset\R^2$, as
\begin{equation*}
\kappa(A)=\liminf_{\varepsilon\to0}\E_\varepsilon^{\mu_+,\mu_-}[\sigma^\varepsilon,\varphi_1^\varepsilon,\ldots,\varphi_N^\varepsilon;A]\,.
\end{equation*}
Furthermore, introduce the nonnegative Borel measure
\begin{equation*}
\lambda(A)=\alpha_0|\sigma^\perp|(A)+\int_{S_\sigma\cap A}\tau(m_\sigma)\,\d\hdone
\end{equation*}
as well as the $|\sigma|$-measureable Borel functions
\begin{equation*}
\psi_j:\R^2\to\R,\qquad
\psi_j=\left|\frac{\sigma}{|\sigma|}\cdot\xi^j\right|
\end{equation*}
for some sequence $\xi^j$, $j\in\N$, dense in $S^1$.

Since $\sigma$ is a divergence measure vector field, we have
\begin{align*}
\kappa(A)
&\geq\int_{-\infty}^\infty\F[\sigma_{\xi^j,t},(\varphi_1)_{\xi^j,t},\ldots,(\varphi_N)_{\xi^j,t};A_{\xi^j,t}]\,\d t\\
&=\int_{-\infty}^\infty\alpha_0|(\sigma_{\xi^j,t})^\perp|(A_{\xi^j,t})+\int_{S_{\sigma_{\xi^j,t}}\cap A_{\xi^j,t}}\tau(|m_{\sigma_{\xi^j,t}}|)\,\d\hd^0\,\d t\\
&=\alpha_0|\sigma^\perp\cdot\xi^j|(A)+\int_{S_\sigma\cap A}\tau(|m_\sigma|)|\theta_\sigma\cdot\xi^j|\,\d\hdone
\geq\int_A\psi_j\,\d\lambda
\end{align*}
for all $j\in\N$ where we used~\Cref{rem:characterizationSlicing} in the last equality. By \cite[Prop.\,1.16]{Br98} the above inequality implies
\begin{equation*}
\kappa(A)\geq\int_A\sup_j\psi_j\,\d\lambda
\end{equation*}
for any open $A\subset\R^2$.
In particular, choosing $A$ as the $1$-neighbourhood of $\Omega$ we obtain
\begin{multline*}
\liminf_{\varepsilon\to0}\E_\varepsilon^{\mu_+,\mu_-}[\sigma^\varepsilon,\varphi_1^\varepsilon,\ldots,\varphi_N^\varepsilon;\Omega]
=\kappa(A)
\geq\int_A\sup_j\psi_j\,\d\lambda\\
=\alpha_0|\sigma^\perp|(A)+\int_{S_\sigma\cap A}\tau(m_\sigma)\,\d\hdone
=\E^{\mu_+,\mu_-}[\sigma]\,,
\end{multline*}
the desired result.
\end{proof}


\subsection{Equicoercivity}
\begin{proof}[Proof of \Cref{thm:equicoercivity}]
Due to $C>E_\varepsilon^{\mu_+,\mu_-}[\sigma^\varepsilon,\varphi_1^\varepsilon,\ldots,\varphi_N^\varepsilon]\geq\frac{\beta_i}{2\varepsilon}\|\varphi_i^\varepsilon-1\|_{L^2}^2$ for all $i=1,\ldots,N$,
we have $\varphi_i^\varepsilon\to1$ in $L^2(\Omega)$ and thus also in $L^1(\Omega)$.
Furthermore, we will show further below that $\|\sigma^\varepsilon\|_{L^1}=\|\sigma^\varepsilon\|_\meas$ is uniformly bounded,
which by the Banach--Alaoglu theorem implies existence of a weakly-* converging subsequence (still denoted $\sigma^\varepsilon$) with limit $\sigma\in\meas(\overline\Omega;\R^2)$.
It is now a standard property of $\Gamma$-convergence that, due to the above equicoercivity,
any sequence of minimizers of $E_\varepsilon^{\mu_+,\mu_-}$ contains a subsequence converging to a minimizer of $E^{\mu_+,\mu_-}$.

To finish the proof we show uniform boundedness of $\sigma^\varepsilon$ in $\meas(\overline\Omega;\R^2)$.
Indeed, using $\omega_\varepsilon(\alpha_0,\frac{\gamma_\varepsilon(x)}\varepsilon,|\sigma(x)|)\geq\frac{\alpha_0}2|\sigma(x)|$ for $x\in K_0^\varepsilon$
(remember that $K_0^\varepsilon=\emptyset$ for $\alpha_0=\infty$) we obtain
\begin{equation*}
\|\sigma^\varepsilon\|_\meas
=\sum_{i=0}^N\int_{K_i^\varepsilon}|\sigma^\varepsilon|\,\d x
\leq\frac{2C}{\alpha_0}+\sum_{i=1}^N\int_{K_i^\varepsilon}|\sigma^\varepsilon|\,\d x
\end{equation*}
(the first term is interpreted as zero for $\alpha_0=\infty$).
Furthermore, by H\"older's inequality we have
\begin{multline*}
\left(\int_{K_i^\varepsilon}|\sigma^\varepsilon|\,\d x\right)^2
\leq\left(\int_{K_i^\varepsilon}\frac{(\varphi_i^\varepsilon)^2+\alpha_i^2\varepsilon^2/\beta_i}{2\varepsilon}|\sigma^\varepsilon|^2\,\d x\right)
\left(\int_{K_i^\varepsilon}\frac{2\varepsilon}{(\varphi_i^\varepsilon)^2+\alpha_i^2\varepsilon^2/\beta_i}\,\d x\right)\\
\leq C\left(\int_{K_i^\varepsilon}\frac{2\varepsilon}{(\varphi_i^\varepsilon)^2+\alpha_i^2\varepsilon^2/\beta_i}\,\d x\right)\,.
\end{multline*}
Choosing now some arbitrary $\lambda\in(0,1)$ we can estimate
\begin{multline*}
\int_{K_i^\varepsilon}\frac{2\varepsilon}{(\varphi_i^\varepsilon)^2+\alpha_i^2\varepsilon^2/\beta_i}\,\d x
=\int_{K_i^\varepsilon\cap\{\varphi_i^\varepsilon<\lambda\}}\frac{2\varepsilon}{(\varphi_i^\varepsilon)^2+\alpha_i^2\varepsilon^2/\beta_i}\,\d x
+\int_{K_i^\varepsilon\cap\{\varphi_i^\varepsilon\geq\lambda\}}\frac{2\varepsilon}{(\varphi_i^\varepsilon)^2+\alpha_i^2\varepsilon^2/\beta_i}\,\d x\\
\leq\frac4{\alpha_i^2(1-\lambda)^2}\int_{K_i^\varepsilon\cap\{\varphi_i^\varepsilon<\lambda\}}\frac{\beta_i(1-\varphi_i^\varepsilon)^2}{2\varepsilon}\,\d x
+\frac{2\varepsilon}{\lambda^2}\hd^2(\Omega)
\leq\frac{4C}{\alpha_i^2(1-\lambda)^2}+\frac{2\varepsilon}{\lambda^2}\hd^2(\Omega)\,.
\end{multline*}
Summarizing, $\|\sigma^\varepsilon\|_\meas\leq\frac C{2\alpha_0}+\sum_{i=1}^N\sqrt{\frac{4C^2}{\alpha_i^2(1-\lambda)^2}+\frac{2\varepsilon C}{\lambda^2}\hd^2(\Omega)}$.
\end{proof}

\subsection{The $\Gamma-\limsup$ inequality}
\begin{proposition}[$\Gamma-\limsup$ of phase field functional]\label{thm:GammaLimsup}
Let $\mu_+,\mu_-\in\prob(\overline\Omega)$ be an admissible source and sink.
We have
\begin{equation*}
\Gamma-\limsup_{\varepsilon\to0} E_\varepsilon^{\mu_+,\mu_-}= E^{\mu_+,\mu_-}
\end{equation*}
with respect to weak-$*$ convergence in $\meas(\overline\Omega;\R^2)$ and strong convergence in $L^1(\Omega)^N$.
\end{proposition}
\begin{proof}
Consider a mass flux $\sigma$ between the measures $\mu_+$ and $\mu_-$.
We will construct a recovery sequence $(\sigma^\varepsilon, \varphi^\varepsilon_1,\ldots, \varphi^\varepsilon_N )$ such that $\sigma^\varepsilon \weakstarto\sigma$ and $\varphi_1^\varepsilon\to 1,\ldots,\varphi_N^\varepsilon\to 1$ in the desired topology as $\varepsilon\to0$
as well as $\limsup_{\varepsilon\to0} E_\varepsilon^{\mu_+,\mu_-}[\sigma^\varepsilon, \varphi^\varepsilon_1,\ldots, \varphi^\varepsilon_N]\leq  E^{\mu_+,\mu_-}[\sigma,1,\ldots,1]$.
Without loss of generality we may restrict our attention to fluxes for which
\begin{equation*}
\E^{\mu_+,\mu_-}[\sigma]\leq C<\infty
\end{equation*}
since otherwise there is nothing to prove.
By \cite[Def.\,2.2 \& Prop.\,2.32]{BrWi17} there exists a sequence
\begin{equation*}\label{def:sigmaj}
\sigma_j=\sum_{k=1}^{M_j}m_{k,j}\theta_{k,j}\hdone\restr\Sigma_{k,j}
\end{equation*}
of polyhedral divergence measure vector fields in $\Omega$
such that
\begin{equation*}
\begin{aligned}
&\sigma_j\weakstarto\sigma,&\qquad\qquad& \dive\sigma_j=\mu_+^j-\mu_-^j,\\
&\mu_\pm^j\weakstarto\mu_\pm,&\qquad \qquad&\E^{\mu_+^j,\mu_-^j}[\sigma_j]\to\E^{\mu_+,\mu_-}[\sigma].
\end{aligned}
\end{equation*}
If $\mu_+$ and $\mu_-$ are finite linear combinations of Dirac masses (which we have assumed in the case $\alpha_0=\infty$), we may even choose $\mu_\pm^j=\mu_\pm$.
We will construct the recovery sequence based on those polyhedral divergence measure vector fields.
In the following we will use the notation
\begin{equation*}
\E_\varepsilon[\sigma,\varphi_1,\ldots,\varphi_N]
=\int_\Omega\omega_\varepsilon\left(\alpha_0,\frac{\gamma_\varepsilon(x)}{\varepsilon},|\sigma(x)|\right)\,\d x
+\sum_{i=1}^N\beta_i\L_\varepsilon[\varphi_i]
\end{equation*}
for the phase field cost functional without divergence constraints.

\medskip\noindent\textit{Step 1. Initial construction for a single polyhedral segment}\\
In this step we approximate a single line element $m_{k,j}\theta_{k,j}\hdone\restr\Sigma_{k,j}$ of $\sigma_j$ by a phase field version.
To this end we fix $j$ and $0\leq k\leq M_j$ and drop these indices from now on in the notation for the sake of legibility.
Without loss of generality we may assume $\Sigma=[0,L]\times \{0\}$, $m>0$, and $\theta=e_1$ the standard Euclidean basis vector. Set 
\begin{equation*}
\bar\iota=\alphargmin\{\alpha_i\,m+\beta_i\,|\,i=0,\ldots,N\}
\end{equation*}
to identify the phase field that will be \textit{active} on $\Sigma$ ($\bar\iota=0$ means that no phase field is active).
We first specify (a preliminary version of) the vector field $\sigma^\varepsilon$.
To this end let $d_\Sigma$ denote the distance function associated with $\Sigma$ and define the width
\begin{equation}\label{eqn:supportWidth}
a^\varepsilon_{\bar\iota}= \frac{\alpha_{\bar\iota}m\varepsilon}{2\beta_{\bar\iota}} \quad\text{ if }\bar\iota>0\quad\text{ and } a^\varepsilon_{\bar\iota}= \alpha_{0}m\varepsilon \quad\mbox{otherwise}
\end{equation}
over which the vector field will be diffused.
We now set
\begin{equation*}
\overline \sigma^\varepsilon=\frac{m}{2a^\varepsilon_{\bar\iota}}\;\chi_{\{d_\Sigma\leq a_{\bar\iota}^\varepsilon\}}\;e_1\,,
\end{equation*}
where $\chi_A$ shall denote the characteristic function of a set $A$.
Note that this vector field encodes a total mass flux of $m$ that is evenly spread over the $a_{\bar\iota}^\varepsilon$-enlargement of $\Sigma$.
The corresponding active phase field will be zero in that region.
Indeed, consider the auxiliary Cauchy problem
\begin{equation*}
\phi'=\frac{1}{\varepsilon}(1-\phi),\qquad
\phi(0)=0,
\end{equation*}
whose solution $\phi_\varepsilon(t)=1-\exp\left(-\frac{t}{\varepsilon}\right)$ represents the well-known optimal Ambrosio--Tortorelli phase field profile.
Then we set $\overline\varphi_i^\varepsilon=0$ for all $i\neq \bar\iota$ and,
if $\bar\iota\neq 0$,
\begin{equation*}
\overline\varphi^\varepsilon_{\bar\iota}(x)
=\phi_\varepsilon(\max\{0,d_\Sigma(x)-a^\varepsilon_{\bar\iota}\})=
\begin{cases}
0,&\mbox{ if } d_{\Sigma}(x)<a^\varepsilon_{\bar\iota},\\
1-\exp\left(\frac{a^\varepsilon_{\bar\iota}-d_{\Sigma}(x)}{\varepsilon}\right),&\mbox{ if } d_{\Sigma}(x)\geq a^\varepsilon_{\bar\iota}\,.
\end{cases}
\end{equation*}

\begin{figure}
\centering
\includegraphics[width = 0.45\textwidth,trim=65 15 45 25,clip]{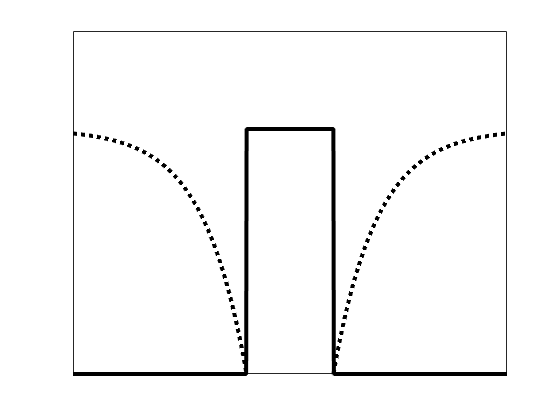}
\begin{picture}(0,0)(1,0)
\put(-155,10.3){\vector(1,0){157}} \put(-150,0){\vector(0,1){120}} 
\put(-95,5){\vector(1,0){33}} \put(-66,5){\vector(-1,0){30}} \put(-87,-5){\small$\frac{\alpha_{\bar\iota}m\varepsilon}{\beta_{\bar\iota}}$} 
\put(-90,35){\vector(0,1){60}} \put(-90,95){\vector(0,-1){80}} \put(-90,68){\small$\frac{\beta_{\bar\iota}}{\alpha_{\bar\iota}\varepsilon}$}  
\put(-90,105){\small$\varepsilon|\overline\sigma^\varepsilon|$} 
\put(-35,74){\small$|\overline\varphi_{\bar\iota}^\varepsilon|$} 
\put(-158,97){\small$1$} 
\end{picture} 
\hspace{0.5cm}
\includegraphics[width = 0.45\textwidth,trim=65 15 45 25,clip]{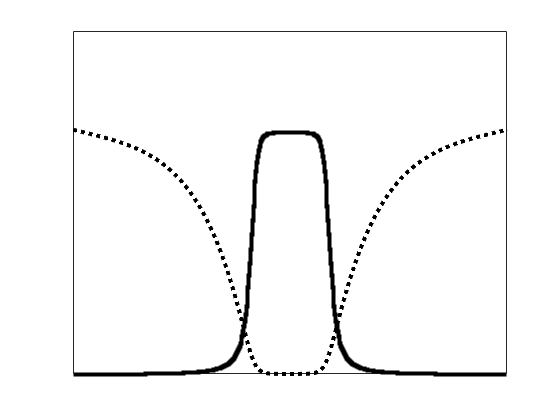} 
\begin{picture}(0,0)(1,0)
\put(-155,10.3){\vector(1,0){157}} \put(-150,0){\vector(0,1){120}} 
\put(-90,105){\small$\varepsilon|\overline\sigma^\varepsilon|$} 
\put(-35,74){\small$|\varphi^\varepsilon|$} 
\put(-158,97){\small$1$} 
\end{picture} 
\caption{Left: Sketch of the optimal profile of $|\overline\sigma^\varepsilon|$ and a phase field $\overline\varphi_{\bar\iota}^\varepsilon$ for some $\bar\iota>0$ with $m=2$, $\varepsilon=0.1$, $\alpha_{\bar\iota}=1$, $\beta_{\bar\iota}=1$. Right: Sketch of the numerical solution to the 1D problem with the same parameters.}
\end{figure}
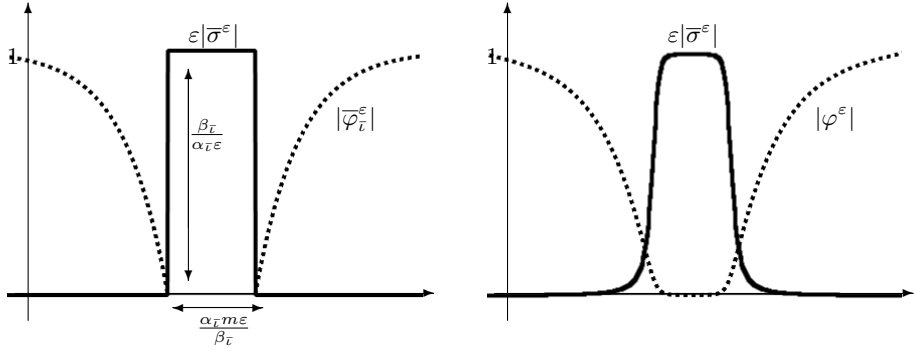


Let us now evaluate the corresponding phase field cost. In the case $\bar\iota=0$ (which can only occur for $\alpha_0<\infty$) we obtain
\begin{multline*}
\E_\varepsilon[\overline\sigma^\varepsilon(x),\overline\varphi_1^\varepsilon,\ldots,\overline\varphi_N^\varepsilon]
=\int_{\Omega}\omega_\varepsilon\left(\alpha_0,\frac{\overline\gamma_\varepsilon}{\varepsilon},|\overline\sigma^\varepsilon|\right)\,\d x
\leq\int_\Omega\alpha_0|\overline\sigma^\varepsilon|+\varepsilon^p|\overline\sigma^\varepsilon|^2\,\d x\\
=\int_{\{d_\Sigma\leq a_0^\varepsilon\}}\frac{1+\tfrac{\varepsilon^{p-1}}{2\alpha_0^2}}{2\varepsilon}\,\d x
=(\alpha_0\,m\,L+\pi\,\alpha_0^2\,m^2\,\varepsilon)(1+\tfrac{\varepsilon^{p-1}}{2\alpha_0^2})
=\alpha_0\,m\,L+C(m,L)\varepsilon^{q}\,,
\end{multline*}
where we abbreviated $q=\min\{1,p-1\}>0$ and $C(m,L)>0$ denotes a constant depending on $m$ and $L$.
In the case $\bar\iota\neq 0$ we have $|\overline\sigma^\varepsilon|=\beta_{\bar\iota}/(\alpha_{\bar\iota}\varepsilon)$ as well as $\overline\gamma_\varepsilon=\alpha_{\bar\iota}^2\varepsilon^2/\beta_{\bar\iota}$ on the support of $|\overline\sigma^\varepsilon|$ so that
\begin{align*}
\int_{\Omega}\omega_\varepsilon\left(\alpha_0,\frac{\overline\gamma_\varepsilon}{\varepsilon},|\overline\sigma^\varepsilon|\right)\,\d x
&=\left(\frac{\alpha_{\bar\iota}^2\varepsilon^2}{\beta_{\bar\iota}}\left|\frac{\beta_{\bar\iota}}{\alpha_{\bar\iota}\varepsilon}\right|^2\frac{1}{2\varepsilon}+\varepsilon^{p-2}\left|\frac{\beta_{\bar\iota}}{\alpha_{\bar\iota}}\right|^2\right) (2\,a_{\bar\iota}^\varepsilon\,L+\pi \,(a_{\bar\iota}^\varepsilon)^2)\\
&=\frac{\alpha_{\bar\iota}}{2}\,mL+C(m,L)\varepsilon^q\,.
\end{align*}
Furthermore we have $\L_{\varepsilon}(\overline\varphi^\varepsilon_i)=0$ for $i\neq \bar\iota$ and, employing the coarea formula,
\begin{equation*}
\begin{aligned}
\beta_{\bar\iota}\L_{\varepsilon}(\overline\varphi^\varepsilon_{\bar\iota})&=\frac{\beta_{\bar\iota}}{2}\int_{\Omega}\left[\varepsilon|\nabla \overline\varphi^\varepsilon_{\bar\iota}|^2+\frac{(1-\overline\varphi_{\bar\iota}^\varepsilon)^2}{\varepsilon}\right]\d x\\
&=\frac{\beta_{\bar\iota}}{2\varepsilon}\left(2La^\varepsilon_{\bar\iota}+\pi(a^\varepsilon_{\bar\iota})^2\right)+\frac{\beta_{\bar\iota}}2\int_{a^\varepsilon_{\bar\iota}}^\infty\int_{\{d_\Sigma=t\}}\left[\varepsilon|\phi_\varepsilon'(t-a_{\bar\iota}^\varepsilon)|^2+\frac{(\phi_\varepsilon(t-a_{\bar\iota}^\varepsilon)-1)^2}\varepsilon\right]\,\d\hdone(x)\,\d t\\
  &=\frac{\beta_{\bar\iota}}{2\varepsilon}\left(2La^\varepsilon_{\bar\iota}+\pi(a^\varepsilon_{\bar\iota})^2\right)+\beta_{\bar\iota}\int_{a^\varepsilon_{\bar\iota}}^\infty \frac{1}{\varepsilon}\exp\left(\frac{2(a^\varepsilon_{\bar\iota}-t)}{\varepsilon}\right)(2L+2\pi t)\,\d t\\
  &=\left(\frac{\alpha_{\bar\iota}}{2}m+\beta_{\bar\iota}\right)L+C(m,L)\varepsilon\,.
\end{aligned}
\end{equation*}
Summarizing,
\begin{equation*}\label{firstestimate}
\E_\varepsilon[\overline\sigma^\varepsilon,\overline\varphi^\varepsilon_1,\ldots, \overline\varphi^\varepsilon_N]\leq\tau(m)\,L+C(m,L)\varepsilon^q\,.
\end{equation*}

\medskip\noindent\textit{Step 2. Adapting sources and sinks of all polyhedral segments}\\
The vector field constructions $\overline\sigma^\varepsilon_{k,j}$ from the previous step for each polyhedral segment $\Sigma_{k,j}$ are not compatible with the divergence constraint associated with the measure $\sigma^j$, that is,
\begin{equation*}
\dive\left(\sum_{k=1}^{M_j}\overline\sigma_{k,j}^\varepsilon\right )\neq\rho_\varepsilon*\left(\mu_+^j-\mu_-^j\right)\,.
\end{equation*}
We remedy this by adapting the source and sink of each $\overline\sigma^\varepsilon_{k,j}$.
Set
\begin{equation*}
r(j)= \max_{k=1,\ldots,M_j}|m_{k,j}|\cdot\begin{cases}
\frac{\alpha_1}{\beta_1}&\text{if }\alpha_0=\infty,\\
\max\left\{\frac{\alpha_1}{\beta_1},\alpha_0\right\}&\text{else,}
\end{cases}
\end{equation*}
then all vector fields $\overline\sigma^\varepsilon_{k,j}$ have support in a band around $\Sigma_{k,j}$ of width no larger than $r(j)\varepsilon$.
Without loss of generality we assume $r(j)\geq1$ (else we just increase it).
Again we concentrate on a single segment with fixed $j$ and $k$ and drop these indices in the following (we will also write $r$ instead of $r(j)$).
Denote by $s^+$ and $s^-$ the starting and ending point of the segment $\Sigma$ with respect to the orientation induced by $\theta$.
Consider the elliptic boundary value problems
\begin{equation*}\label{eq:problems}
\begin{cases}
\Delta u^\pm(x) = \pm m\,(\rho * \delta_{0})(x) & \mbox{ on } B_{r}(0),\\
\nabla u^\pm(x) \cdot \nu(x) =\varepsilon \overline\sigma^\varepsilon(\varepsilon x+s^\pm) \cdot \nu(x)& \mbox{ on } \partial B_{r}(0),
\end{cases}
\end{equation*}
where $\delta_y$ denotes a Dirac mass centered at $y$, $B_{r}(y)$ denotes the open ball of radius $r$ around $y$, and $\nu$ denotes the outer unit normal to $\partial B_{r}(0)$.
Note that the boundary value problems and their solutions $u^+$ and $u^-$ are independent of $\varepsilon$ due to the definition of $\overline\sigma^\varepsilon$.
Setting
\begin{equation*}
\tilde\sigma^\varepsilon(x)
=\begin{cases}
\nabla u^\pm((x-s^\pm)/\varepsilon)/\varepsilon&\text{if }x\in B_{\varepsilon r}(s^\pm),\\
\overline\sigma^\varepsilon(x)&\text{else,}
\end{cases}
\end{equation*}
(where we assume $\varepsilon$ small enough such that $B_{\varepsilon r}(s^+)$ and $B_{\varepsilon r}(s^-)$ do not intersect)
it is straightforward to check
\begin{equation*}
\dive(\tilde\sigma^\varepsilon)=m\,\rho_\varepsilon * (\delta_{s^-} - \delta_{s^-})\,.
\end{equation*}
Furthermore, to have at least one phase field zero on the new additional support $B_{\varepsilon r}(s^+)\cup B_{\varepsilon r}(s^-)$ of the vector field we set
\begin{equation*}
\varphi^\varepsilon_{1}(x)= \min\left\{\overline\varphi^\varepsilon_{1}(x),P(|x-s^-|),P(|x-s^+|)\right\}
\text{ with }
P(t)=
\begin{cases}
0&\text{if } t\leq r\,\varepsilon,\\
\frac{t}{r\,\varepsilon} &\text{else}
\end{cases}
\end{equation*}
and $\varphi_2^\varepsilon=\overline\varphi_2^\varepsilon,\ldots,\varphi_N^\varepsilon=\overline\varphi_N^\varepsilon$.
Reintroducing now the indices $k$ and $j$, we set
\begin{equation*}
\tilde\sigma_j^\varepsilon=\sum_{k=1}^{M_j}\tilde\sigma_{k,j}^\varepsilon
\quad\text{and}\quad
(\varphi_i^\varepsilon)_j=\min\{(\varphi_i^\varepsilon)_{j,1},\ldots,(\varphi_i^\varepsilon)_{j,M_j}\}
\end{equation*}
for $i=1,\ldots,N$.
Obviously, we have, as desired,
\begin{equation*}
\dive\tilde\sigma_j
=\dive(\rho_\varepsilon*\sigma_j)
=\rho_\varepsilon*\left(\mu_+^j-\mu_-^j\right)\,.
\end{equation*}
Let us now estimate the costs.
Let us assume that $\varepsilon$ is small enough so that all balls $B_{\varepsilon r(j)}(s^{\pm}_{k,j})$ are disjoint
as are the supports $\supp\tilde\sigma_{k,j}^\varepsilon\setminus(B_{\varepsilon r(j)}(s^{+}_{k,j})\cup B_{\varepsilon r(j)}(s^{-}_{k,j}))$ for all $k$.
An upper bound can then be achieved via
\begin{align*}
\E_\varepsilon[\tilde\sigma^\varepsilon_j,(\varphi_1^\varepsilon)_{j},\ldots,(\varphi_N^\varepsilon)_{j}]
&\leq\sum_{k=1}^{M_j}\E_\varepsilon[\overline\sigma^\varepsilon_{j,k},(\overline\varphi_1^\varepsilon)_{k,j},\ldots,(\overline\varphi_N^\varepsilon)_{k,j}]\\
&\quad+\int_{B_{\varepsilon r(j)}(s^+_{j,k})\cup B_{\varepsilon r(j)}(s^-_{j,k})}\omega_\varepsilon\left(\alpha_0,\frac{(\tilde\gamma_\varepsilon)_j}\varepsilon,|\tilde\sigma^\varepsilon_j|\right)\,\d x\\
&\quad+\frac{\beta_1}2\int_{(B_{2\varepsilon r(j)}(s^+_{j,k})\cup B_{2\varepsilon r(j)}(s^-_{j,k}))\cap\{(\varphi_1^\varepsilon)_j<(\overline\varphi_1^\varepsilon)_j\}}\varepsilon|\nabla(\varphi_1^\varepsilon)_j|^2+\frac1\varepsilon((\varphi_1^\varepsilon)_j-1)^2\,\d x\,.
\end{align*}
The last summand can be bounded above by
\begin{equation*}
\beta_1\int_{B_{2\varepsilon r(j)}(0)}\varepsilon|P'(|x|)|^2+\frac1\varepsilon(P(|x|)-1)^2\,\d x\leq C\varepsilon
\end{equation*}
for some constant $C>0$.
For the second summand, note that $(\tilde\gamma_\varepsilon)_j\leq\alpha_1^2\varepsilon^2/\beta_1$ on $B_{\varepsilon r(j)}(s^\pm_{k,j})$ due to $(\varphi_1^\varepsilon)_j=0$ there;
furthermore,
$$
\omega_\varepsilon(\alpha_0,(\tilde\gamma_\varepsilon)_j/\varepsilon,|\tilde\sigma^\varepsilon_j|)
\leq\left(\frac{(\tilde\gamma_\varepsilon)_j}{2\varepsilon}+\varepsilon^p\right)|\tilde\sigma^\varepsilon_j|^2
\leq(\frac{\alpha_1^2\varepsilon}{2\beta_1}+\varepsilon^p)|\tilde\sigma^\varepsilon_j|^2\,.
$$
Thus, if we set $S^\pm=\{l\in\{1,\ldots,M_j\}\,|\,s^\pm_{j,l}=s\}$ for fixed $s=s^+_{k,j}$ or $s=s^-_{k,j}$ we have
\begin{align*}
&\int_{B_{\varepsilon r(j)}(s)}\omega_\varepsilon\left(\alpha_0,\frac{(\tilde\gamma_\varepsilon)_j}\varepsilon,|\tilde\sigma^\varepsilon_j|\right)\,\d x
\leq C\varepsilon\int_{B_{\varepsilon r(j)}(s)}|\tilde\sigma^\varepsilon_j|^2\,\d x\\
&=C\varepsilon\int_{B_{\varepsilon r(j)}(s)}\left|\sum_{l\in S^+}\frac{\nabla u^+_{j,l}(\frac{x-s}\varepsilon)}\varepsilon+\sum_{l\in S^-}\frac{\nabla u^-_{j,l}(\frac{x-s}\varepsilon)}\varepsilon\right|^2\,\d x\\
&\leq C\varepsilon M_j\left[\sum_{l\in S^+}\int_{B_{\varepsilon r(j)}(s)}\left|\frac{\nabla u^+_{j,l}(\frac{x-s}\varepsilon)}\varepsilon\right|^2\,\d x+\sum_{l\in S^-}\int_{B_{\varepsilon r(j)}(s)}\left|\frac{\nabla u^-_{j,l}(\frac{x-s}\varepsilon)}\varepsilon\right|^2\,\d x\right]\\
&= C\varepsilon M_j\left[\sum_{l\in S^+}\int_{B_{r(j)}(0)}\left|\nabla u^+_{j,l}(x)\right|^2\,\d x+\sum_{l\in S^-}\int_{B_{r(j)}(s)}\left|\nabla u^-_{j,l}(x)\right|^2\,\d x\right]\\
&=C(s,\sigma_j)\varepsilon
\end{align*}
for some constant $C(\sigma_j)>0$ depending on the polyhedral divergence measure vector field $\sigma_j$ and the considered point $s$.
In summary, we thus have
\begin{multline*}
\E_\varepsilon[\tilde\sigma^\varepsilon_j,(\varphi_1^\varepsilon)_{j},\ldots,(\varphi_N^\varepsilon)_{j}]\\
\leq\sum_{k=1}^{M_j}\big[\tau(m_{j,k})\hd^1(\Sigma_{j,k})+C(m_{j,k},L_{j,k})\varepsilon^q
+C(s_{j,k}^+,\sigma_j)\varepsilon+C(s_{j,k}^-,\sigma_j)\varepsilon
+C\varepsilon\big]
\leq\E^{\mu_+^j,\mu_-^j}[\sigma_j]+C(\sigma_j)\varepsilon^q
\end{multline*}
for some constant $C(\sigma_j)$ depending on $\sigma_j$.

\medskip\noindent\textit{Step 3. Correction of the global divergence}\\
Recall that the vector field $\sigma^\varepsilon$ to be constructed has to satisfy $\dive\sigma^\varepsilon=\rho_\varepsilon*(\mu_+-\mu_-)$.
In the case $\alpha_0=\infty$ the vector field $\tilde\sigma_j^\varepsilon$ already has that property due to $\mu_\pm^j=\mu_\pm$ (thus we set $\sigma_j^\varepsilon=\tilde\sigma_j^\varepsilon$).
However, if $\alpha_0<\infty$ (so that admissible sources and sinks $\mu_+$ and $\mu_-$ do not have to be finite linear combinations of Dirac masses)
we still need to adapt the vector field to achieve the correct divergence.
To this end, let $\lambda_\pm^j\in\meas(\{x\in\Omega\,|\,\dist(x,\partial\Omega)\geq\varepsilon;\R^2)$ be the optimal Wasserstein-1 flux between $\mu^j_\pm$ and $\mu_\pm$,
that is, $\lambda_j^\pm$ minimizes $\|\lambda\|_\meas$ among all vector-valued measures $\lambda$ with $\dive\lambda=\mu^j_\pm-\mu_\pm$.
Setting
\begin{equation*}
\sigma_j^\varepsilon=\tilde\sigma_j^\varepsilon+\rho_\varepsilon*(\lambda_+^j-\lambda_-^j)
\end{equation*}
we thus obtain $\dive\sigma_j^\varepsilon=\rho_\varepsilon*(\mu_+-\mu_-)$, as desired.
The additional cost can be estimated using the fact
\begin{equation*}
\omega_\varepsilon\left(\alpha_0,\frac{(\tilde\gamma_\varepsilon)_j}\varepsilon,|a+b|\right)
\leq\omega_\varepsilon\left(\alpha_0,\frac{(\tilde\gamma_\varepsilon)_j}\varepsilon,|a|\right)
+\alpha_0|b|+\varepsilon^p(a^2+2b^2)
\end{equation*}
as well as $\|\rho_\varepsilon*\lambda_\pm^j\|_{L^\infty}\leq C\frac{\|\mu_+^j\|_\meas+\|\mu_-^j\|_\meas}\varepsilon$,
where $\|\mu_\pm^j\|_\meas$ is an upper bound for the total mass moved by $\lambda_\pm^j$ and the constant $C>0$ depends on $\rho$.
With those ingredients we obtain
\begin{multline*}
\E_\varepsilon[\sigma_j^\varepsilon,(\varphi_1^\varepsilon)_j,\ldots,(\varphi_N^\varepsilon)_j]
\leq\E_\varepsilon[\tilde\sigma_j^\varepsilon,(\varphi_1^\varepsilon)_j,\ldots,(\varphi_N^\varepsilon)_j]\\
+\alpha_0\|\rho_\varepsilon*(\lambda_+^j-\lambda_-^j)\|_{L^1}
+\varepsilon^p(\|\tilde\sigma_j^\varepsilon\|_{L^2}^2+2\|\rho_\varepsilon*(\lambda_+^j-\lambda_-^j)\|_{L^2}^2)\,.
\end{multline*}
Now $\|\rho_\varepsilon*(\lambda_+^j-\lambda_-^j)\|_{L^1}\leq\|\lambda_+^j-\lambda_-^j\|_\meas\leq\|\lambda_+^j\|_\meas+\|\lambda_-^j\|_\meas=W_1(\mu_+^j,\mu_+)+W_1(\mu_-^j,\mu_-)=\kappa_j$ for a constant $\kappa_j>0$ satisfying
\begin{equation*}
\kappa(\sigma_j)\to0\qquad\text{as }j\to\infty
\end{equation*}
since the Wasserstein-1 distance $W_1(\cdot,\cdot)$ metrizes weak-$*$ convergence.
Furthermore, in the previous steps we have already estimated $\varepsilon^p\|\tilde\sigma_j^\varepsilon\|_{L^2}^2\leq C(\sigma_j)\varepsilon^q$.
Finally,
\begin{multline*}
\varepsilon^p\|\rho_\varepsilon*(\lambda_+^j-\lambda_-^j)\|_{L^2}^2
\leq\varepsilon^p\|\rho_\varepsilon*(\lambda_+^j-\lambda_-^j)\|_{L^\infty}\|\rho_\varepsilon*(\lambda_+^j-\lambda_-^j)\|_{L^1}\\
\leq2C\varepsilon^{p-1}(\|\mu_+^j\|_\meas+\|\mu_-^j\|_\meas)\kappa_j\,.
\end{multline*}
Summarizing,
\begin{multline*}
\E_\varepsilon[\sigma_j^\varepsilon,(\varphi_1^\varepsilon)_j,\ldots,(\varphi_N^\varepsilon)_j]
\leq\E^{\mu_+^j,\mu_-^j}[\sigma_j]+C(\sigma_j)\varepsilon^q
+\alpha_0\kappa_j
+2C\varepsilon^{p-1}(\|\mu_+^j\|_\meas+\|\mu_-^j\|_\meas)\kappa_j\\
\leq\E^{\mu_+^j,\mu_-^j}[\sigma_j]+C\kappa_j+C(\sigma_j)\varepsilon^q
\end{multline*}
for some constant $C>0$ and $C(\sigma_j)>0$ depending only on $\sigma_j$.

\medskip\noindent\textit{Step 4. Extraction of a diagonal sequence}\\
We will set $\sigma^\varepsilon=\sigma_j(\varepsilon)^\varepsilon$, $\varphi_1^\varepsilon=(\varphi_1^\varepsilon)_{j(\varepsilon)},\ldots,\varphi_N^\varepsilon=(\varphi_N^\varepsilon)_{j(\varepsilon)}$ for a suitable choice $j(\varepsilon)$.
Indeed, for a monotonic sequence $\varepsilon_1,\varepsilon_2,\ldots$ approaching zero we set $j(\varepsilon_1)=1$ and
\begin{equation*}
j(\varepsilon_{i+1})=\begin{cases}
j(\varepsilon_{i})&\text{if }C(\sigma_{j(\varepsilon_i)+1})\varepsilon_i^q>\frac1{j(\varepsilon_i)+1},\\
j(\varepsilon_{i})+1&\text{else.}
\end{cases}
\end{equation*}
Then $j(\varepsilon_i)\to\infty$ and $C(\sigma_{j(\varepsilon_i)})\varepsilon_i^q\to0$ as $i\to\infty$ so that
\begin{align*}
E^{\mu_+,\mu_-}_{\varepsilon_i}[\sigma^{\varepsilon_i},\varphi_1^{\varepsilon_i},\ldots,\varphi_N^{\varepsilon_i}]
&=\E_{\varepsilon_i}[\sigma_{j({\varepsilon_i})}^{\varepsilon_i},(\varphi_1^{\varepsilon_i})_{j({\varepsilon_i})},\ldots,(\varphi_N^{\varepsilon_i})_{j({\varepsilon_i})}]\\
&\leq\E^{\mu_+^{j({\varepsilon_i})},\mu_-^{j({\varepsilon_i})}}[\sigma_{j({\varepsilon_i})}]+C\kappa_{j({\varepsilon_i})}+C(\sigma_{j({\varepsilon_i})}){\varepsilon_i}^q\\
&\leq\E^{\mu_+^{j({\varepsilon_i})},\mu_-^{j({\varepsilon_i})}}[\sigma_{j({\varepsilon_i})}]+C\kappa_{j({\varepsilon_i})}+\frac1{j({\varepsilon_i})}
\to E^{\mu_+,\mu_-}[\sigma]\,.
\qedhere
\end{align*}
\end{proof}

\section{Numerical experiments}\label{Section4}

Here we describe the numerical discretization with finite elements and the subsequent optimization procedure used in our experiments.

\subsection{Discretization}

The proposed phase field approximation allows a simple numerical discretization with piecewise constant and piecewise linear finite elements.
Let $\mathcal{T}_h$ be a triangulation of the space $\Omega$ of grid size $h$ such that $\overline{\Omega}=\bigcup_{T\in\mathcal{T}_h}\overline T$.
Denoting by $\mathbb P^m$ the space of polynomials of degree $m$, we define the finite element spaces
\begin{align*}
X_h^0 &= \{v_h\in L^\infty(\Omega) \ | \ v_h\restriction_T \in\mathbb{P}^0 \ \forall \ T\in\mathcal{T}_h \}\,,\\
X_h^1 &= \{v_h\in C^0(\overline{\Omega}) \ | \ v_h\restriction_T \in\mathbb{P}^m \ \forall \ T\in\mathcal{T}_h \}\,.
\end{align*}
Using as before the notation
\begin{equation*}
\E_\varepsilon[\sigma,\varphi_1,\ldots,\varphi_N]
=\int_\Omega\omega_\varepsilon\left(\alpha_0,\frac{\gamma_\varepsilon(x)}{\varepsilon},|\sigma(x)|\right)\,\d x
+\sum_{i=1}^N\beta_i\L_\varepsilon[\varphi_i]\,,
\end{equation*}
the discretized phase field problem now reads
\begin{multline*}
\min_{\substack{(\sigma,\varphi_1,\ldots,\varphi_N)\in X_h^0\times(X_h^1)^N\\\varphi_1\restriction_{\partial\Omega}=\ldots=\varphi_N\restriction_{\partial\Omega}=1}}\E_\varepsilon[\sigma,\varphi_1,\ldots,\varphi_N]\\
\quad\text{such that }
\int_\Omega-\sigma\cdot\nabla\lambda\,\d x=\int_\Omega\rho_\varepsilon*(\mu_+-\mu_-)v_h\,\d x\;\forall\lambda\in X_h^1\,,
\end{multline*}
where all integrals are evaluated using midpoint quadrature and the divergence constraint is enforced in weak form.
In our numerical experiments we use $\Omega=(0,1)^2$ with a regular quadrilateral grid whose squares are all divided into two triangles.


\subsection{Optimization}

Here we describe the numerical optimization method used to find a minimizer of the objective functional.
We first elaborate the simplest case in which the transport cost $\tau$ only features a single affine segment, that is, $\alpha_0=\infty$ and $N=1$.
Afterwards we consider the setting with $N>1$ and subsequently also with $\alpha_0<\infty$, which requires more care due to the higher complexity of the energy landscape.

\paragraph{Single phase field and no diffuse mass flux ($N=1$, $\alpha_0=\infty$)}

In this case the energy reads 
\begin{equation*} \label{eq:ProblemWithOnePhi}
\mathcal{E}_\varepsilon[\sigma,\varphi] = \int_\Omega \frac{\gamma_\varepsilon(x)}{\varepsilon}\frac{|\sigma(x)|^2}2 + \frac{\beta_1}2\left(\varepsilon|\nabla\varphi(x)|^2
+ \frac{(\varphi(x)-1)^2}\varepsilon\right)\, \d x
\end{equation*}
with $\gamma_\varepsilon(x) = \varphi(x)^2+\alpha_1^2\varepsilon^2/\beta_1$. The employed optimization method is similar to the one presented in \cite{ChaFerMer17} and updates
the variables $\sigma$ and $\varphi$ alternatingly.

Let us abbreviate $f_\varepsilon=\rho_\varepsilon*(\mu_+-\mu_-)$.
For minimization with respect to $\sigma$, we use the dual variable $\lambda\in X_h^1$ to enforce the divergence constraint and write 
\begin{align*}
\min_{\substack{\sigma\in X_h^0\\\int_\Omega\sigma\cdot\nabla\lambda+\lambda f_\varepsilon\,\d x=0\,\forall\lambda\in X_h^1}}\int_\Omega \frac{\gamma_\varepsilon}{\varepsilon}\frac{|\sigma|^2}2\,\d x
&=\min_{\sigma\in X_h^0}\max_{\lambda\in X_h^1}\int_\Omega \frac{\gamma_\varepsilon}{\varepsilon}\frac{|\sigma|^2}2 - \sigma\cdot\nabla\lambda-\lambda f_\varepsilon\,\d x\\
&=\max_{\lambda\in X_h^1}\min_{\sigma\in X_h^0}\int_\Omega \frac{\gamma_\varepsilon}{\varepsilon}\frac{|\sigma|^2}2 - \sigma\cdot\nabla\lambda-\lambda f_\varepsilon\,\d x\,,
\end{align*}
where the last step follows by standard convex duality.
The minimization in $\sigma$ can be performed explicitly, yielding $\sigma=\frac{\varepsilon\nabla\lambda}{\gamma_\varepsilon}$. Inserting this solution leads to a maximization problem in $\lambda$,
\begin{equation*}
\underset{\lambda}{\max} \ \int_\Omega -\frac{\varepsilon|\nabla\lambda|^2}{2\gamma_\varepsilon} - \lambda f_\varepsilon\,\d x\,. 
\end{equation*}
The corresponding optimality conditions,
\begin{equation} \label{eq:ProblemLambda}
\int_\Omega \frac{\varepsilon\nabla\lambda\cdot\nabla\mu}{\gamma_\varepsilon}\,\d x = - \int_\Omega \mu f_\varepsilon\,\d x \quad \forall \mu\in X_h^1\,, 
\end{equation}
represent a linear system of equations that can readily be solved numerically for $\lambda$ so that subsequently $\sigma$ can be computed
(note that $\int_\Omega f_\varepsilon\,\d x=0$ so that the above equation has a solution).

Fixing $\sigma$, the optimality condition for $\varphi$ reads 
\begin{equation} \label{eq:ProblemPhi}
\int_\Omega \frac{|\sigma|^2\varphi\psi}{\varepsilon} + \beta_1\varepsilon\nabla\varphi\cdot\nabla\psi + \frac{\beta_1}{\varepsilon}(\varphi-1)\psi\,\d x = 0
\quad\forall\psi\in X_h^1\text{ with }\psi\restriction_{\partial\Omega}=0\,,
\end{equation}
which can again be solved under the constraint $\varphi_1\restriction_{\partial\Omega}=1$ by a linear system solver.

In addition to the alternating minimization a stepwise decrease of the phase field parameter $\varepsilon$ is performed,
starting from a large value $\varepsilon_{\text{start}}$ for which the energy landscape shows fewer local minima.
\Cref{alg:SPFS} summarizes the procedure.

\begin{algorithm}
	\caption{Minimization for $N=1$, $\alpha_0=\infty$}
	\label{alg:SPFS}
	\begin{algorithmic}
		\Function{SPFS}{$\varepsilon_{\text{start}}, \varepsilon_{\text{end}}, N_{\text{iter}}, \alpha_1, \beta_1, \mu_+, \mu_-, \rho_{\varepsilon_{\text{end}}}$}
		\State set $f_\varepsilon = (\mu_+-\mu_-)\ast\rho_{\varepsilon_{\text{end}}}$, $\sigma^0=0$
		\For{$j=1,\ldots,N_{\text{iter}}$}
		\State set $\varepsilon_j=\varepsilon_{\text{start}}-(j-1)\frac{\varepsilon_{\text{start}}-\varepsilon_{\text{end}}}{N_{\text{iter}}-1}$
		\State set $\varphi^j$ to the solution of \eqref{eq:ProblemPhi} for given fixed $\sigma=\sigma^{j-1}$
		\State set $\gamma^j_\varepsilon = \left(\varphi^j\right)^2+\alpha_1^2\varepsilon_j^2/\beta_1$ 
		\State set $\lambda^j$ to the solution of \eqref{eq:ProblemLambda} for given fixed $\gamma_\varepsilon=\gamma^j_\varepsilon$
		\State set $\sigma^j = \frac{\varepsilon_j\nabla\lambda^j}{2\gamma^j_\varepsilon}$
		\EndFor
		\EndFunction
		\State \textbf{return} $\sigma^{N_{\text{iter}}},\varphi^{N_{\text{iter}}},\lambda^{N_{\text{iter}}}$
	\end{algorithmic}
\end{algorithm}

\paragraph{Multiple phase fields and no diffuse mass flux ($N>1$, $\alpha_0=\infty$)}
Again we aim for an alternating minimization scheme.
Fixing $\varphi_1,\ldots,\varphi_N$, the optimization for $\sigma$ is the same as before, since only $\gamma_\varepsilon$ changes.
However, the optimization in the phase fields $\varphi_1,\ldots,\varphi_N$ is strongly nonconvex (due to the minimum in $\gamma_\varepsilon$)
and thus requires a rather good initialization and care in the alternating scheme.

To avoid minimization for phase field $\varphi_i$ with the $\min$-function inside $\gamma_\varepsilon$ we perform a heuristic operator splitting:
in each iteration of the optimization we first identify at each location which term inside $\gamma_\varepsilon=\min(\varphi_1^2+\alpha_1^2\varepsilon^2/\beta_1,\ldots,\varphi_N^2+\alpha_N^2\varepsilon^2/\beta_N)$ is the minimizer,
which is equivalent to specifying the regions
\begin{equation}\label{eqn:regionUpdate}
R_i^\varepsilon = \{x\in\Omega \, | \, \gamma_\varepsilon(x) = \varphi_i(x)^2+\alpha_i^2\varepsilon^2/\beta_i \}\,,
\quad i=1,\ldots,N.
\end{equation}
Afterwards we then optimize the energy $\E_\varepsilon$ separately for each phase field $\varphi_i$ assuming the regions $R_i^\varepsilon$ fixed, that is, we minimize
\begin{equation*}
\sum_{i=1}^N \int_\Omega \frac{\varphi_i(x)^2+\alpha_i^2\varepsilon^2/\beta_i}{\varepsilon}\frac{|\sigma(x)|^2}2\chi_{R_i^\varepsilon}(x) +
\frac{\beta_i}2\varepsilon|\nabla\varphi_i(x)|^2 + \frac{\beta_i}2\frac{(\varphi_i(x)-1)^2}{\varepsilon}\,\d x\,,
\end{equation*}
where $\chi_{R_i^\varepsilon}$ is the characteristic function of region $R_i^\varepsilon$.
Similarly to the case $N=1$ of a single phase field, this amounts to solving the linear system
\begin{equation}\label{eq:ProblemMultPhi}
\int_\Omega \frac{\varphi_i\psi}{\varepsilon}|\sigma|^2\chi_{R_i^\varepsilon} +
\beta_i\varepsilon\nabla\varphi_i\cdot\nabla\psi + \beta_i\frac{\varphi_i\psi}{\varepsilon}\,\d x
\quad\forall\psi\in X_h^1\text{ with }\psi\restriction_{\partial\Omega}=0
\end{equation}
for $\varphi_i\in X_h^1$ with $\varphi_i\restriction_{\partial\Omega}=1$.


Since in the above simple approach the regions $R_i^\varepsilon$ and phase fields $\varphi_i$ can move, but cannot nucleate within a different region
(indeed, imagine for instance $R_1^\varepsilon=\Omega$, then $\varphi_2,\ldots,\varphi_N$ will be optimized to equal $1$ everywhere
so that in the next iteration again $R_1^\varepsilon=\Omega$),
a suitable initial guess is crucial.
To provide such a guess for the initial regions $R_i^\varepsilon$, we proceed as follows.
We first generate some flux network $\sigma^0$ consistent with the given source and sink.
This can for instance be done using the previously described algorithm for just a single phase field:
in our simulations we simply ignore $\varphi_2,\ldots,\varphi_N$ and pretend only the phase field $\varphi_1$ would exist
(essentially this means we replace the cost function $\tau$ by $m\mapsto\alpha_1m+\beta_1$ for $m>0$;
of course, an alternative choice would be to take $m\mapsto\alpha m+\beta$ for some $\alpha,\beta>0$ that better approximate $\tau$ for a larger range of values $m$).
We then identify the total mass flowing through each branch of $\sigma^0$.
To this end we convolve $|\sigma^0|$ with the characteristic function $\chi_{B_r(0)}$ of a disc of radius $r$.
If $r$ is sufficiently large compared to the width of the support of $\sigma^0$, we obtain
\begin{equation*}
\left(\chi_{B_r(0)}\ast|\sigma^0|\right)(x)
= \int_{B_r(x)} |\sigma^0|(y) \,\d y
\approx 2rm(x)\,,
\end{equation*}
where $m(x)$ denotes the mass flux through the nearby branch of $\sigma^0$.
Now we can compute the regions
\begin{equation}\label{eq:OptIndex}
R_i^\varepsilon = \{x\in\Omega\,|\,i=\alphargmin_{j=1,\ldots,N}\{\alpha_jm(x)+\beta_j\}\}
\end{equation}
and furthermore use $\sigma^0$ as initial guess of the vector field.
\Cref{alg:MPFS} summarizes the full procedure.

\begin{algorithm}
	\caption{Minimization for $N>1$, $\alpha_0=\infty$}
	\label{alg:MPFS}
	\begin{algorithmic}
		\Function{MPFS}{$\varepsilon_{\text{start}}, \varepsilon_{\text{end}}, N_{\text{iter}}, \alpha_1,\ldots,\alpha_N, \beta_1,\ldots,\beta_N, \mu_+, \mu_-, \rho_{\varepsilon_{\text{end}}}$}
		\State set $f_\varepsilon = (\mu_+-\mu_-)\ast\rho_{\varepsilon_{\text{end}}}$
		\State set $\left(\sigma^0,\cdot,\cdot\right) = SPFS(\varepsilon_{\text{start}},\varepsilon_{\text{end}},N_{\text{iter}},\alpha_1, \beta_1, \mu_+, \mu_-, \rho_{\varepsilon_{\text{end}}})$ 
		\State compute regions $R_1^\varepsilon,\ldots,R_N^\varepsilon$ via \eqref{eq:OptIndex}
		\For{$j=1,\ldots,N_{\text{iter}}$}
		\State set $\varepsilon_j=\varepsilon_{\text{start}}-(j-1)\frac{\varepsilon_{\text{start}}-\varepsilon_{\text{end}}}{N_{\text{iter}}-1}$
		\State set $\varphi_i^j$ to the solution of \eqref{eq:ProblemMultPhi} for given fixed $\sigma=\sigma^{j-1}$, $i=1,\ldots,N$
		\State update regions $R_1^\varepsilon,\ldots,R_N^\varepsilon$ via \eqref{eqn:regionUpdate}
		\State set $\gamma^j_\varepsilon = \min_{i=1,\ldots,N} \left((\varphi_i^j)^2+\alpha_i^2\varepsilon^2/\beta_i\right)$ 
		\State set $\lambda^j$ to the solution of \eqref{eq:ProblemLambda} for given fixed $\gamma_\varepsilon=\gamma^j_\varepsilon$
		\State set $\sigma^j = \frac{\varepsilon_j\nabla\lambda^j}{2\gamma^j_\varepsilon}$
		\EndFor
		\EndFunction
		\State \textbf{return} $\sigma^{N_{\text{iter}}},\varphi_1^{N_{\text{iter}}},\ldots,\varphi_N^{N_{\text{iter}}}$
	\end{algorithmic}
\end{algorithm}

Note that the estimate $m(x)$ of the flowing mass is only valid in close proximity of the support of $\sigma^0$
so that the regions $R_i^\varepsilon$ are only reliable near $\sigma^0$.
However, away from $\sigma^0$ all phase fields will be close to $1$ anyway so that the regions $R_i^\varepsilon$ do not play a role there.
The effectiveness of the initialization can be further improved by an energy rescaling which we typically perform in our simulations:
Recall that the optimal width \eqref{eqn:supportWidth} of the support of the vector field $\sigma$ not only depends on the transported mass $m(x)$, but also on which phase field is active at $x$.
Thus, initializing with some vector field $\sigma$ that was computed based on preliminary active phase fields may erroneously give slight preference to incorrect phase fields.
This can be avoided by a small parameter change which assigns a different $\varepsilon$ to each phase field.
Indeed, setting $\varepsilon_i=\beta_i\varepsilon/\alpha_i$ to be the phase field parameter associated with phase field $\varphi_i$,
equation \eqref{eqn:supportWidth} shows that the support width of $\sigma$ becomes $m\varepsilon$ and thus no longer depends on the phase field.
Thus, in practice we usually work with the phase field cost
\begin{equation}\label{eqn:rescaledCost}
\tilde\E_\varepsilon[\sigma,\varphi_1,\ldots,\varphi_N]
=\int_\Omega\omega_\varepsilon\left(\alpha_0,\frac{\tilde\gamma_\varepsilon(x)}{\varepsilon},|\sigma(x)|\right)\,\d x
+\sum_{i=1}^N\beta_i\L_{\varepsilon_i}[\varphi_i]
\end{equation}
with $\tilde\gamma_\varepsilon(x)=\min_{i=1,\ldots,N}\{\varphi_i(x)^2+\alpha_i^2\varepsilon\varepsilon_i/\beta_i\}=\min_{i=1,\ldots,N}\{\varphi_i(x)^2+\alpha_i\varepsilon^2\}$.
The $\Gamma$-convergence result can readily be adapted to this case.


\paragraph{Multiple phase fields and diffuse mass flux ($N>1$, $\alpha_0<\infty$)}
The difference to the previous case is that now there may be nonnegligible mass flux $\sigma$ in regions where no phase field $\varphi_1,\ldots,\varphi_N$ is active.
Correspondingly, we adapt the previous alternating minimization scheme by introducing the set
\begin{equation*}
R_0^\varepsilon = \left\{x\in\Omega \ | \ |\sigma(x)| > \frac{\alpha_0}{\gamma_\varepsilon(x)\slash\varepsilon} \right\}\,,
\end{equation*}
which according to the form of $\omega_\varepsilon$ in \Cref{def:phase fieldFunctional} describes the region in which mass flux $\sigma$ is penalized by $\alpha_0|\sigma|$.
The regions in which the $i$\textsuperscript{th} phase field is active are thus modified to
\begin{equation}\label{eq:OptIndex2}
\tilde{R}_i^{\varepsilon} = R_i^{\varepsilon}\setminus R_0^{\varepsilon}\,.
\end{equation}
As before, we now separately minimize
\begin{equation}\label{eq:ProblemMultPhi2}
\sum_{i=1}^N \int_\Omega \frac{\varphi_i(x)^2+\alpha_i^2\varepsilon^2/\beta_i}{\varepsilon}\frac{|\sigma(x)|^2}2\chi_{\tilde{R}_i^\varepsilon}(x) +
\frac{\beta_i}2\varepsilon|\nabla\varphi_i(x)|^2 + \frac{\beta_i}2\frac{(\varphi_i(x)-1)^2}{\varepsilon}\,\d x\,,
\end{equation}
for each phase field $\varphi_i$.
The optimization for $\sigma$ changes a little compared to the previous case since the problem is no longer quadratic and thus no longer reduces to solving a linear system.
Instead we will perform a single step of Newton's method in each iteration.
The optimization problem in $\sigma$ reads
\begin{equation*}
\min_{\substack{\sigma\in X_h^0\\\int_\Omega\sigma\cdot\nabla\lambda+\lambda f_\varepsilon\,\d x=0\,\forall\lambda\in X_h^1}}\int_\Omega\omega_\varepsilon\left(\alpha_0,\frac{\gamma_\varepsilon}{\varepsilon},|\sigma|\right)\,\d x\,,
\end{equation*}
and its optimality conditions are
\begin{equation}
\begin{aligned}
0 &= \int_{\Omega} \xi(|\sigma|)\sigma\cdot\psi - \nabla\lambda\cdot\psi \, \d x
&\text{for all }\psi\in X_h^0\,,\\
0 &= \int_\Omega\sigma\cdot\nabla\mu+\mu f_\varepsilon\,\d x
&\text{for all }\mu\in X_h^1\,,
\end{aligned}
\end{equation}
where
\begin{equation*}
\xi(|\sigma|)=\frac1{|\sigma|}\partial_3\omega_\varepsilon\left(\alpha_0,\frac{\gamma_\varepsilon}{\varepsilon},|\sigma|\right)=\min\left(\frac{\gamma_\varepsilon}\varepsilon,\frac{\alpha_0}{|\sigma|}\right)+2\varepsilon^p\,.
\end{equation*}
Letting $\hat\sigma$ and $\hat\lambda$ be the coefficients of $\sigma$ and $\lambda$ in some basis $\{b_i^0\}_i$ of $X_h^0$ and $\{b_i^1\}_i$ of $X_h^1$, respectively, the optimality conditions can be expressed as
\begin{equation*}
0 = R(\hat\sigma,\hat\lambda)
= \begin{pmatrix} M\left[\xi(|\sigma|)\right] & B \\ B^T & 0 \end{pmatrix} \begin{pmatrix} \hat{\sigma} \\ \hat{\lambda} \end{pmatrix} + \begin{pmatrix} 0 \\ F \end{pmatrix}\,,
\end{equation*}
where the finite element matrices and vectors are defined as
\begin{equation*}
M[\xi]_{ij}=\int_\Omega \xi b_i^0 b_j^0\,\d x\,,\quad
B_{ij}=\int_\Omega b_i^0 \nabla b_j^1\,\d x\,,\quad
F_{i}=\int_\Omega b_i^1 f_\varepsilon\,\d x\,.
\end{equation*}
In each iteration of the alternating minimization scheme we now take one Newton step for $0=R(\hat\sigma,\hat\lambda)$.
As before, the algorithm requires a suitable initial guess, which is determined in the same way as for the case without diffuse component. \Cref{alg:MPFSD} summarizes the alternating scheme.

\begin{algorithm}
	\caption{Minimization for $N>1$, $\alpha_0<\infty$}
	\label{alg:MPFSD}
	\begin{algorithmic}
		\Function{MPFSD}{$\varepsilon_{\text{start}}, \varepsilon_{\text{end}}, N_{\text{iter}}, \alpha_1,\ldots,\alpha_N, \beta_1,\ldots,\beta_N, \mu_+, \mu_-, \rho_{\varepsilon_{\text{end}}}$}
		\State set $f_\varepsilon = (\mu_+-\mu_-)\ast\rho_{\varepsilon_{\text{end}}}$
		\State set $\left(\sigma^0,\cdot,\lambda^0\right) = SPFS(\varepsilon_{\text{start}},\varepsilon_{\text{end}},N_{\text{iter}},\alpha_1, \beta_1, \mu_+, \mu_-, \rho_{\varepsilon_{\text{end}}})$
		\State compute regions $R_0^\varepsilon,\tilde{R}_1^\varepsilon,\ldots,\tilde{R}_N^\varepsilon$ via \eqref{eq:OptIndex2}
		\For{$j=1,\ldots,N_{\text{iter}}$}
		\State set $\varepsilon_j=\varepsilon_{\text{start}}-(j-1)\frac{\varepsilon_{\text{start}}-\varepsilon_{\text{end}}}{N_{\text{iter}}-1}$
		\State set $\varphi_i^j$ to the minimizer of \eqref{eq:ProblemMultPhi2} for given fixed $\sigma=\sigma^{j-1}$, $i=1,\ldots,N$
		\State update regions $R_0^\varepsilon,\tilde{R}_1^\varepsilon,\ldots,\tilde{R}_N^\varepsilon$ via \eqref{eqn:regionUpdate} and \eqref{eq:OptIndex2}
		\State set $\gamma^j_\varepsilon = \min_{i=1,\ldots,N} \left((\varphi_i^j)^2+\alpha_i^2\varepsilon^2/\beta_i\right)$ 
		\State set $(\hat\sigma^j,\hat\lambda^j)=(\hat\sigma^{j-1},\hat\lambda^{j-1})-DR(\hat\sigma^{j-1},\hat\lambda^{j-1})^{-1}R(\hat\sigma^{j-1},\hat\lambda^{j-1})$ for $\gamma_\varepsilon=\gamma^{j-1}_\varepsilon$
		\EndFor
		\EndFunction
		\State \textbf{return} $\sigma^{N_{\text{iter}}},\varphi_1^{N_{\text{iter}}},\ldots,\varphi_N^{N_{\text{iter}}}$
	\end{algorithmic}
\end{algorithm}

\subsection{Experimental results}

The algorithms were implemented in MATLAB$^{\tiny{\textcircled{c}}}$; parameters reported in this section refer to the rescaled cost \eqref{eqn:rescaledCost}.
We first present simulation results for a single phase field and no diffuse mass flux, $N=1$ and $\alpha_0=\infty$.
\Cref{fig:SteinerProblem1,fig:SteinerProblem2} show solutions for a source and a number of equal sinks arranged as a regular polygon.
If $\alpha_1$ is small as in \Cref{fig:SteinerProblem1}, the solution looks similar to the Steiner tree (which would correspond to $\alpha_1=0$),
while the solutions become much more asymmetric for larger $\alpha_1$ as in \Cref{fig:SteinerProblem2}.
More complex examples are displayed in \Cref{fig:singlePhaseField}.


\begin{figure}
	\centering
	\setlength\unitlength{.23\textwidth}
	\reflectbox{\includegraphics[width = .23\textwidth,trim=130 85 100 80,clip]{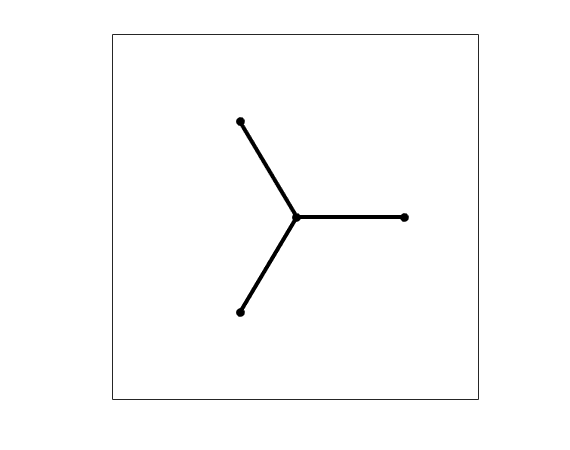}}
	\begin{picture}(0,0)(1,0)\put(-.05,.42){\small$+$}\put(.72,.0){\small$-$}\put(.72,.87){\small$-$}\end{picture}
	\reflectbox{\includegraphics[width = .23\textwidth,trim=130 85 100 80,clip]{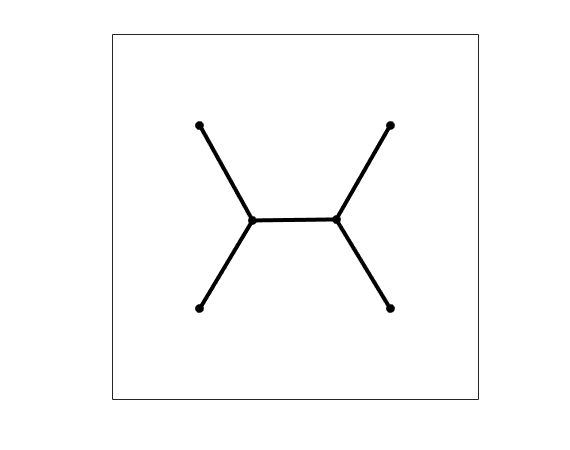}}
	\begin{picture}(0,0)(1,0)\put(.87,.01){\small$+$}\put(.02,.01){\small$-$}\put(.06,.85){\small$-$}\put(.87,.85){\small$-$}\end{picture}
	\reflectbox{\includegraphics[width = .23\textwidth,trim=130 85 110 80,clip]{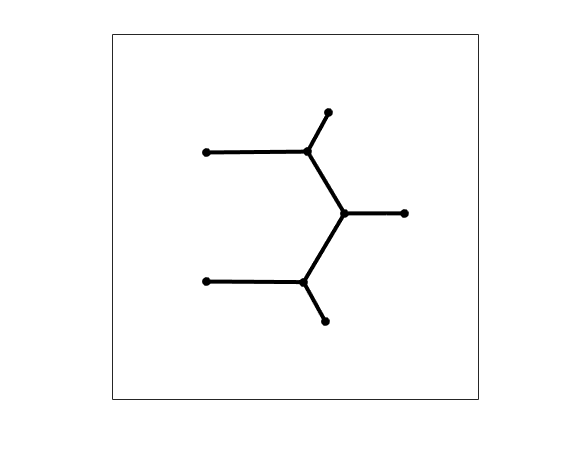}}
	\begin{picture}(0,0)(1,0)\put(-.1,.46){\small$+$}\put(.26,-.02){\small$-$}\put(.86,.14){\small$-$}\put(.86,.77){\small$-$}\put(.26,.94){\small$-$}\end{picture}
	\reflectbox{\includegraphics[width = .23\textwidth,trim=130 85 100 80,clip]{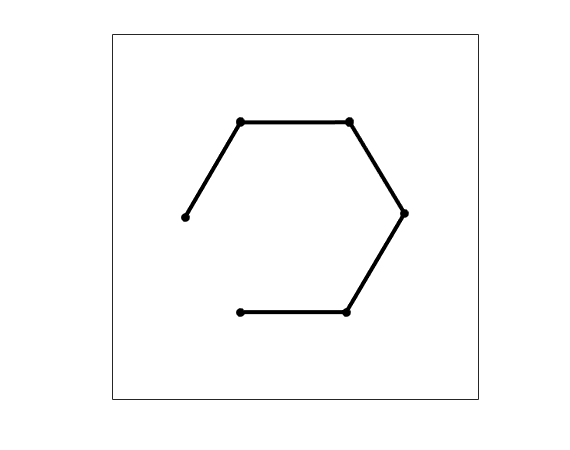}}
	\begin{picture}(0,0)(1,0)\put(-.05,.44){\small$+$}\put(.20,.0){\small$-$}\put(.68,.0){\small$-$}\put(.23,.86){\small$-$}\put(.67,.87){\small$-$}
	\put(.91,.36){\small$-$}\end{picture} \\
	\includegraphics[width = \unitlength,trim=25 15 45 15,clip]{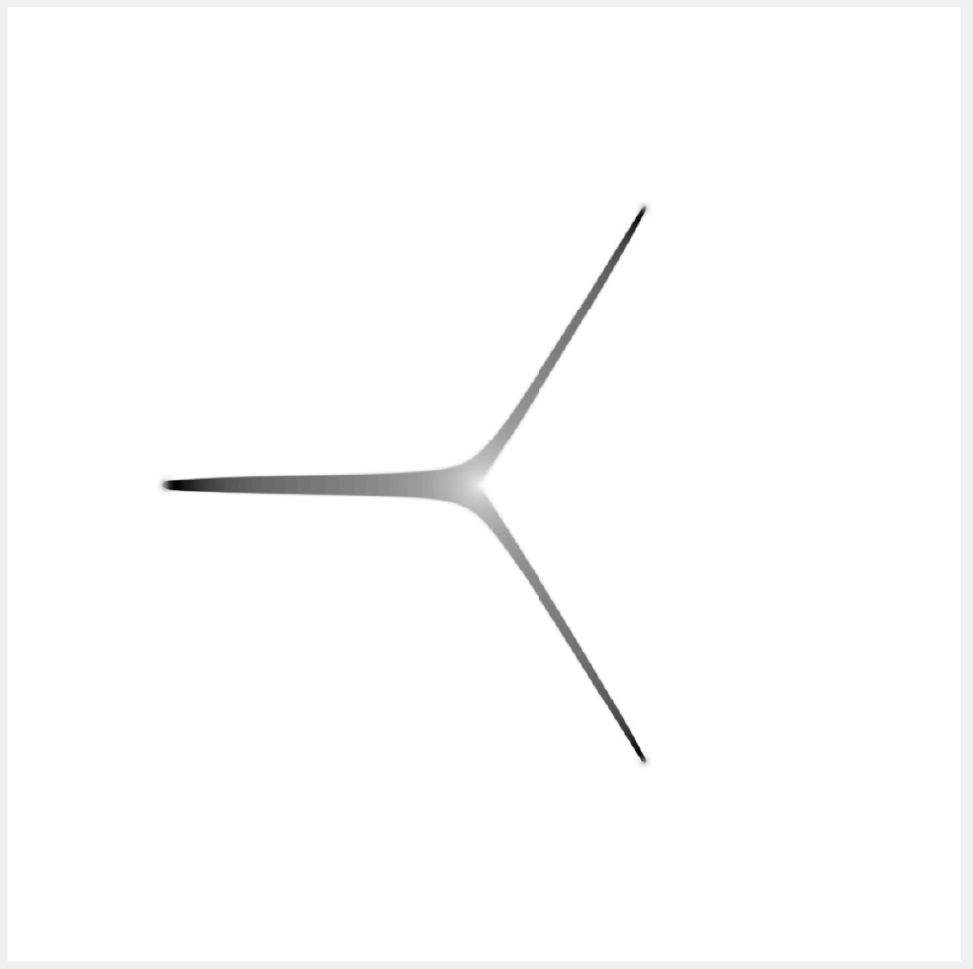}
	\begin{picture}(0,0)(1,0)\put(-.05,.57){\small$+$}\put(.73,.13){\small$-$}\put(.73,1.0){\small$-$}\end{picture}
	\includegraphics[width = \unitlength,trim=25 15 35 15,clip]{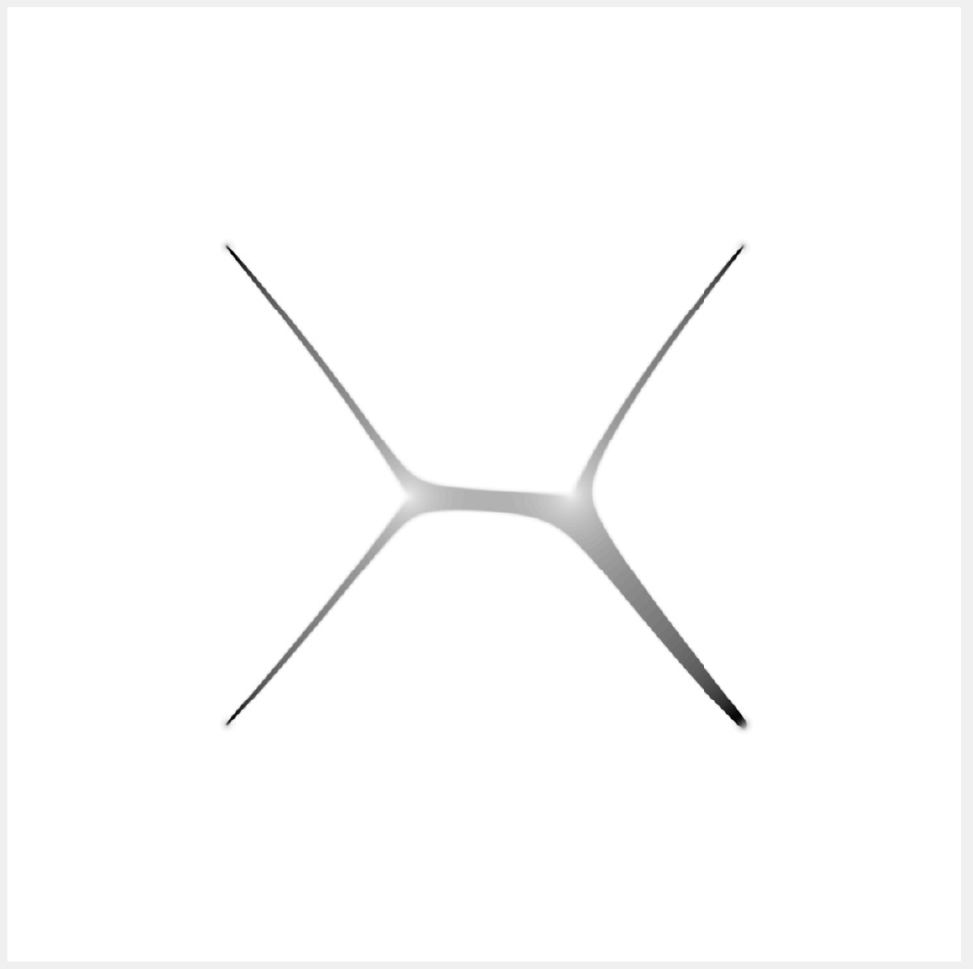} 
	\begin{picture}(0,0)(1,0)\put(.83,.15){\small$+$}\put(.06,.15){\small$-$}\put(.07,.91){\small$-$}\put(.82,.91){\small$-$}\end{picture}
	\includegraphics[width = \unitlength,trim=25 15 35 15,clip]{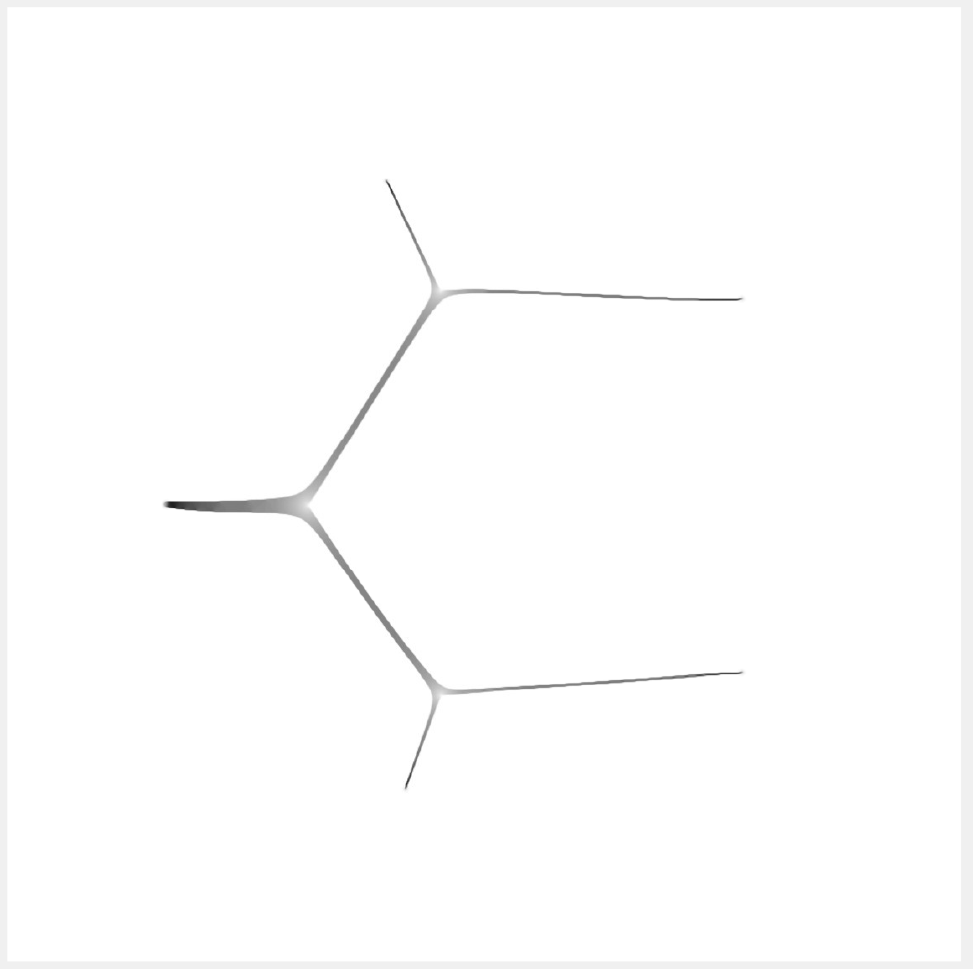} 
	\begin{picture}(0,0)(1,0)\put(-.05,.5){\small$+$}\put(.3,.09){\small$-$}\put(.85,.24){\small$-$}\put(.85,.82){\small$-$}\put(.3,.98){\small$-$}\end{picture}
	\includegraphics[width = \unitlength,trim=25 15 35 15,clip]{Result6Points1_sigma.pdf}
	\begin{picture}(0,0)(1,0)\put(-.05,.5){\small$+$}\put(.2,.11){\small$-$}\put(.65,.11){\small$-$}\put(.21,.94){\small$-$}\put(.66,.93){\small$-$}
	\put(.9,.51){\small$-$}\end{picture} \\
	\includegraphics[width = \unitlength,trim=64 36 68 25,clip]{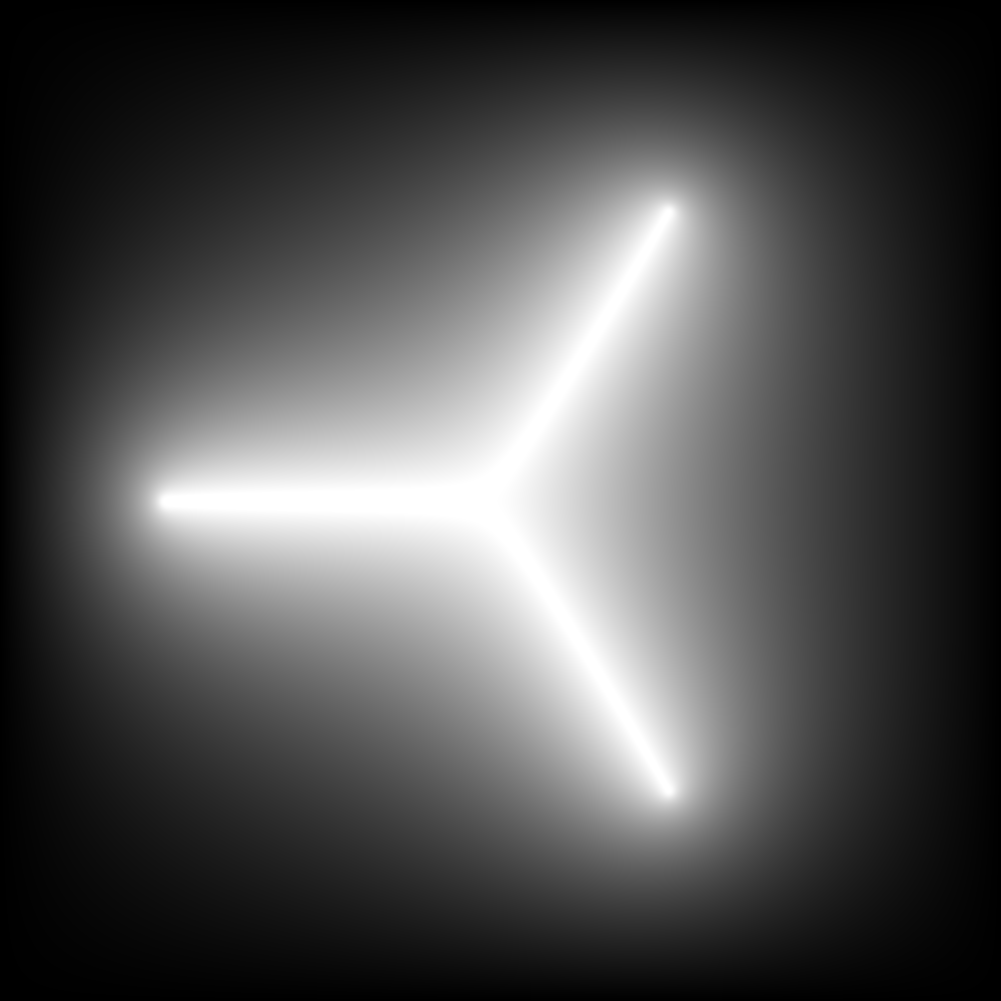} 
	\includegraphics[width = \unitlength,trim=64 36 68 25,clip]{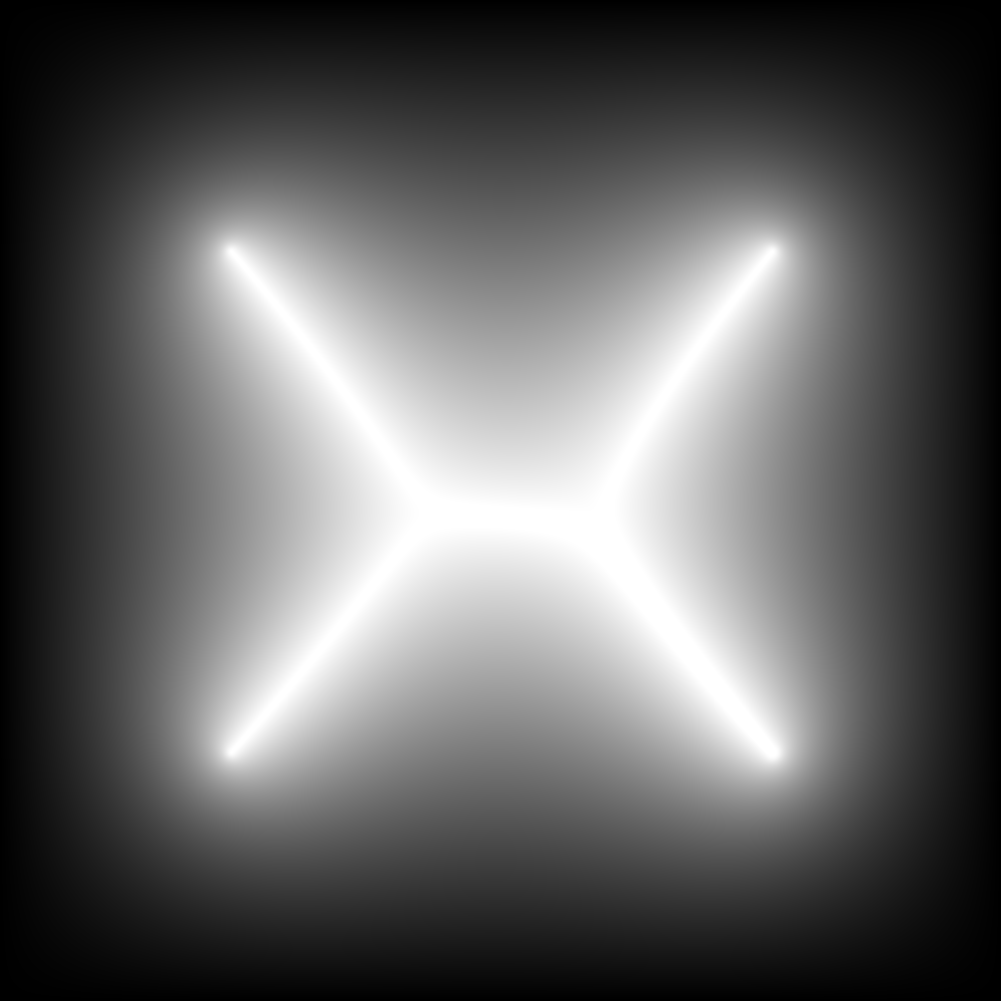} 
	\includegraphics[width = \unitlength,trim=64 36 68 25,clip]{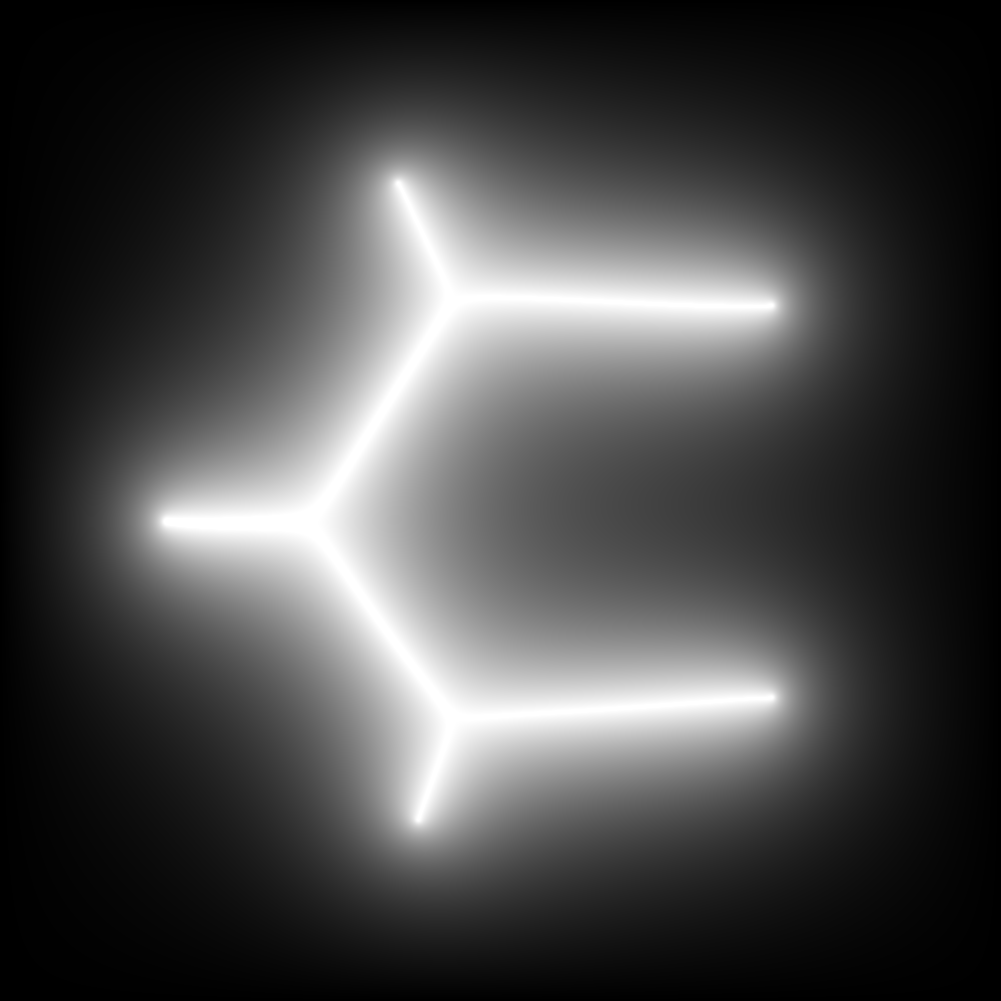} 
	\includegraphics[width = \unitlength,trim=64 36 68 25,clip]{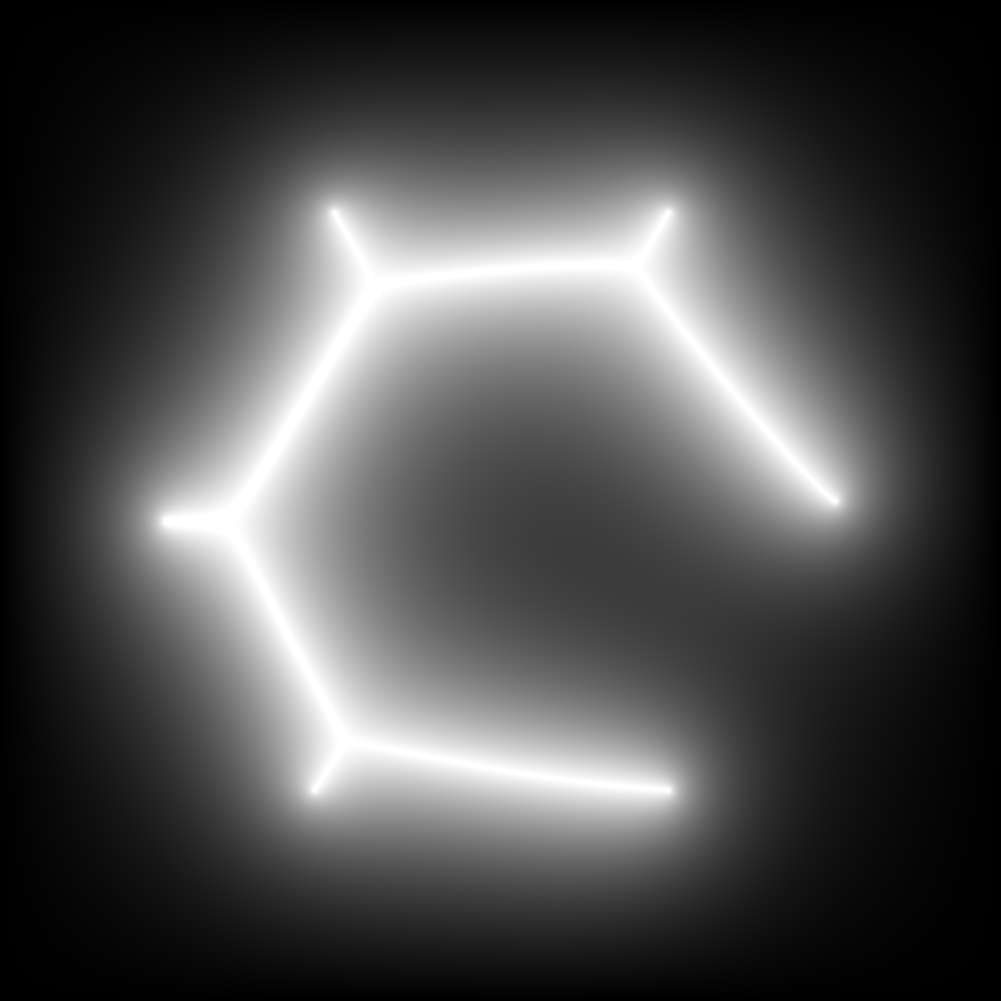} 
	\caption{Optimal transportation networks for branched transportation from a single source to a number of identical sinks at the corners of regular polygons.
	The top row shows the ground truth, computed by finite-dimensional optimization of the vertex locations in a network with straight edges,
	the bottom rows show the computation results from the phase field model, the mass flux $\sigma$ (middle, only support shown) and phase field $\varphi$ (bottom).
	Parameters were $\alpha_1=0.05$, $\beta_1=1$, $\varepsilon=0.005$.}
	\label{fig:SteinerProblem1}
\end{figure}


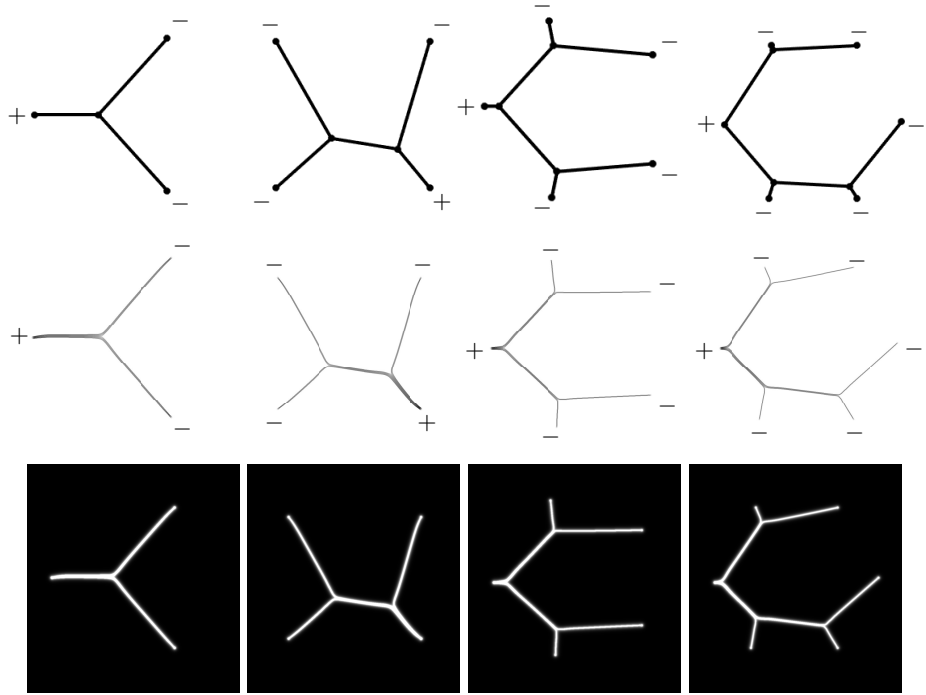
\begin{figure}
	\centering
	\setlength\unitlength{.23\textwidth}
	\reflectbox{\includegraphics[width = .23\textwidth,trim=130 85 100 80,clip]{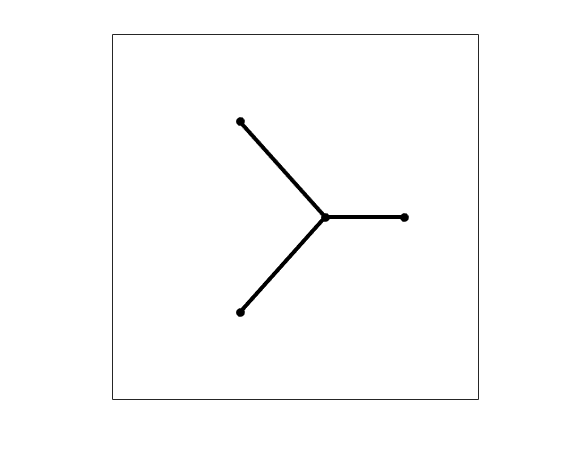}}
	\begin{picture}(0,0)(1,0)\put(-.05,.42){\small$+$}\put(.72,.0){\small$-$}\put(.72,.87){\small$-$}\end{picture}
	\reflectbox{\includegraphics[width = .23\textwidth,trim=130 85 100 80,clip]{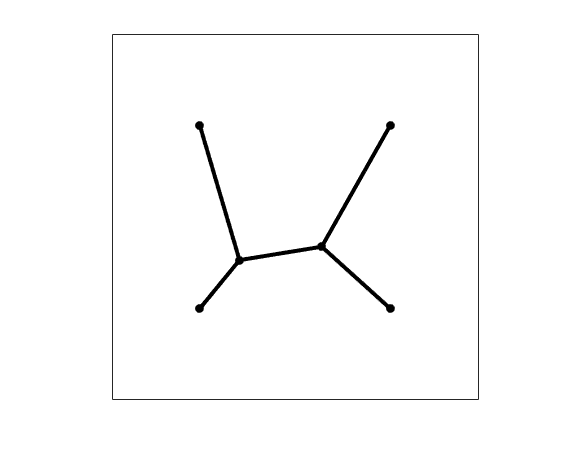}}
	\begin{picture}(0,0)(1,0)\put(.87,.01){\small$+$}\put(.02,.01){\small$-$}\put(.06,.85){\small$-$}\put(.87,.85){\small$-$}\end{picture}
	\reflectbox{\includegraphics[width = .23\textwidth,trim=130 85 110 80,clip]{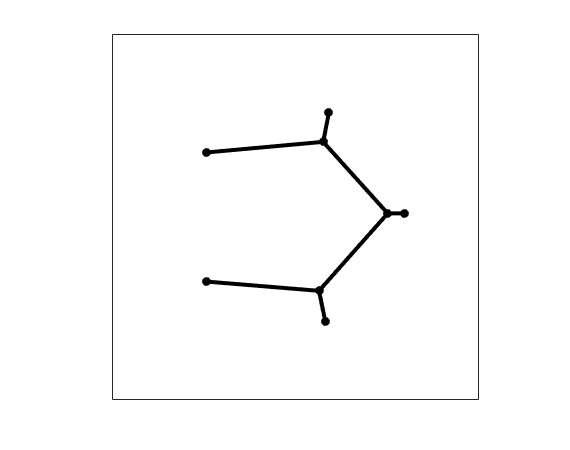}}
	\begin{picture}(0,0)(1,0)\put(-.1,.46){\small$+$}\put(.26,-.02){\small$-$}\put(.86,.14){\small$-$}\put(.86,.77){\small$-$}\put(.26,.94){\small$-$}\end{picture}
	\includegraphics[width = .23\textwidth,trim=130 85 100 80,clip,angle=180,origin=0]{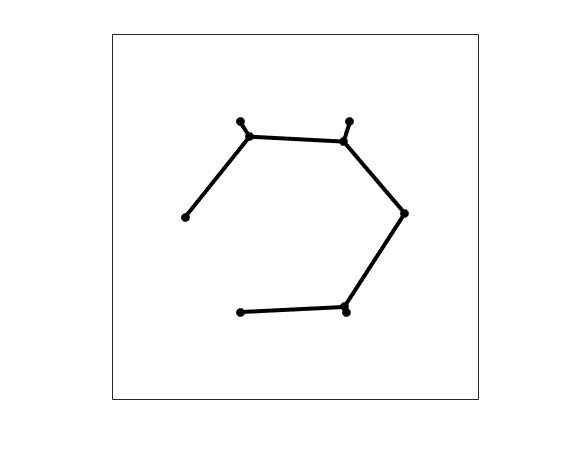}
	\begin{picture}(0,0)(1,0)\put(-.05,.38){\small$+$}\put(.22,-.04){\small$-$}\put(.68,-.04){\small$-$}\put(.23,.82){\small$-$}\put(.67,.82){\small$-$}
	\put(.94,.37){\small$-$}\end{picture} \\	
	\includegraphics[width = \unitlength,trim=25 15 45 15,clip]{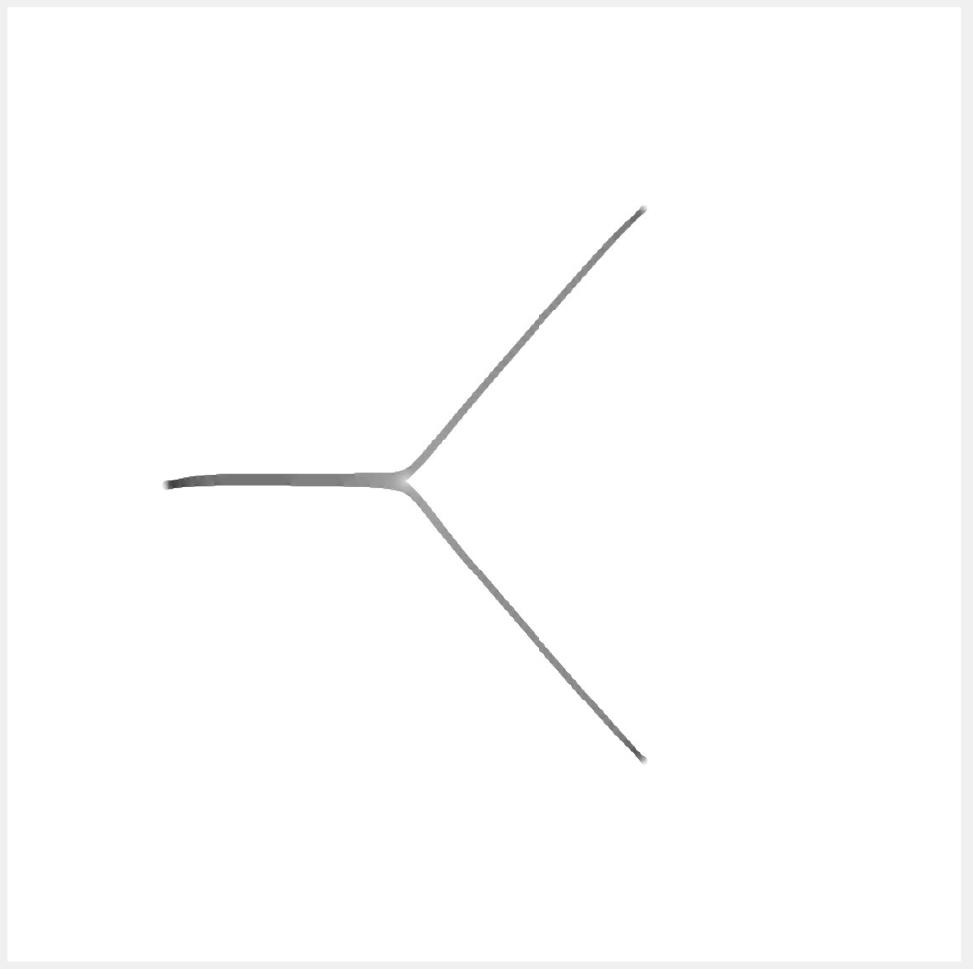} 
	\begin{picture}(0,0)(1,0)\put(-.04,.57){\small$+$}\put(.73,.13){\small$-$}\put(.73,1.0){\small$-$}\end{picture}
	\includegraphics[width = \unitlength,trim=25 15 35 15,clip]{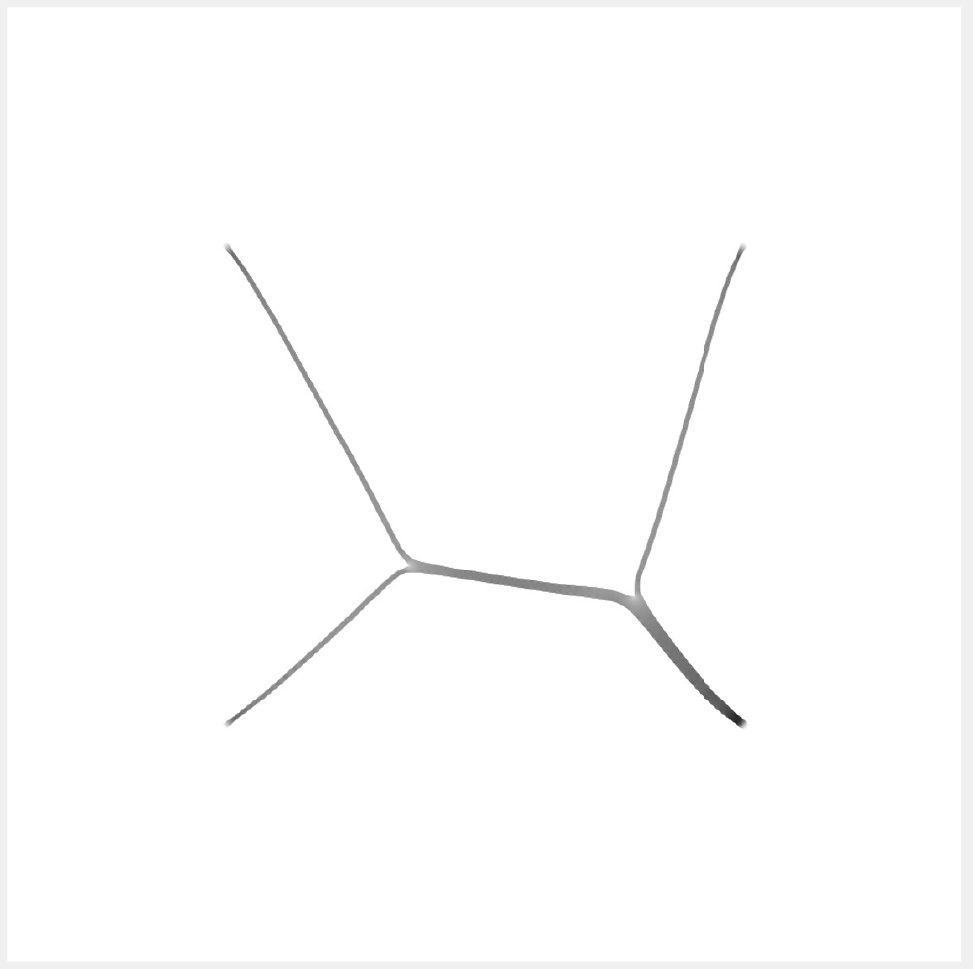} 
	\begin{picture}(0,0)(1,0)\put(.8,.16){\small$+$}\put(.08,.16){\small$-$}\put(.08,.91){\small$-$}\put(.78,.91){\small$-$}\end{picture}
	\includegraphics[width = \unitlength,trim=25 15 35 15,clip]{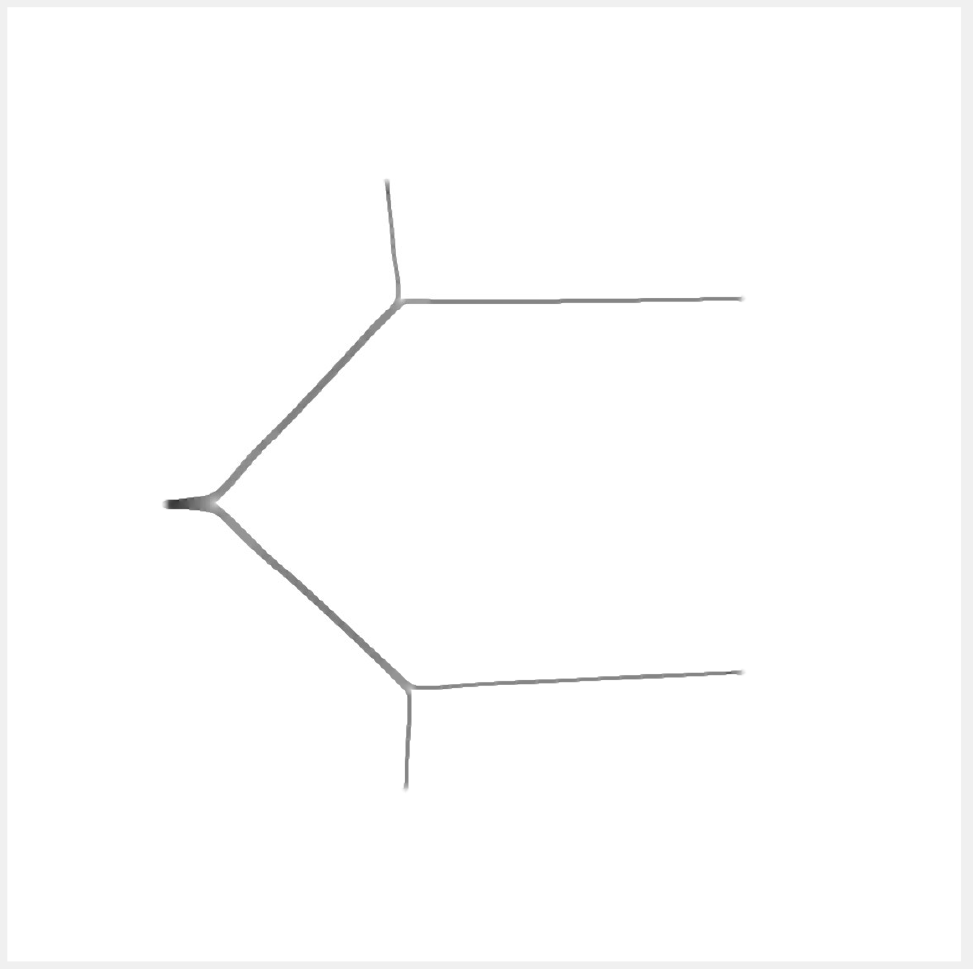} 
	\begin{picture}(0,0)(1,0)\put(-.06,.5){\small$+$}\put(.3,.09){\small$-$}\put(.85,.24){\small$-$}\put(.85,.82){\small$-$}\put(.3,.98){\small$-$}\end{picture}
	\includegraphics[width = \unitlength,trim=25 15 35 15,clip]{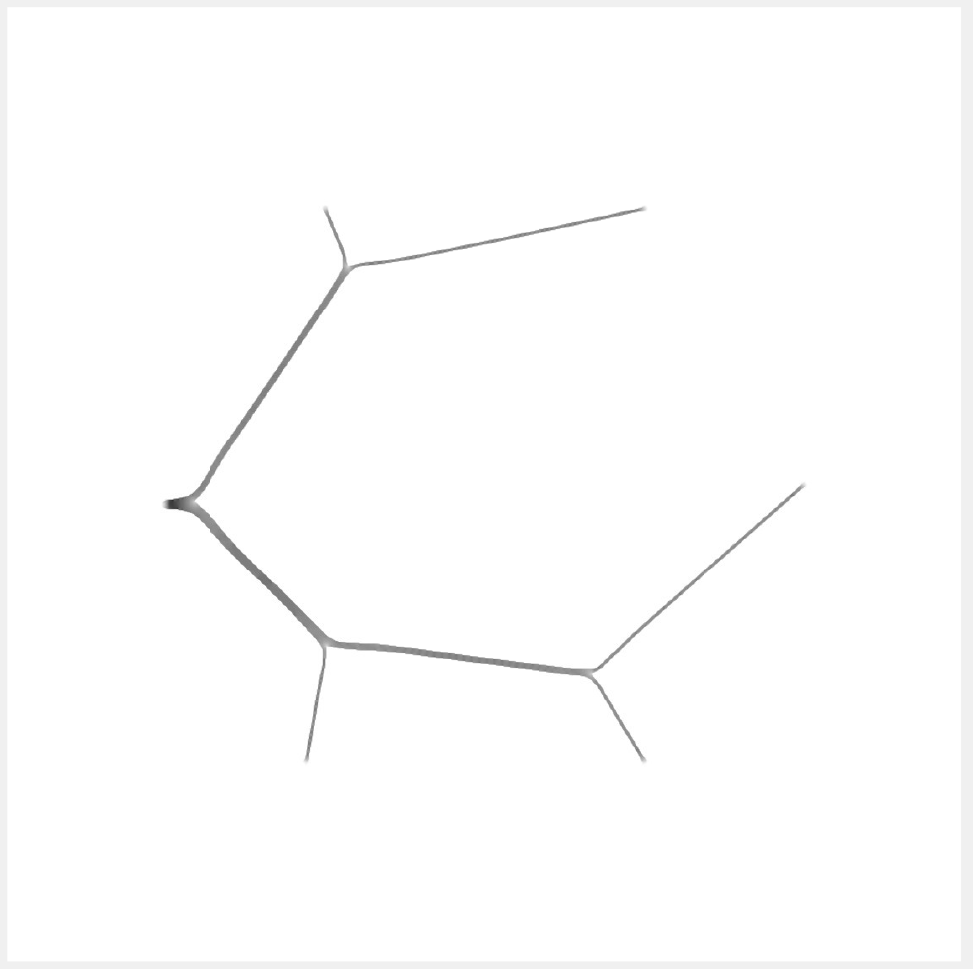}
	\begin{picture}(0,0)(1,0)\put(-.06,.5){\small$+$}\put(.2,.11){\small$-$}\put(.65,.11){\small$-$}\put(.21,.94){\small$-$}\put(.66,.93){\small$-$}
	\put(.93,.51){\small$-$}\end{picture} \\
	\raisebox{\depth}{\scalebox{1}[-1]{\includegraphics[width = \unitlength,trim=64 36 68 25,clip]{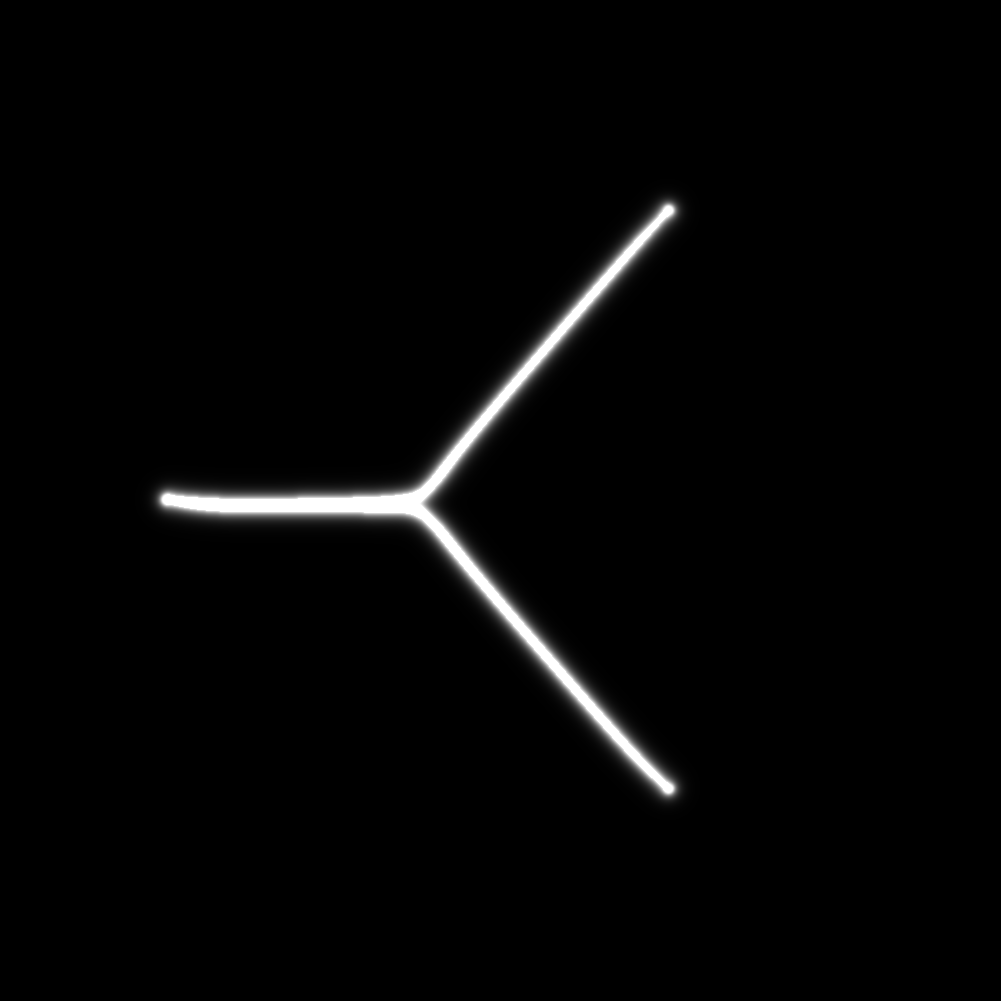}}}
	\raisebox{\depth}{\scalebox{1}[-1]{\includegraphics[width = \unitlength,trim=64 36 68 25,clip]{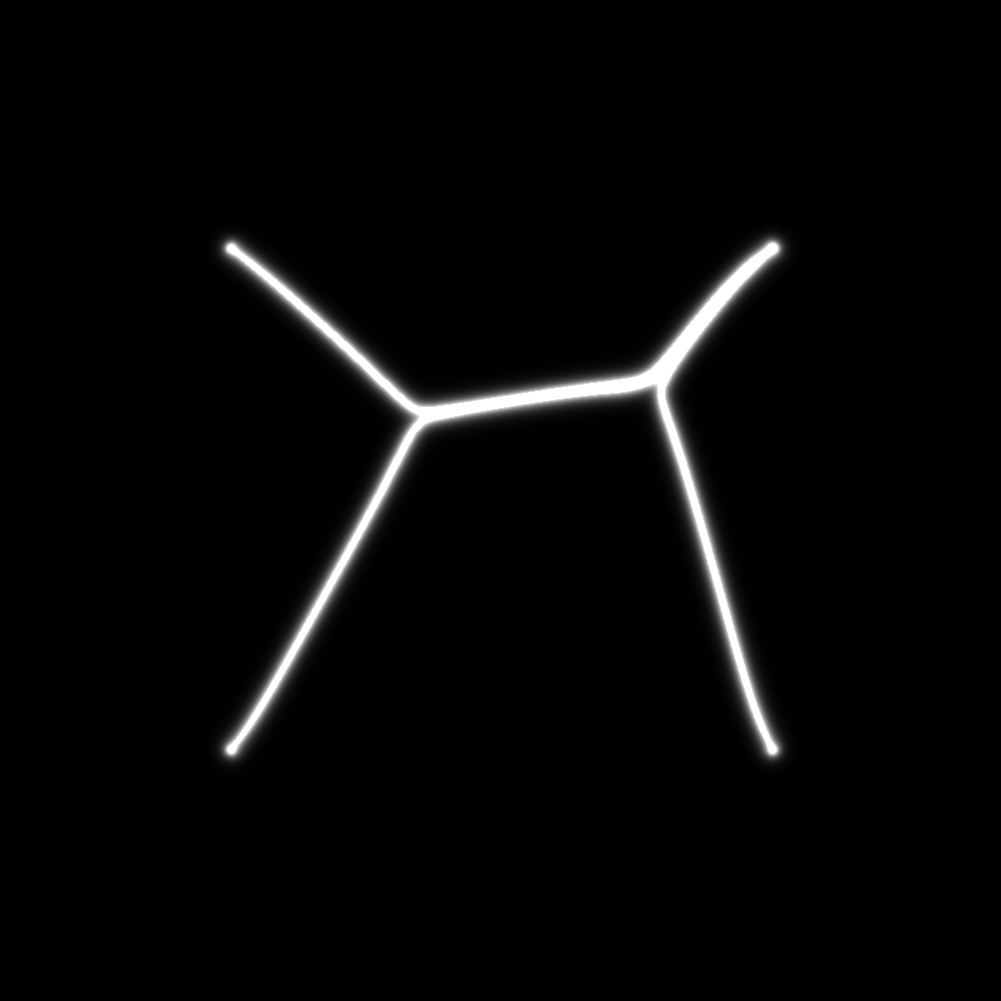}}}
	\raisebox{\depth}{\scalebox{1}[-1]{\includegraphics[width = \unitlength,trim=64 36 68 25,clip]{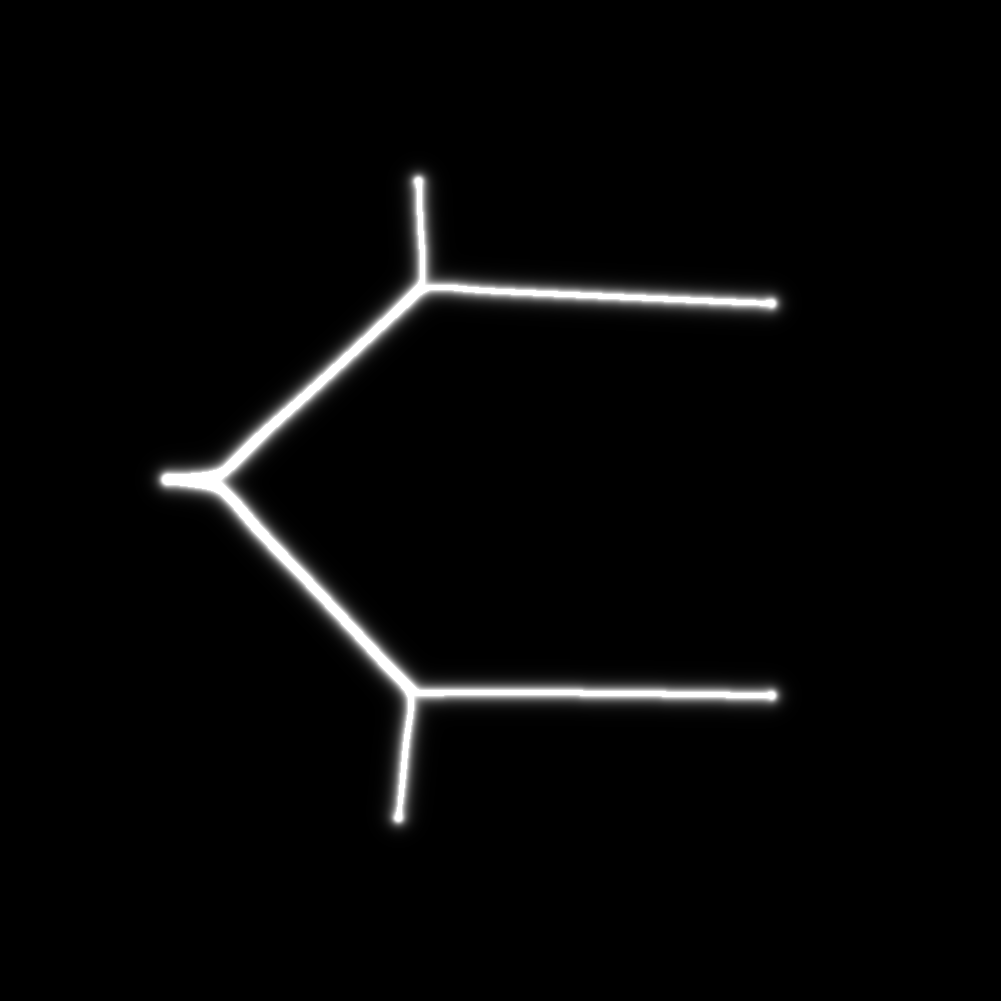}}} 
	\raisebox{\depth}{\scalebox{1}[-1]{\includegraphics[width = \unitlength,trim=64 36 68 25,clip]{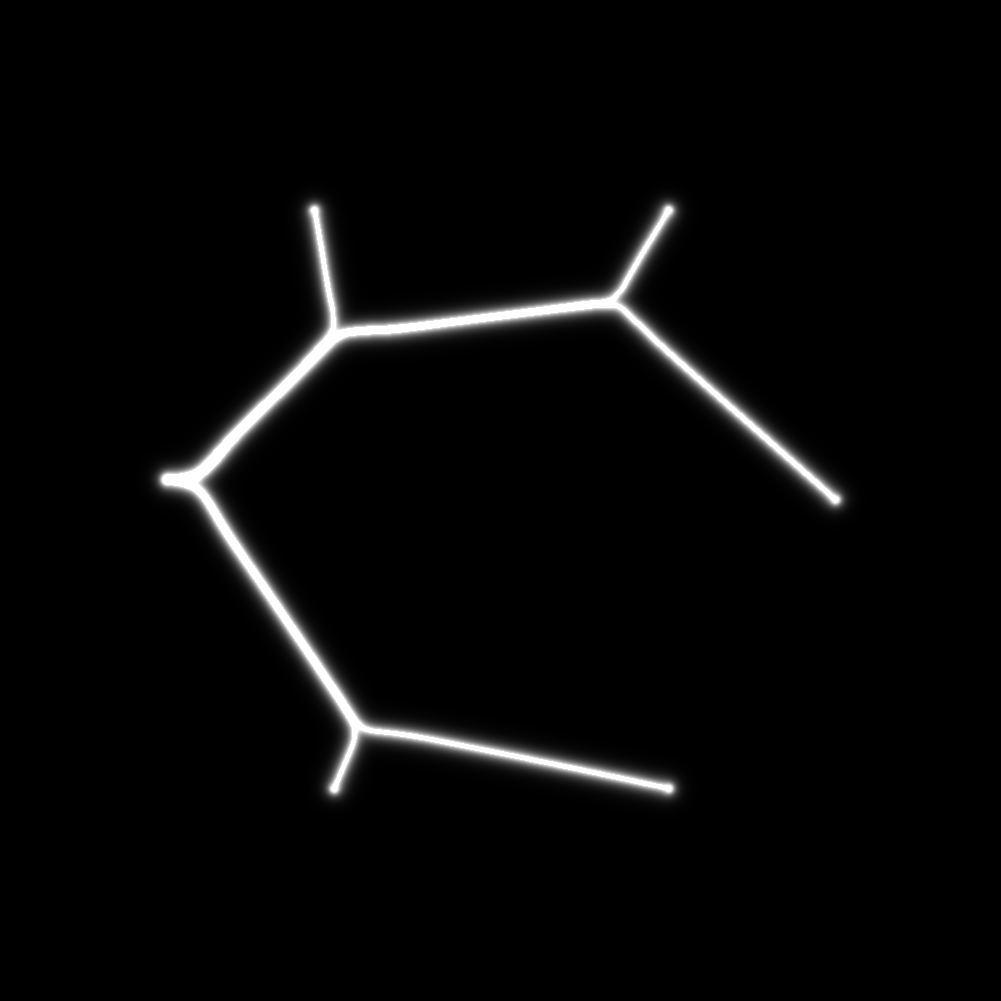}}} 
	\caption{Truly optimal network (top), computed mass flux $\sigma$ (middle), and phase field $\varphi$ (bottom) for same branched transportation problems as in \Cref{fig:SteinerProblem1}, only with $\alpha_1=1$, $\beta_1=1$, $\varepsilon=0.005$.}
	\label{fig:SteinerProblem2}
\end{figure}


\begin{figure}
	\centering 
	\includegraphics[width = 0.3\textwidth,trim=10 10 10 10,clip]{Result4To4_sigma.pdf} 
	\begin{picture}(0,0)(1,0)
	\put(-107.12,18.52){\small$+$}
	\put(-107.12,40.52){\small$+$}
	\put(-107.12,64.52){\small$+$}
	\put(-107.12,87.52){\small$+$}
	\put(-15.12,17.52){\small$-$}
	\put(-15.12,40.52){\small$-$}
	\put(-15.12,64.52){\small$-$}
	\put(-15.12,87.52){\small$-$}
	\end{picture}
	\includegraphics[width = 0.3\textwidth,trim=15 15 15 15,clip]{ResultCircle_sigma.pdf}
	\begin{picture}(0,0)(1,0)
	\put(-57.12,57.52){\small$+$}
	
	\put(-87.12,10.52){\small$-$}
	\put(-71.10,4.52){\small$-$}
	\put(-52.12,3.52){\small$-$}
	\put(-35.12,10.52){\small$-$}
	
	\put(-20.12,26.52){\small$-$}
	\put(-10.12,43.52){\small$-$}
	\put(-10.12,62.52){\small$-$}
	\put(-17.12,80.52){\small$-$}
	
	\put(-106.12,26.52){\small$-$}
	\put(-113.12,42.52){\small$-$}
	\put(-113.12,61.52){\small$-$}
	\put(-103.12,80.52){\small$-$}
	
	\put(-89.12,96.52){\small$-$}
	\put(-71.12,103.52){\small$-$}
	\put(-53.12,103.52){\small$-$}
	\put(-37.12,98.52){\small$-$}
	\end{picture}
	\\
	\includegraphics[width = 0.3\textwidth,trim=0 0 0 0,clip]{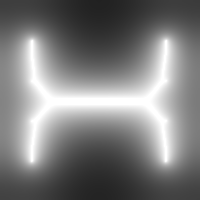} 
	\includegraphics[width = 0.3\textwidth,trim=0 0 0 0,clip,angle=90]{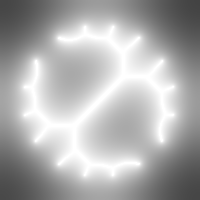}
	\caption{Computed mass flux $\sigma$ and phase field $\varphi$ for same parameters as in \Cref{fig:SteinerProblem1}.}
	\label{fig:singlePhaseField}
\end{figure}

\Cref{fig:multiplePhaseFields} shows simulation results for the same source and sink configuration as in \Cref{fig:singlePhaseField},
but this time with $N=3$ different linear segments in $\tau$ and corresponding phase fields.
It is clear that different phase fields become active on the different network branches according to the mass flux through each branch.
This can be interpreted as having streets of three different qualities:
the street $\varphi_3$ allows faster (cheaper) transport, but requires more maintenance than the others, while street $\varphi_1$ requires the least maintenance and only allows expensive transport.


\begin{figure}
	\centering\mbox{%
	\includegraphics[width = 0.21\textwidth,trim=20 30 20 20,clip]{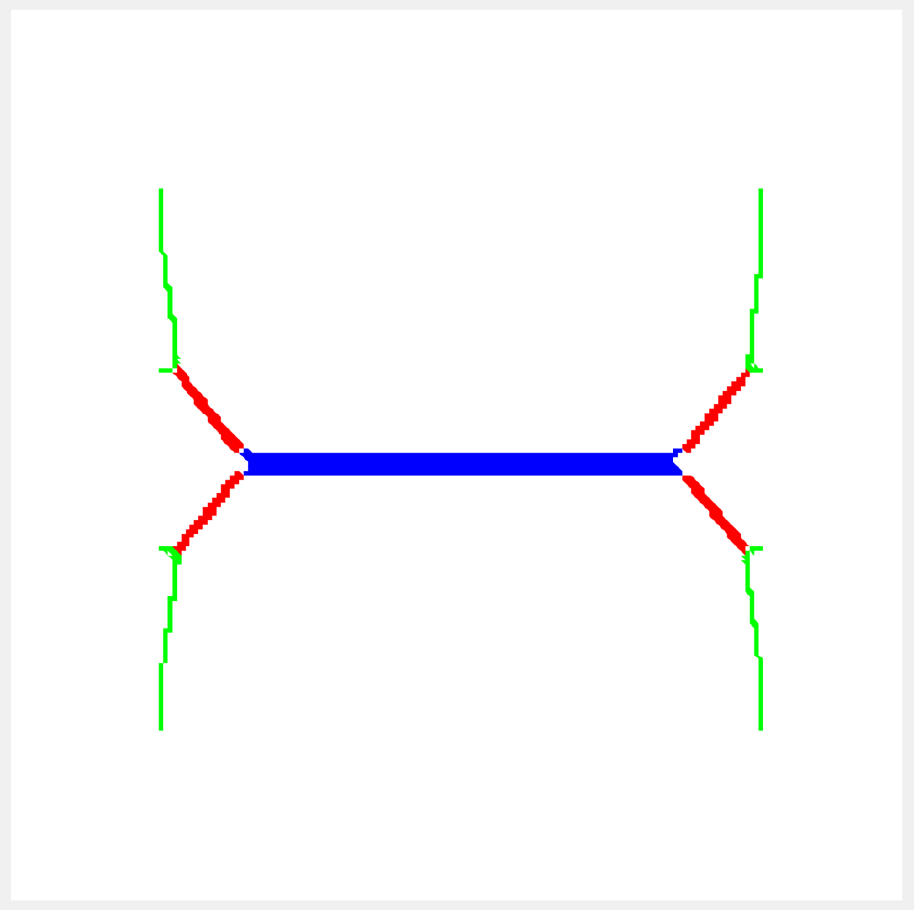} 
	\begin{picture}(0,0)(1,0)
	\put(-80.12,5.52){\small$+$}
	\put(-80.12,23.52){\small$+$}
	\put(-80.12,40.52){\small$+$}
	\put(-80.12,59.52){\small$+$}
	\put(-9.12,5.52){\small$-$}
	\put(-9.12,23.52){\small$-$}
	\put(-9.12,40.52){\small$-$}
	\put(-9.12,59.52){\small$-$}
	\end{picture}
	\includegraphics[width = 0.17\textwidth,clip]{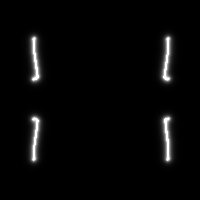} 
	\includegraphics[width = 0.17\textwidth,clip]{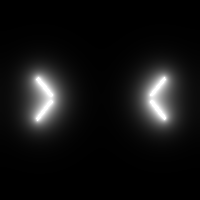} 
	\includegraphics[width = 0.17\textwidth,clip]{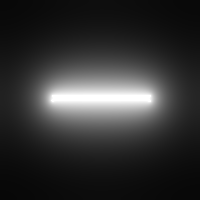} 
	\includegraphics[width = 0.23\textwidth,trim=8 5 30 18,clip]{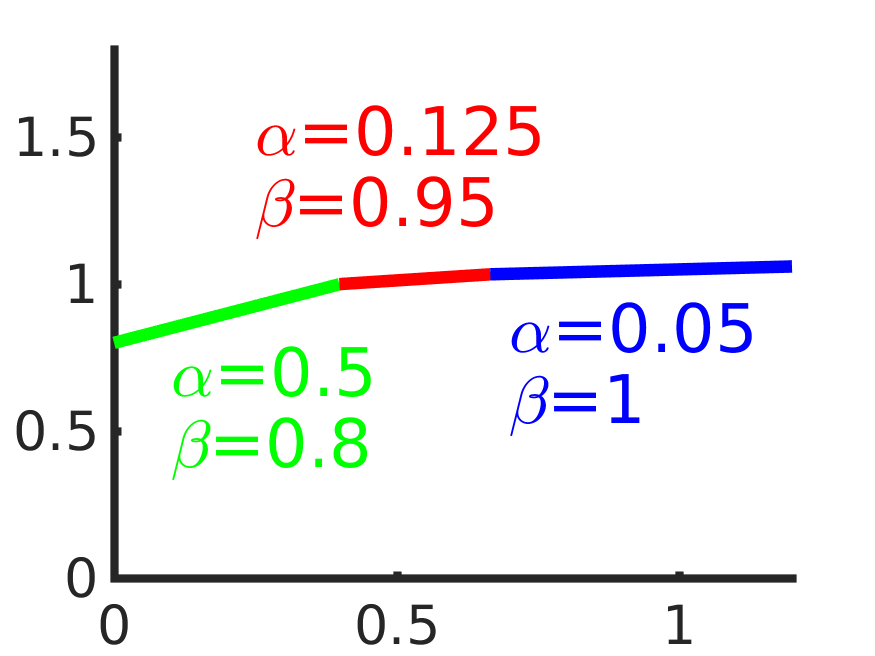}}
	\includegraphics[width = 0.21\textwidth,trim=20 30 20 20,clip]{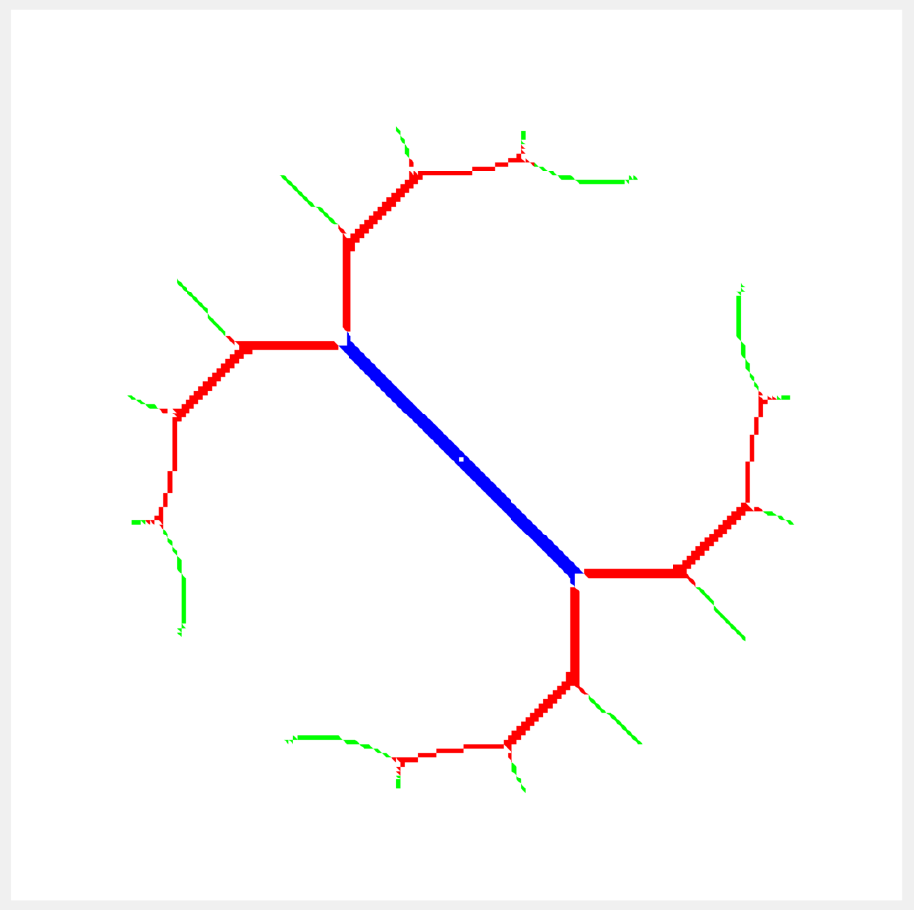} 
	\begin{picture}(0,0)(1,0)
	\put(-41.12,37.52){\small$+$}
	
	\put(-63.12,1.52){\small$-$}
	\put(-51.10,-3.52){\small$-$}
	\put(-38.12,-3.52){\small$-$}
	\put(-25.12,1.52){\small$-$}
	
	\put(-14.12,10.52){\small$-$}
	\put(-8.10,23.52){\small$-$}
	\put(-7.12,39.52){\small$-$}
	\put(-12.12,52.52){\small$-$}
	
	\put(-75.12,11.52){\small$-$}
	\put(-81.10,25.52){\small$-$}
	\put(-81.12,39.52){\small$-$}
	\put(-75.12,52.52){\small$-$}
	
	\put(-64.12,63.52){\small$-$}
	\put(-53.10,69.52){\small$-$}
	\put(-39.12,69.52){\small$-$}
	\put(-25.12,64.52){\small$-$}
	\end{picture}
	\includegraphics[width = 0.17\textwidth,clip,angle=90]{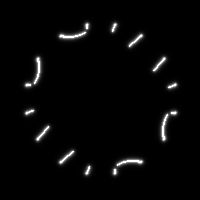} 
	\includegraphics[width = 0.17\textwidth,clip,angle=90]{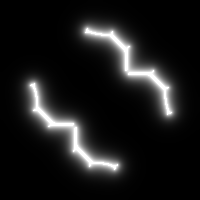} 
	\includegraphics[width = 0.17\textwidth,clip,angle=90]{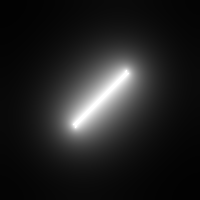} 
	\includegraphics[width = 0.23\textwidth,trim=8 5 30 18,clip]{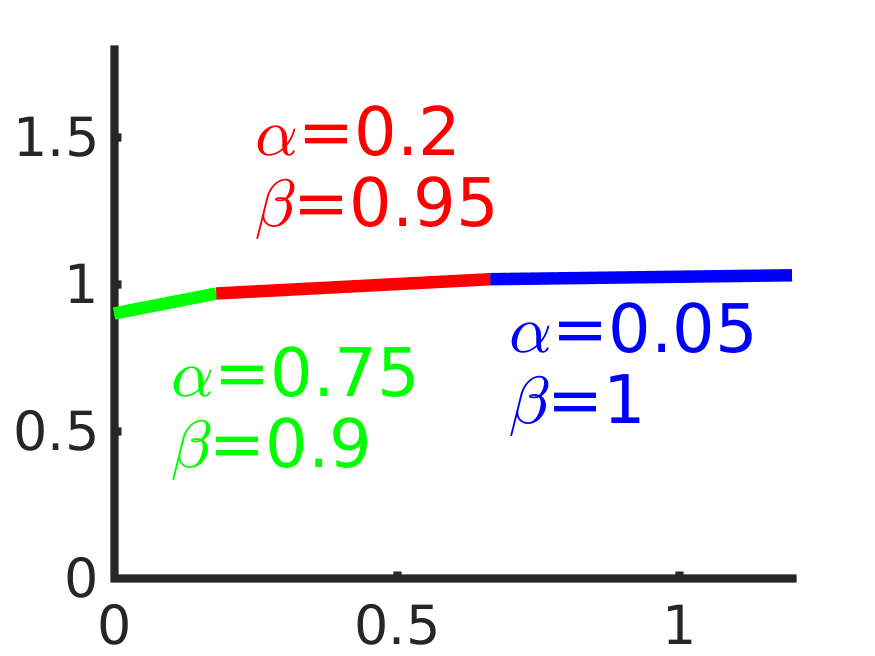} 
	\caption{Computed mass flux $\sigma$ and phase fields $\varphi_1,\varphi_2,\varphi_3$ for the same source and sink as in \Cref{fig:singlePhaseField} and for the cost function shown on the right, $\varepsilon=0.005$. The color in $\sigma$ indicates which phase field is active.}
	\label{fig:multiplePhaseFields}
\end{figure}


The case $\alpha_0<\infty$ finally can be interpreted as the situation in which mass can also be transported off-road, that is,
part of the transport may happen without a street network, thus having maintenance cost $\beta_0=0$, but at the price of large transport expenses $\alpha_0$ per unit mass.
Corresponding results for again the same source and sink configuration are shown in \Cref{fig:multiplePhaseFieldsDiffuse}.
In contrast to the case $\alpha_0=\infty$ it is now also possible to have sources and sinks that are not concentrated in a finite number of points.
A corresponding example is shown in \Cref{fig:diffuse}.



\begin{figure}
	\centering 
	\includegraphics[width = 0.25\textwidth,trim=20 30 20 20,clip]{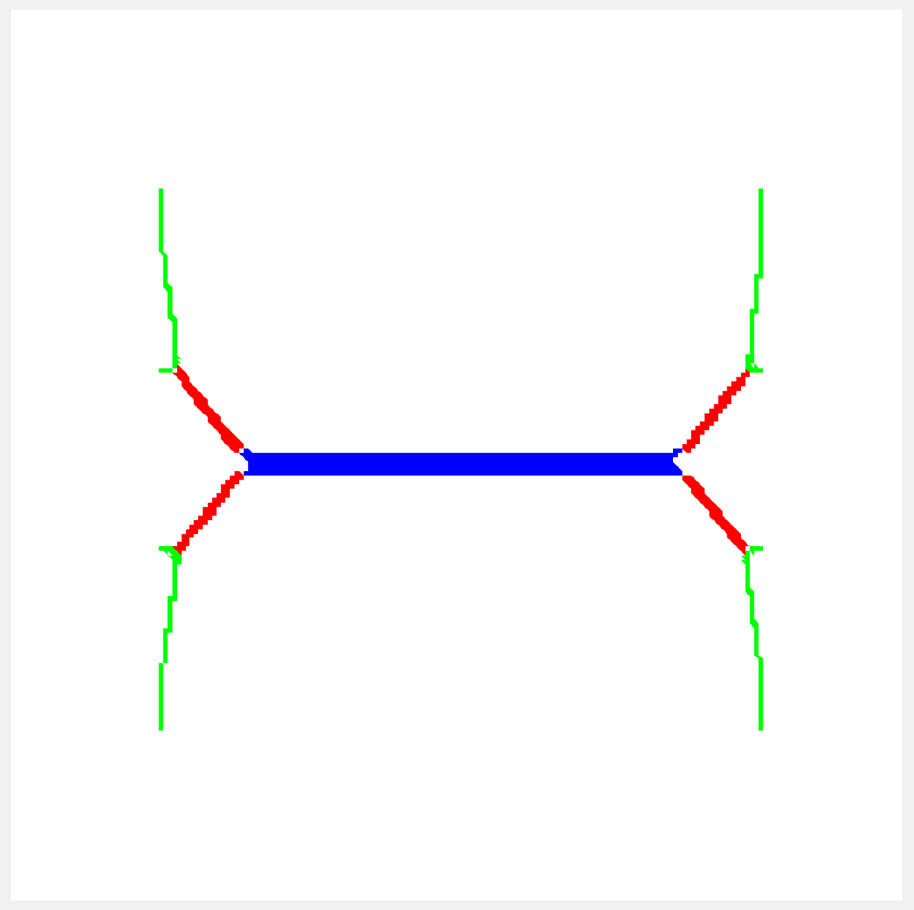} 
	\begin{picture}(0,0)(1,0)
	\put(-94.12,6.52){\small$+$}
	\put(-94.12,27.52){\small$+$}
	\put(-94.12,49.52){\small$+$}
	\put(-94.12,70.52){\small$+$}
	\put(-10.12,6.52){\small$-$}
	\put(-9.12,27.52){\small$-$}
	\put(-9.12,49.52){\small$-$}
	\put(-9.12,71.52){\small$-$}
	\end{picture}
	\includegraphics[width = 0.22\textwidth,clip]{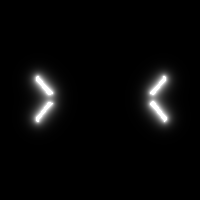} 
	\includegraphics[width = 0.22\textwidth,clip]{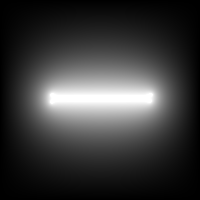} 
	\includegraphics[width = 0.27\textwidth,clip]{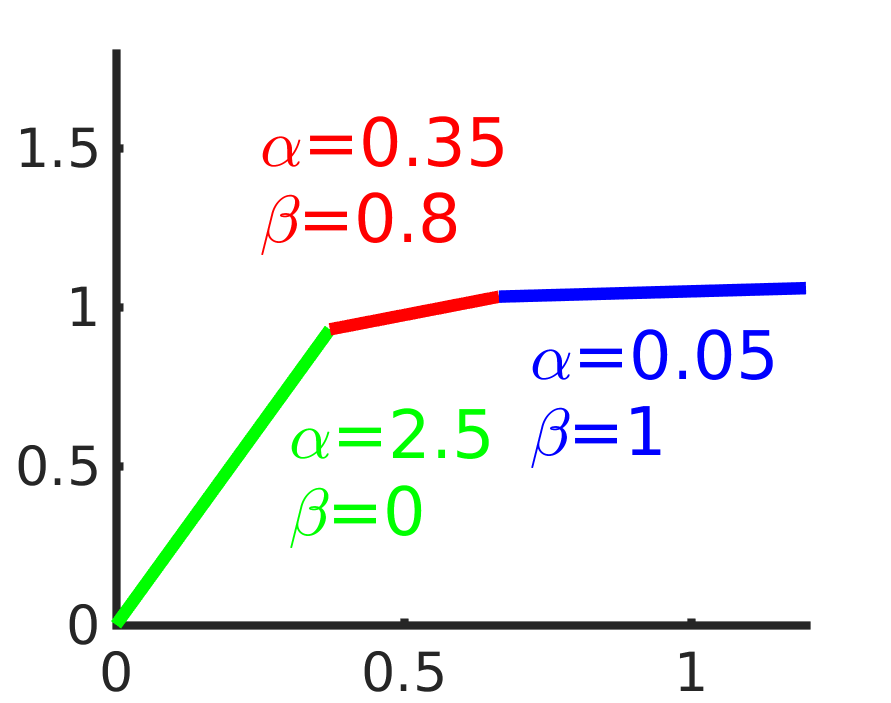}  \\[\baselineskip]
	\includegraphics[width = 0.25\textwidth,trim=20 30 20 20,clip]{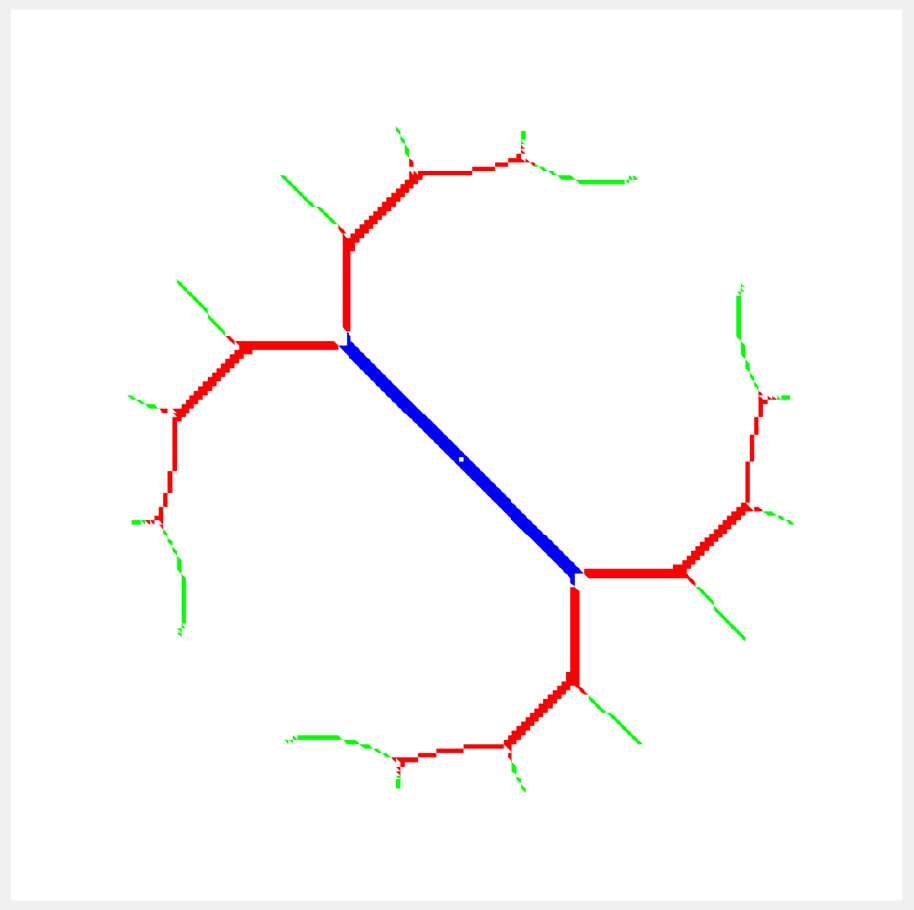} 
	\begin{picture}(0,0)(1,0)
	\put(-47.12,42.52){\small$+$}
	
	\put(-74.12,1.52){\small$-$}
	\put(-61.10,-3.52){\small$-$}
	\put(-44.12,-4.52){\small$-$}
	\put(-28.12,1.52){\small$-$}
	
	\put(-15.12,13.52){\small$-$}
	\put(-8.10,28.52){\small$-$}
	\put(-7.12,46.52){\small$-$}
	\put(-15.12,62.52){\small$-$}
	
	\put(-87.12,13.52){\small$-$}
	\put(-96.10,30.52){\small$-$}
	\put(-96.12,46.52){\small$-$}
	\put(-86.12,62.52){\small$-$}
	
	\put(-74.12,75.52){\small$-$}
	\put(-62.10,82.52){\small$-$}
	\put(-46.12,82.52){\small$-$}
	\put(-31.12,77.52){\small$-$}
	\end{picture}
	\includegraphics[width = 0.22\textwidth,clip,angle=90]{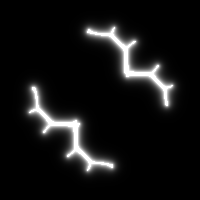} 
	\includegraphics[width = 0.22\textwidth,clip,angle=90]{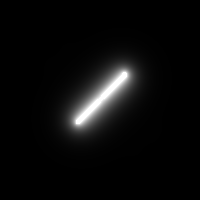} 
	\includegraphics[width = 0.27\textwidth,clip]{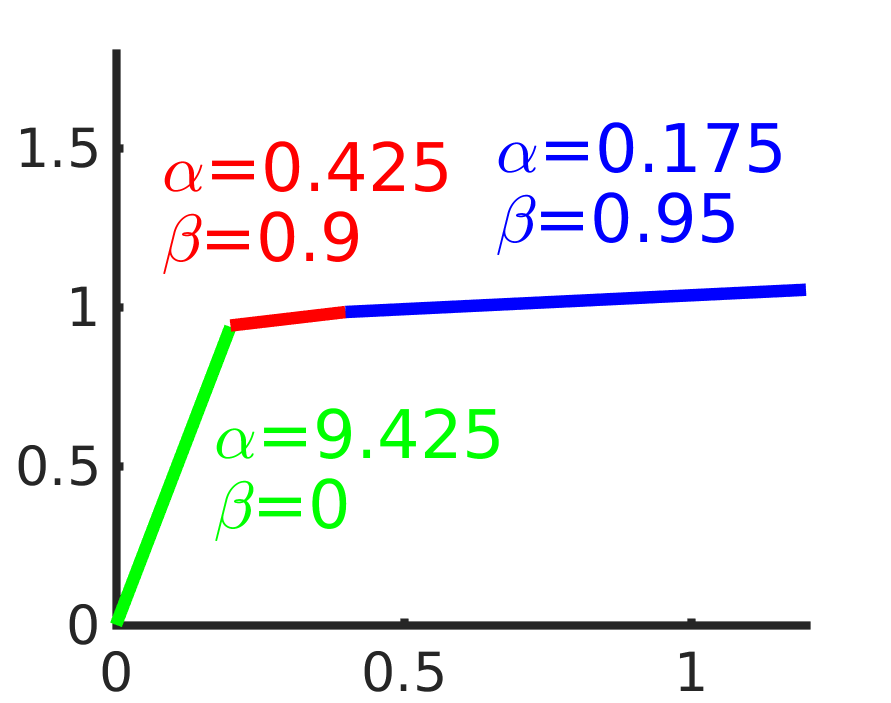} 
	\caption{Computed mass flux $\sigma$ and phase fields $\varphi_1,\varphi_2$ for the same source and sink as in \Cref{fig:singlePhaseField} and for the cost function shown on the right, $\varepsilon=0.005$.}
	\label{fig:multiplePhaseFieldsDiffuse}
\end{figure}

\begin{figure}
	\centering 
	\begin{tikzpicture}
	 \pgftext{\includegraphics[width = 0.28\textwidth,trim=10 10 10 10,clip]{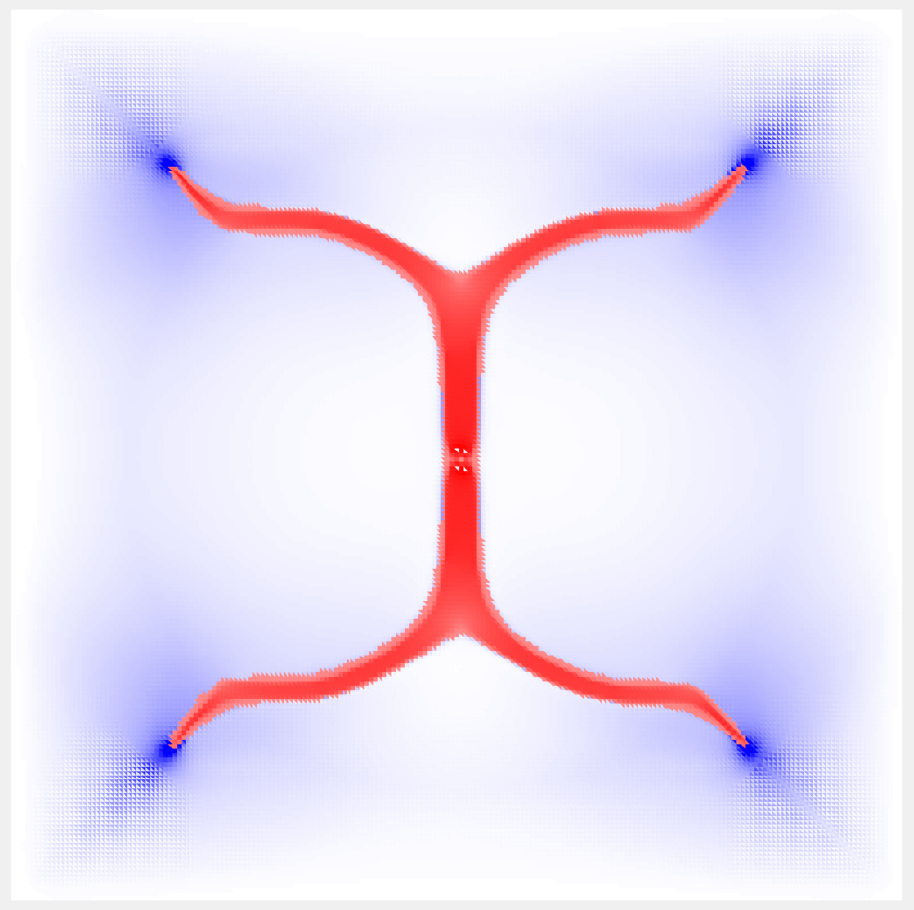}};
	 \draw[dashed,thick] (0.01,0) circle (40pt);
	\end{tikzpicture}
	\hspace{0.3cm}
	\includegraphics[width = 0.28\textwidth,clip]{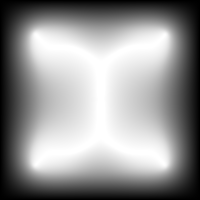} 
	\includegraphics[width = 0.35\textwidth,clip]{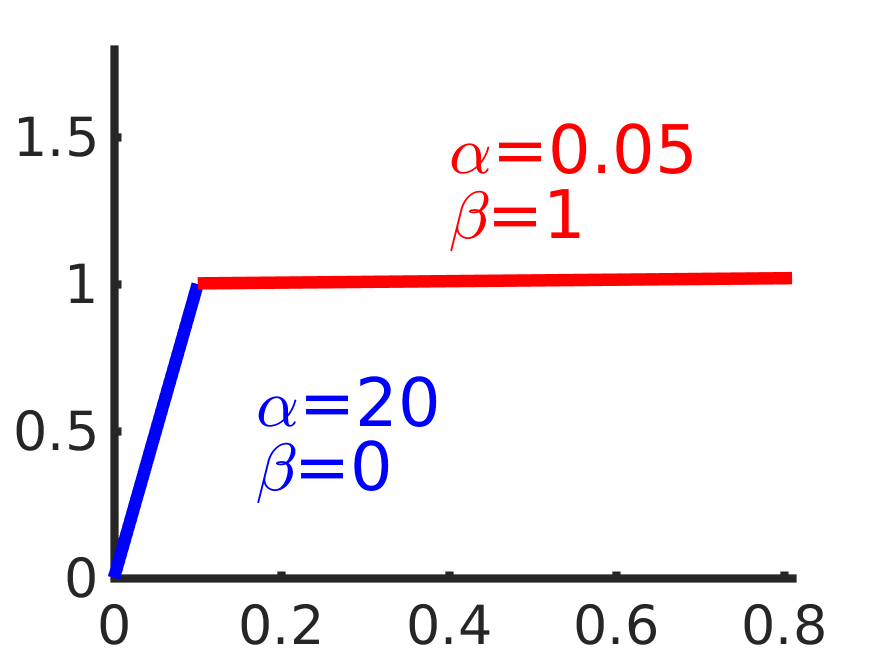} 
	\caption{Computed mass flux $\sigma$ and phase field $\varphi_1$ for a central point source and a spatially uniform sink outside a circle of radius and for the cost function shown on the right, $\varepsilon=0.005$.}
	\label{fig:diffuse}
\end{figure}


\paragraph{Acknowledgements}
The authors thank Antonin Chambolle and B\'enoit Merlet for various discussions and suggestions.
B.W.'s research was supported by the Alfried Krupp Prize for Young University Teachers awarded by the Alfried Krupp von Bohlen und Halbach-Stiftung.
The work was also supported by the Deutsche Forschungsgemeinschaft (DFG), Cells-in-Motion Cluster of Excellence (EXC1003 -- CiM), University of M\"unster, Germany.

\bibliographystyle{plain}
\bibliography{ChFeRoWi17}

\begin{thebibliography}{10}

\bibitem{AmFuPa00}
Luigi Ambrosio, Nicola Fusco, and Diego Pallara.
\newblock {\em Functions of bounded variation and free discontinuity problems}.
\newblock Oxford Mathematical Monographs. Oxford University Press, New York,
  2000.

\bibitem{Ambr_Tilli}
Luigi Ambrosio and Paolo Tilli.
\newblock {\em Topics on analysis in metric spaces}, volume~25 of {\em Oxford
  Lecture Series in Mathematics and its Applications}.
\newblock Oxford University Press, Oxford, 2004.

\bibitem{Amb_Tort2}
Luigi Ambrosio and V.~M. Tortorelli.
\newblock On the approximation of free discontinuity problems.
\newblock {\em Boll. Un. Mat. Ital. B (7)}, 6(1):105--123, 1992.

\bibitem{Amb_Tort1}
Luigi Ambrosio and Vincenzo~Maria Tortorelli.
\newblock Approximation of functionals depending on jumps by elliptic
  functionals via {$\Gamma$}-convergence.
\newblock {\em Comm. Pure Appl. Math.}, 43(8):999--1036, 1990.

\bibitem{MR3337998}
Matthieu Bonnivard, Antoine Lemenant, and Filippo Santambrogio.
\newblock Approximation of length minimization problems among compact connected
  sets.
\newblock {\em SIAM J. Math. Anal.}, 47(2):1489--1529, 2015.

\bibitem{Br98}
Andrea Braides.
\newblock {\em Approximation of free-discontinuity problems}, volume 1694 of
  {\em Lecture Notes in Mathematics}.
\newblock Springer-Verlag, Berlin, 1998.

\bibitem{Bra2002}
Andrea Braides.
\newblock {\em {$\Gamma$}-convergence for beginners}, volume~22 of {\em Oxford
  Lecture Series in Mathematics and its Applications}.
\newblock Oxford University Press, Oxford, 2002.

\bibitem{BrWi17}
Alessio Brancolini and Benedikt Wirth.
\newblock General transport problems with branched minimizers as functionals of
  1-currents with prescribed boundary.
\newblock {\em Calculus of Variations and Partial Differential Equations},
  57(3):82, Apr 2018.

\bibitem{MR2176114}
Giuseppe Buttazzo and Filippo Santambrogio.
\newblock A model for the optimal planning of an urban area.
\newblock {\em SIAM J. Math. Anal.}, 37(2):514--530, 2005.

\bibitem{MR2535060}
Giuseppe Buttazzo and Filippo Santambrogio.
\newblock A mass transportation model for the optimal planning of an urban
  region.
\newblock {\em SIAM Rev.}, 51(3):593--610, 2009.

\bibitem{ChaFerMer16}
Antonin Chambolle, Luca Ferrari, and Benoit Merlet.
\newblock A phase-field approximation of the {Steiner} problem in dimension
  two.
\newblock {\em Advances in Calculus of Variations}, 2017.

\bibitem{ChaFerMer17}
Antonin Chambolle, Luca Ferrari, and Benoit Merlet.
\newblock Variational approximation of size-mass energies for $k$-dimensional
  currents.
\newblock {\em ESAIM: Control, Optimisation and Calculus of Variations}, 2018.

\bibitem{CoRoMa17}
M.~Colombo, A.~De~Rosa, A.~Marchese, and S.~Stuvard.
\newblock On the lower semicontinuous envelope of functionals defined on
  polyhedral chains.
\newblock preprint on http://cvgmt.sns.it/paper/3347/, 2017.

\bibitem{ContiFocardiIurlano}
S.~Conti, M.~Focardi, and F.~Iurlano.
\newblock Phase field approximation of cohesive fracture models.
\newblock {\em Ann. Inst. H. Poincar\'e Anal. Non Lin\'eaire},
  33(4):1033--1067, 2016.

\bibitem{DalMaso}
Gianni Dal~Maso.
\newblock {\em An introduction to {$\Gamma$}-convergence}, volume~8 of {\em
  Progress in Nonlinear Differential Equations and their Applications}.
\newblock Birkh\"auser Boston, Inc., Boston, MA, 1993.

\bibitem{Fe69}
Herbert Federer.
\newblock {\em Geometric measure theory}.
\newblock Die Grundlehren der mathematischen Wissenschaften, Band 153.
  Springer-Verlag New York Inc., New York, 1969.

\bibitem{BLTJ:BLTJ4250}
E.~N. Gilbert.
\newblock Minimum cost communication networks.
\newblock {\em Bell System Technical Journal}, 46(9):2209--2227, 1967.

\bibitem{Gilb_Poll}
E.~N. Gilbert and H.~O. Pollak.
\newblock Steiner minimal trees.
\newblock {\em SIAM J. Appl. Math.}, 16:1--29, 1968.

\bibitem{Iur}
Flaviana Iurlano.
\newblock Fracture and plastic models as {$\Gamma$}-limits of damage models
  under different regimes.
\newblock {\em Adv. Calc. Var.}, 6(2):165--189, 2013.

\bibitem{Karp}
Richard~M. Karp.
\newblock Reducibility among combinatorial problems.
\newblock In {\em Complexity of computer computations ({P}roc. {S}ympos., {IBM}
  {T}homas {J}. {W}atson {R}es. {C}enter, {Y}orktown {H}eights, {N}.{Y}.,
  1972)}, pages 85--103. Plenum, New York, 1972.

\bibitem{LeloScVa14}
Jan Lellmann, Dirk~A. Lorenz, Carola Sch\"onlieb, and Tuomo Valkonen.
\newblock Imaging with {K}antorovich-{R}ubinstein discrepancy.
\newblock {\em SIAM J. Imaging Sci.}, 7(4):2833--2859, 2014.

\bibitem{Mod_Mort}
Luciano Modica and Stefano Mortola.
\newblock Un esempio di {$\Gamma ^{-}$}-convergenza.
\newblock {\em Boll. Un. Mat. Ital. B (5)}, 14(1):285--299, 1977.

\bibitem{MuSh89}
David Mumford and Jayant Shah.
\newblock Optimal approximation by piecewise smooth functions and associated
  variational problems.
\newblock {\em Communications on Pure and Applied Mathematics}, 42(5):577--685,
  1989.

\bibitem{OuSa11}
Edouard Oudet and Filippo Santambrogio.
\newblock A {M}odica-{M}ortola approximation for branched transport and
  applications.
\newblock {\em Arch. Ration. Mech. Anal.}, 201(1):115--142, 2011.

\bibitem{Paol_Step}
Emanuele Paolini and Eugene Stepanov.
\newblock Existence and regularity results for the {S}teiner problem.
\newblock {\em Calc. Var. Partial Differential Equations}, 46(3-4):837--860,
  2013.

\bibitem{Sa15}
Filippo Santambrogio.
\newblock {\em Optimal Transport for Applied Mathematicians}, volume~87 of {\em
  Progress in Nonlinear Differential Equations and Their Applications}.
\newblock Birkh\"auser Boston, 2015.

\bibitem{Sm93}
S.~K. Smirnov.
\newblock Decomposition of solenoidal vector charges into elementary solenoids,
  and the structure of normal one-dimensional flows.
\newblock {\em Algebra i Analiz}, 5(4):206--238, 1993.

\bibitem{Si07}
Miroslav \v{S}ilhav\'y.
\newblock Divergence measure vectorfields: their structure and the divergence
  theorem.
\newblock In {\em Mathematical modelling of bodies with complicated bulk and
  boundary behavior}, volume~20 of {\em Quad. Mat.}, pages 217--237. Dept.
  Math., Seconda Univ. Napoli, Caserta, 2007.

\bibitem{Wh99}
Brian White.
\newblock Rectifiability of flat chains.
\newblock {\em Annals of Mathematics}, 150(1):165--184, 1999.

\bibitem{Wh57}
Hassler Whitney.
\newblock {\em Geometric integration theory}.
\newblock Princeton University Press, Princeton, N. J., 1957.

\end{thebibliography}

\end{document}